\documentclass[11pt, a4paper]{article}
\usepackage[T1]{fontenc} 
\usepackage[latin1]{inputenc}
\usepackage[usenames,dvipsnames]{xcolor}
\usepackage{amsmath}
\usepackage{amsthm}
\usepackage{amsfonts}
\usepackage{amssymb}
\usepackage{array}
\usepackage{graphicx} 
\usepackage{color}
\usepackage[english]{babel}
\usepackage{mathrsfs}
\usepackage{graphicx}
\usepackage{dsfont}
\definecolor{orangebis}{rgb}{0.99,0.25,0.00}
\definecolor{greenbis}{rgb}{0.10,0.85,0.10}
\definecolor{bluebis}{rgb}{0.10,0.30,0.99}
\usepackage[final]{hyperref}   
\hypersetup{
    linktoc=page,
    linkcolor=red,          
    citecolor=blue,        
    filecolor=blue,      
    urlcolor=cyan,
   colorlinks=true      }     

\usepackage{multicol}

\usepackage{lipsum}
\usepackage{geometry}
\author{Christophe Garban\thanks{Universit\'e Lyon 1. 
Supported by the ANR grant Liouville 15-CE40-0013 and the ERC grant LiKo 676999
} 
 \and Hugo Vanneuville\footnotemark[1]}
\title{Exceptional times for percolation under exclusion dynamics}
\date{}

\theoremstyle{plain}
\newtheorem{thm}{Theorem}[section]
\newtheorem{prop}[thm]{Proposition}
\newtheorem{lem}[thm]{Lemma}
\newtheorem{conj}[thm]{Conjecture}
\newtheorem{cor}[thm]{Corollary}

\theoremstyle{definition}
\newtheorem{defi}[thm]{Definition}
\newtheorem{definition}[thm]{Definition}

\theoremstyle{remark}
\newtheorem{rem}[thm]{Remark}
\newtheorem{remark}[thm]{Remark}

\newcommand{\margin}[1]{\textcolor{magenta}{*}\marginpar[\textcolor{magenta} {  \raggedleft  \footnotesize  #1 }  ]{ \textcolor{magenta} { \raggedright  \footnotesize  #1 }  }}

\renewcommand{\margin}[1]{}

\newcommand{\N}{\mathbb{N}}
\newcommand{\R}{\mathbb{R}}

\newcommand{\Z}{\mathbb{Z}}
\newcommand{\Q}{\mathbb{Q}}

\newcommand{\half}{\mathbb{H}}
\newcommand{\Pro}{\mathbb{P}}
\newcommand{\E}{\mathbb{E}}
\newcommand{\T}{\mathbb{T}}

\renewcommand{\Q}{\mathbb{Q}}

\renewcommand{\Z}{\mathbb{Z}}
\renewcommand{\N}{\mathbb{N}}

\def\H{\mathbb{H}}

\def\T{\mathbb{T}}

\newcommand{\un}{\mathds{1}}
\newcommand{\petito}[1]{o\mathopen{}\left(#1\right)}
\newcommand{\grandO}[1]{O\mathopen{}\left(#1\right)}

\newcommand{\cond}{\, \Big| \,}
\renewcommand{\textbf}[1]{\begingroup\bfseries\mathversion{bold}#1\endgroup}
\let\qed=\QED

\def\calI{\mathcal{I}}

\def\calS{\mathcal{S}}

\def\SLE{\mathrm{SLE}}

\def\E{\mathbb{E}} 
\def\md{\mid}

\def \eps {\epsilon}

\def\Bb#1#2{{\def\md{\bigm| }#1\bigl[#2\bigr]}}

\def\Eb{\Bb\E}

\def\<#1{\langle #1\rangle}

\def\nn{\nonumber}
\def\bi{\begin{itemize}}  
\def\ei{\end{itemize}}
\def\bnum{\begin{enumerate}} 
\def\enum{\end{enumerate}}
\def\ni{\noindent}
\def\bf{\bfseries}

\numberwithin{equation}{section}
\begin{document}

\maketitle

\abstract{We analyse in this paper a conservative analogue of the celebrated model of \textbf{dynamical percolation} introduced by H{\"a}ggstr{\"o}m, Peres and Steif in \cite{olle1997dynamical}. 
It is simply defined as follows: start with an initial percolation configuration $\omega(t=0)$. Let this configuration evolve in time according to a simple exclusion process with symmetric kernel $K(x,y)$. 
We start with a general investigation (following \cite{olle1997dynamical}) of this dynamical process $t\mapsto \omega_K(t)$ which we call \textbf{$K$-exclusion dynamical percolation}. We then proceed with a detailed analysis of the planar case at the critical point (both for the triangular grid and the square lattice $\Z^2$) where we consider the power-law kernels $K^\alpha$
\[
K^{\alpha}(x,y) \propto \frac 1 {\|x-y\|_2^{2+\alpha}}\,.
\]
We prove that if $\alpha>0$ is chosen small enough,
there exist \textbf{exceptional times} $t$ for which an infinite cluster appears in $\omega_{K^{\alpha}}(t)$. (On the triangular grid, we prove that this holds for all $\alpha<\alpha_0= \frac {217}{816}$.) The existence of such exceptional times for standard i.i.d. dynamical percolation (where sites evolve according to independent Poisson point processes) goes back to the work by Schramm-Steif in \cite{SS10}. In order to handle such a $K$-exclusion dynamics, we push further the spectral analysis of \textbf{exclusion noise sensitivity} which has been initiated in \cite{BGS}. (The latter paper can be viewed as a conservative analogue of the seminal paper by Benjamini-Kalai-Schramm \cite{BKS} on i.i.d. noise sensitivity.) 
The case of a nearest-neighbour simple exclusion process, corresponding to the limiting case $\alpha=+\infty$, is left widely open.}

\tableofcontents

\section{Introduction}

\subsection{Dynamical percolation}\label{sectioniid}

We consider bond percolation on an infinite, countable, connected, locally finite graph $G=(V,E)$. 
We write $\Pro_p$ for the probability measure of (bond) \textbf{percolation of parameter~$p$} on $G$ i.e. the probability measure on $\Omega = \lbrace -1,1 \rbrace^{E}$ obtained by declaring each edge \textbf{open} with probability $p$ and \textbf{closed} with probability $1-p$, independently of the others ($1$ means open and $-1$ means closed). More formally, $\Pro_p$ is the product measure $\left( p\delta_1 + (1-p)\delta_{-1} \right)^{\otimes E}$ on $\Omega$ equipped with the product $\sigma$-algebra. An element $\omega \in \Omega$ is called a \textbf{percolation configuration}. Moreover, a connected component of the graph obtained by keeping only the open edges is called a \textbf{cluster}. It is a simple consequence of Kolmogorov's $0$-$1$ law that, for each $p$, $\Pro_p \left[ \exists \text{ an infinite cluster} \right] \in \lbrace 0,1 \rbrace$. Moreover, it is well known (see for instance~\cite{grimmett1999percolation} or~\cite{bollobas2006percolation}) that there exists a \textbf{critical point} $p_c = p_c(G) \in [0,1]$ such that:
\begin{eqnarray*}
\forall p \in [0,p_c), \, \Pro_p \left[ \exists \text{ an infinite cluster} \right] = 0 \, ,\\
\forall p \in (p_c,1], \, \Pro_p \left[ \exists \text{ an infinite cluster} \right] = 1 \, .
\end{eqnarray*}

The most studied model is bond percolation on the Euclidean lattice $\Z^d$, $d \geq 2$. 
For this model, it is known that $p_c = p_c(d) \in (0,1)$. In other words, there exists a phase transition. Moreover, it is a celebrated theorem by Kesten \cite{kesten1980critical} that $p_c(2)=1/2$ and it is conjectured that, for any $d \geq 2$, $\Pro_{p_c} \left[ \exists \text{ an infinite cluster} \right] = 0$. This last property has been proved for $d=2$ (\cite{harris1960lower})  and $d \geq 11$ (see \cite{hara1994mean,fitzner2015nearest}).
\smallskip

In \cite{olle1997dynamical}, H{\"a}ggstr{\"o}m, Peres and Steif define and study the model of \textbf{dynamical} (bond) \textbf{percolation} (this model was invented independently by Benjamini). Dynamical percolation of parameter $p \in [0,1]$ is defined very easily as follows: we sample a percolation configuration $\omega(0)$ according to some initial law and we then let evolve each edge independently of each other according to Poisson point processes: at rate one, the states of edges are resampled using $p\delta_1 + (1-p)\delta_{-1}$. 
We obtain this way a c\`adl\`ag Markov process $\left( \omega(t) \right)_{t \geq 0}$ on the space $\Omega$ (seen as the compact metric product space) with $\Pro_p$ as (unique) invariant probability measure. The main question is whether, if $\omega(0) \sim \Pro_p$\footnote{Where $X \sim P$ means that $P$ is the distribution of the random variable $X$.}, there exist \textbf{exceptional times} for which the percolation configuration is very atypical.  Exceptional times are defined as follows: if $\Pro_p \left[ \exists \text{ an infinite cluster} \right] = 0$, then an exceptional time is a time for which there is an infinite cluster. On the other hand, if $\Pro_p \left[ \exists \text{ an infinite cluster} \right] = 1$, then it is a time for which there is no infinite cluster.

From now on, we assume that $\omega(0) \sim \Pro_p$. Since $\Pro_p$ is an invariant measure, then (by Fubini) a.s. $\text{Leb}$-a.e. there is no exceptional time (where $\text{Leb}$ is the Lebesgue measure on $\R_+$). This does not imply that a.s. there does not exist any exceptional time. However, this is the case away from the critical point: the authors of~\cite{olle1997dynamical} have proved that, for any graph $G$, if $p \neq p_c$ then a.s. there is no exceptional time (see their Proposition~$1.1$).

The case $p=p_c$ is in general much more difficult. First, let us note that, for bond percolation on the Euclidean lattice $\Z^d$, this is for now interesting only for $d = 2$ and $d \geq 11$ since these are the only dimensions for which we know what happens at criticality. For $d \geq 11$, thanks to a result proved in~\cite{hara1994mean} for $d \geq 19$ (and extended very recently to $d \ge 11$ in~\cite{fitzner2015nearest}), the authors of \cite{olle1997dynamical} have proved that, even at criticality, a.s. there is no exceptional time (see their Theorem~$1.3$). However, for $d=2$, the following is proved in~\cite{GPS} (Theorem~$1.4$):
\[
\text{For dynamical bond percolation on } \Z^2, \text{ a.s. there are exceptional times if } p=p_c=1/2.
\]

Such a result had been proved earlier in \cite{SS10} for the model of \textbf{site percolation on the triangular lattice}. Let $\T$ denote the (planar) triangular lattice and let $\Pro_p$ denote the probability measure of site percolation on $\T$ (this is the analogous model where the sites - i.e. the vertices of $\T$ - are open or closed; in this context a cluster is a connected component of the graph obtained by keeping only the open sites). Kesten's work also implies that $p_c=1/2$ for this model. Of course, one can define dynamical site percolation on $\T$ in the same way as for dynamical bond percolation i.e. by associating exponential clocks to the sites of $\T$. 
Much more is known for site percolation on $\T$ than for bond percolation on $\Z^2$. Indeed conformal invariance (as the mesh goes to zero) has been proved by Smirnov in \cite{smirnov2001criticalp}, and the exact value of several critical exponents (see Subsection \ref{sectionarm}) has been derived in \cite{lawler2002onearm, smirnov2001critical} using the \textbf{Schramm Loewner Evolution (SLE)} processes introduced by Schramm.
Using the knowledge of these critical exponents,  the following is proved in~\cite{SS10} (Theorem~$1.3$):
\[
\text{For dynamical site percolation on } \T, \text{ a.s. there are exceptional times if } p=p_c=1/2.
\]

Finally, let us mention that in \cite{GPS} it is shown that, for critical site percolation on $\T$, the Hausdorff dimension of the set of exceptional times is a.s. $31/36$. For other results, see for instance~\cite{HPS15} where the authors show that typical exceptional times are intimately related to the so-called Incipient Infinite Cluster introduced by Kesten.

In both \cite{SS10} and \cite{GPS}, the main methods are related to the theory of \textbf{Fourier decomposition of Boolean functions}. In the present paper, we will also rely extensively on such tools, see Subsections~\ref{ss.spectral} and ~\ref{ss.SC}.

\subsection{Percolation 
 under exclusion dynamics}\label{sectionexclu}

We study in this paper the same question of existence of exceptional times but with a different underlying dynamical process: we let the configuration evolve according to a \textbf{symmetric exclusion process}. Percolation evolving according to an exclusion process has already been studied by Broman, the first author and Steif in~\cite{BGS} where the authors introduce and study the notion of \textbf{exclusion sensitivity}. (We will say more about this notion in Section~\ref{s.back}, see also~\cite{For15a, For15b}.) To define and study a symmetric exclusion process (which is a Feller Markov process), the most efficient way is to rely on its infinitesimal generator, see~\cite{liggett2005interacting}. However, we will sometimes need to use a more explicit construction of this dynamics, usually called a \textbf{graphical construction} after \cite{Harris}. We provide such a construction in Appendix \ref{a.graphical}. 
\begin{defi}\label{d.exclusion}
Consider a symmetric transition matrix $K$ on the set of edges $E$. Sample a percolation configuration $\omega_K(0)$ according to some initial law. To each pair of edges $\lbrace e,f \rbrace$, associate an exponential clock of parameter $K(e,f)=K(f,e)$  independent of the others and $\omega_K(0)$.
When the clock of a pair $\lbrace e,f \rbrace$ rings, exchange the states of the two edges. This way, we obtain a  càdlàg Markov process $\left( \omega_K(t) \right)_{t \geq 0}$ on the space $\Omega$ that is called a $K$-(symmetric) exclusion process. For every $p \in [0,1]$, $\Pro_p$ is an invariant measure for this process. In the following, we will always consider the case $\omega_K(0) \sim \Pro_p$  for some $p \in [0,1]$, and we will call the corresponding process a \textbf{$K$-exclusion dynamical percolation of parameter~$p$}. (Of course, a similar definition holds for a dynamics on site-configurations.)
\end{defi}

For more clarity, we call the dynamical percolation process of~\cite{olle1997dynamical} (defined in Subsection~\ref{sectioniid}) \textbf{i.i.d. dynamical percolation} (indeed, in this process, the states of the edges evolve independently of each other and according to the same law).


Following \cite{BGS}, our main motivation in this paper is guided by the following observation. For an exclusion process of parameter~$p$ which starts at equilibrium, one has $\omega_K(t) \sim \Pro_p$ for all $t\geq 0$. As such, one may ask the same natural  questions as for the i.i.d. dynamical process. In particular, we define exceptional times exactly in the same way. We shall consider in this article the following families of symmetric transition kernels $K$.

\begin{definition}[Symmetric transition kernels $K$ considered in this work]\label{d.kernels} $ $ 
\bnum
\item First, we shall analyse what happens away from the critical point $p_c=p_c(G)$ for any graph $G$ and at $p_c$ for percolation on $\Z^d$ with $d \geq 11$ (see Propositions \ref{outsidepc} and \ref{dgeq11}). The proofs in these cases are very close to the i.i.d. setting and work for \textbf{any} symmetric kernel $K$. 

\item Then, we focus on what happens at the critical point in the planar setting (critical bond percolation on $\Z^2$ or critical site percolation on $\T$). The most natural and most studied conservative dynamics in this case is certainly the \textbf{nearest-neighbour simple exclusion process} which on the triangular lattice $\T$ corresponds to the following kernel:
\[
K(v,w):= \frac 1 6 \un_{v\sim w} \; .
\]
(For bond percolation on $\Z^2$, one may consider several natural versions of nearest-neigbour exclusion process acting on edges.) 
As we shall explain later, we are far from being able to prove the existence of exceptional times (even on the triangular lattice $\T$) for this classical dynamics. This is why we consider the following kernels which in some sense interpolate between the i.i.d. case and the nearest-neighbour dynamics.

\item The main class of dynamics that we will analyse are the following \textbf{power-law kernels}. For any $\alpha\in (0,+\infty)$, let 
\[
\begin{cases}
K^\alpha(v,v) & := 0  \\
K^\alpha(v,w) & := c_\alpha\, \frac 1 {\|v-w\|_2^{2+\alpha} } \text{  if }v \neq w \, ,
\end{cases}
\]
where $c_\alpha$ is a normalization factor so that $K^\alpha$ is a transition kernel (i.e. $\sum_w K^\alpha(v,w)=1$). For bond percolation on $\Z^2$, one measures the Euclidean distance between two edges as the distance between their mid-points. 
The shape of the decay (in $r^{-(2+\alpha)}$) is chosen here in such a way that particles move at large scales according to  \textbf{$\alpha$-stable} processes. Note that by letting $\alpha \searrow 0$, one recovers in some sense an i.i.d. dynamics, while $\alpha \to +\infty$ converges to the nearest-neighbour simple exclusion process. 

\item Finally, the last family of kernels that we shall investigate are the following ones which are designed to be \textbf{super-heavy-tailed} (or in other words, they induce a very long-range exclusion process). The reason to consider these long-range dynamics is that the spectral analysis will be much simpler in this case than for the power-law kernels $K^\alpha$. Consider for any $a\in (0,+\infty)$,
\[
\begin{cases}
K_{\log}^a(v,v) & := 0  \\
K_{\log}^a(v,w) & := c_a \frac 1 {\|v-w\|_2^{2} \, \log(\|v-w\|_2+1)^{1+a} } \text{  if }v \neq w \, .
\end{cases}
\]
\enum
\end{definition}

\ni
We now list the main results proved in this paper. 

\subsection{Main results}\label{ss.results}

Our first two results are the direct analogues for exclusion dynamics of the main results on i.i.d. dynamics proved in the seminal paper on dynamical percolation \cite{olle1997dynamical}. The proofs follow very closely the ideas from \cite{olle1997dynamical} and do not require any assumption on the symmetric kernel $K$. 


\begin{prop}\label{outsidepc}
For any graph $G$ and any symmetric transition matrix $K$ on the edges of $G$, if $p \neq p_c$ then a.s. there is no exceptional time for the $K$-exclusion dynamical percolation of parameter~$p$. (This result is also true for site percolation and the proof is the same.)
\end{prop}

\begin{prop}\label{dgeq11}
Let $d \geq 11$.\footnote{This estimate $d\geq 11$ follows from the recent strengthening \cite{fitzner2015nearest} of \cite{hara1994mean} which would have given $d\geq 19$ instead.} Then, for any symmetric transition matrix $K$ on the edges of the Euclidean lattice $\Z^d$, a.s. there is no exceptional time for the $K$-exclusion dynamical percolation of parameter~$p$ even if $p = p_c = p_c(d)$.
\end{prop}

Let us now state the main theorem of this paper (which answers a question that motivated~\cite{BGS}).


\begin{thm}\label{tropbien}
Let $K^\alpha$ be the transition matrix from Definition \ref{d.kernels} on the edges of $\Z^2$ or on the sites of $\T$. If $\alpha>0$ is sufficiently small, then a.s. there exist exceptional times for the $K^\alpha$-exclusion dynamical percolation of parameter $p=p_c=1/2$.

Moreover, in the case of dynamics on the sites of $\T$, one has the following explicit lower-bound (as a function of $\alpha$) on the Hausdorff dimension of exceptional times. Let 
\begin{align*}
d(\alpha) := 1 - \frac 5 {36}  \left(1 - \frac {68} {21} \alpha \right)^{-1} \,.
\end{align*}
The Hausdorff dimension of the set of exceptional times of a $K^\alpha$-exclusion dynamical percolation of parameter $p_c= 1/2$ is an a.s. constant that lies in $[d(\alpha),31/36]$. See Figure \ref{f.main} for a plot of this estimate. 
In particular, we see that we obtain the existence of exceptional times for any $\alpha < \alpha_0= 217/816$.\footnote{The proof we will need for $\Z^2$ implies that one can go in fact slightly above this threshold $\alpha_0$. See the upper-bound given by~\eqref{e.alphamax}.} Note also that our lower-bound $d(\alpha)$ converges to $31/36$ as $\alpha \searrow 0$, which is known to be the Hausdorff dimension of exceptional times for the \textbf{i.i.d. dynamical percolation} (\cite{GPS}).
\end{thm}

\begin{figure}[!htp]
\begin{center}
\includegraphics[width=0.7\textwidth]{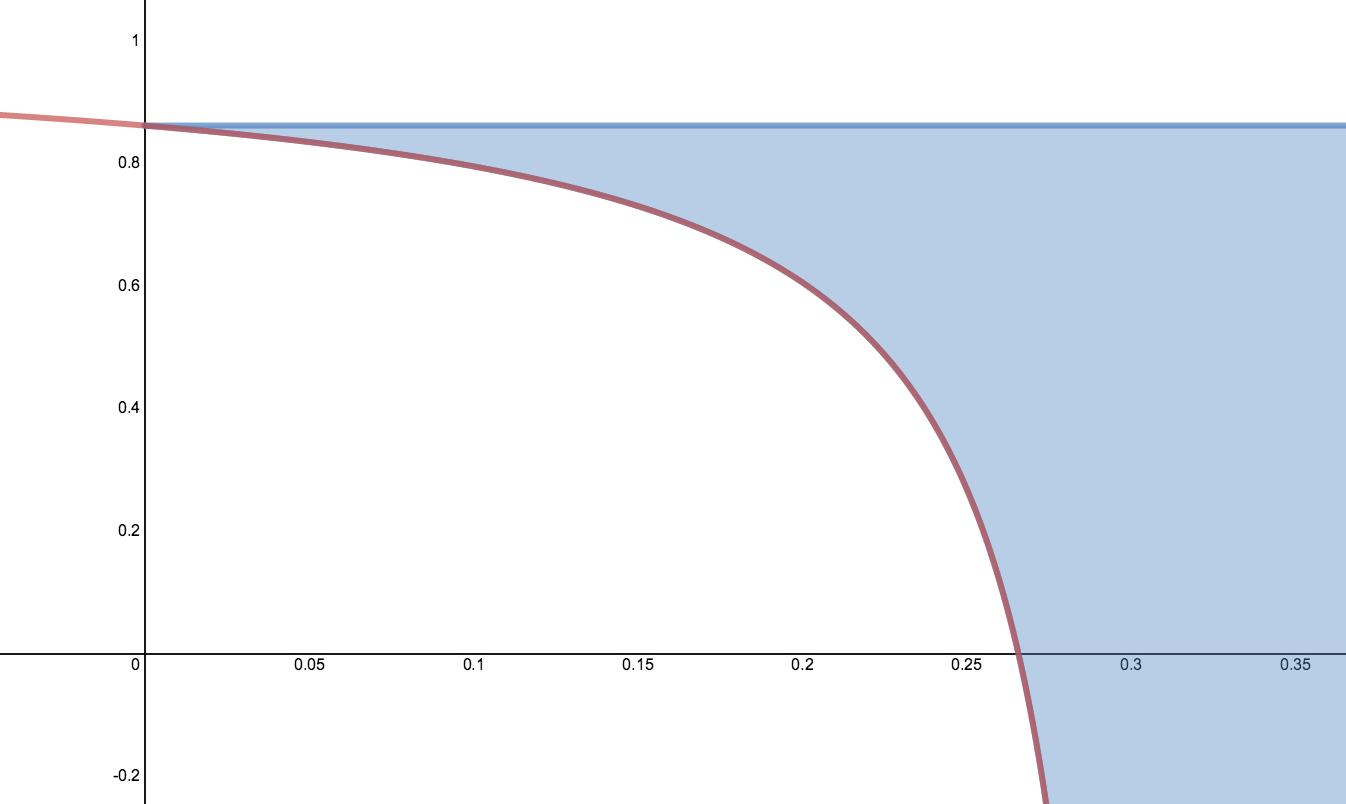}
\end{center}
\caption{The red curve represents the lower-bound $\alpha \mapsto d(\alpha)$ we obtain in Theorem \ref{tropbien} on the Hausdorff dimension of exceptional times for the power law kernels $K^\alpha(x,y) \propto \|x-y\|^{-(2+\alpha)}$. 
 As we see from this plot, we obtain in particular the existence of exceptional times for all $\alpha < \alpha_0 = 217/816 \approx 0.266$.
 The blue curve on the top is the (much easier to obtain) upper-bound equal to $31/36$.
 We conjecture that the dimension is a.s. equal to this value of $31/36$ for any $\alpha \in (0,+\infty]$ where $``\alpha=+\infty"$ would be the nearest-neighbour case.}\label{f.main}
\end{figure}

Theorem~\ref{tropbien} will be proved in Section~\ref{sectionexcepcriticality} (thanks to a result on a clustering effect for spectral sets of percolation whose proof will be postponed to Section~\ref{sectionspectralnotlocalized}). The proof will deeply use the fact that the dynamics we consider are not very localized.
With this in mind, it is not surprising that the proof is simpler (and gives a more precise estimate about the Hausdorff dimension) if we let $K=K_{\log}^{a}$ be the \textbf{long-range} symmetric transition matrix introduced in Definition~\ref{d.kernels} (on the edges of~$\Z^2$ or on the sites of~$\T$).
\begin{prop}\label{tropbienlog}
Take any $a>0$. 
 Then a.s. there exist exceptional times for the $K_{\log}^a$-exclusion dynamical percolation of parameter $p=p_c=1/2$. Moreover, in the case of dynamics on the sites of $\T$, a.s. the Hausdorff dimension of this set of times equals $31/36$.
\end{prop}
Proposition~\ref{tropbienlog} will be proved in Section~\ref{sectionexcepcriticality}. As mentioned above, our proof does not work for very localized dynamics (and in particular not for the nearest-neighbour process). Still we conjecture:

\begin{conj}\label{tropbienconj}
Take any symmetric transition matrix $K$ on the edges of $\Z^2$ or on the sites of $\T$. Assume that $K$ equals $0$ on the diagonal. Then a.s. there are exceptional times for the $K$-exclusion dynamical percolation of parameter $p=p_c=1/2$. Moreover, a.s. the Hausdorff dimension of this set of times equals $31/36$.
\end{conj}

\begin{rem}\label{ongoingwork}
As suggested by the above conjecture, in some sense any symmetric exclusion dynamical percolation should ``behave like the i.i.d. dynamical percolation" at the critical point. While the proofs of Theorem~\ref{tropbien} and Proposition~\ref{tropbienlog} heavily rely on the ``not localized" properties of the chosen dynamics, we have other clues to support this assertion.
This is actually the purpose of an ongoing work where we plan to prove that in some sense for any symmetric and translation invariant kernel $K$ on the sites of $\T$ (that equals $0$ on the diagonal), the $K$-exclusion dynamical percolation of parameter $p_c=1/2$ has a limit in the continuum that (if we change the time by only a constant factor) is the same as the limit of the i.i.d. process (this last limit has been proved to exist and studied in \cite{GPS2a} and \cite{GPS2b}). That would in particular imply that, for any such $K$, ``Conjecture~\ref{tropbienconj} is true in the continuum". 
\end{rem}


To conclude this section on main results, let us stress that another important contribution of this paper is a strengthening of some of the spectral estimates on the Fourier spectrum of critical percolation obtained in \cite{GPS}. It would be too early at this stage to state these results, but Theorem \ref{soimportant} will be an example of this.

\paragraph{Organization of the paper.} The paper is organized as follows: In Section~\ref{s.back}, we provide an outline of the proof of our main result Theorem~\ref{tropbien}. In particular, we discuss clustering effects for the spectral sets of percolation. Next, Section~\ref{sectionabsence} is devoted to the proof of Propositions~\ref{outsidepc} and~\ref{dgeq11}. In this section we do not use any spectral analysis tool and we follow the seminal paper~\cite{olle1997dynamical}. Then, in Sections~\ref{sectionexcepcriticality} and~\ref{sectionspectralnotlocalized}, we focus on the planar case at the critical point. Our main goal is to prove Theorem~\ref{tropbien}. As explained at the end of Subsection~\ref{ss.outline}, there are two steps in the proof of this theorem: \textit{(a)} showing a result about a clustering effect for spectral sets of percolation. The proof of this last result - in the spirit of~\cite{GPS} - is written in Section~\ref{sectionspectralnotlocalized}. \textit{(b)} Showing that this clustering effect implies a singularity property for the spectral sets of percolation when we let them evolve under our exclusion dynamics. This last step is the subject of Section~\ref{sectionexcepcriticality}.


\paragraph{Acknowledgments:} $ $ \\
We wish to thank C\'edric Bernardin for useful discussions on exclusion processes and duality formulas as well as Jeff Steif for pointing to us very useful references. Finally we wish to thank both anonymous referees for their very helpful and detailed comments.

\section{Background and outline of proof}\label{s.back}

\subsection{Arm events and some other notations}\label{sectionarm}





\ni
In this subsection, we list some classical notations/inequalities on (static) \textbf{critical} percolation. We refer for example to \cite{GPS, werner2007lectures} for more background.
We will focus from now on (except in Section \ref{sectionabsence}) on two models: bond percolation on $\Z^2$ and site percolation on $\T$. In both cases, we think about percolation configurations $\omega \in \Omega$ as colourings of the plane.\footnote{Consider a site percolation configuration on $\T$ and a hexagon $H$ of the hexagonal lattice dual to $\T$. We colour $H$ white (respectively black) if the corresponding site of $\T$ is open (respectively closed) in $\omega$. See Subsection~$2.1$ of~\cite{GPS} for the colouring of the plane induced by a percolation configuration on the edges of $\Z^2$. In both cases, an open (respectively a dual) path is a continuous path included in the white (respectively black) region of the plane.} Also, we let $\{ 0 \leftrightarrow R \}$ denote the event that there is an open path from $0$ to $\partial [-R,R]^2$.

The \textbf{tile} of a site/edge is the set of all points of the plane whose colour is determined by this site/edge.
Let $R \geq 0$. In the context of site percolation of $\T$, we let $\mathcal{I}_R$ denote the set of all sites whose tile intersects $[-R,R]^2$. In the context of bond percolation on $\Z^2$, we let $\mathcal{I}_R$ denote the set of the midpoints of all edges whose tile intersects $[-R,R]^2$ (we choose the midpoints only to obtain a discrete set). We then define $\Omega_R = \lbrace -1,1 \rbrace^{\mathcal{I}_R}$. In other words, $\Omega_R$ is the set of percolation configurations restricted to the window $[-R,R]^2$. We also write $\mathcal{I} = \cup_{R \geq 0} \mathcal{I}_R$.
We say that two disjoint subsets $A$ and $B$ of the plane are \textbf{percolation disjoint} if there is no tile that intersects both $A$ and $B$.

\paragraph{Arm events.} An annulus of the form $\left( x + [-R,R]^2 \right) \setminus \left( x + (-r,r)^2 \right)$ (where $0 \leq r \leq R$ and $x \in \R^2$) is called a square annulus. The square $x+[-R,R]^2$ (respectively $x+(-r,r)^2$) is called the outer square (respectively the inner square); $R$ and $r$ are called the outer radius and the inner radius. If $A$ is a square annulus, the $k$-arm event in $A$ is the event that there exist $k$ paths (included in $A$) of alternating colours from the boundary of the inner square of $A$ to the boundary of the outer square of $A$. Let $1 \leq r < R$. We write $\alpha_k(r,R)$ for the probability of this event with $A = [-R,R]^2 \setminus (-r,r)^2$ and $p=p_c=1/2$. We will also need the notion of $k$-arm events in the half-plane. We use the same definitions except that we ask that the paths live in the annulus intersected with the upper half-plane $\R \times \R_+$ (and the estimates that we are going to state below are also true for the lower, right and left half-planes). Finally, we will need the notion of $k$-arm event in the quarter plane that is the obvious analogue in the quarter-plane. The analogues of $\alpha_k(r,R)$ in the half-plane and in the quarter-plane are denoted (following~\cite{GPS}) by $\alpha^+_k(r,R)$ and $\alpha^{++}_k(r,R)$.

We write $\alpha_1(R)$ for the probability of $\lbrace 0 \leftrightarrow R \rbrace$. Note that $\alpha_1(R) \asymp \alpha_1(1,R)$. Also, we write $\alpha_k(R) := \alpha_k(1,R)$ for any $k \geq 2$. If $r \geq R$, we let $\alpha_k(r,R):=1$.

By using RSW techniques, one can prove that there exists $C =C(k) \in [1,+\infty)$ such, that for all $R \geq r$ large enough:
\begin{equation}\label{poly}
\frac{1}{C} \left( \frac{r}{R} \right)^C \leq \alpha_k(r,R) \leq C \left( \frac{r}{R} \right)^{1/C} \, .
\end{equation}
The obvious analogous results for $\alpha_k^+(r,R)$ and $\alpha_k^{++}(r,R)$ also hold. An important property (also true for $\alpha_k^+(\cdot,\cdot)$ and $\alpha_k^{++}(\cdot,\cdot)$) is the \textbf{quasi-multiplicativity} property (see~\cite{kesten1987scaling}, the Appendix of~\cite{SS10} or Section~$4$ of~\cite{nolin2008near}): there exists $C =C(k) \in [1,+\infty)$ such that, for all $r_3 \geq r_2 \geq r_1 \geq 1$:
\begin{equation}\label{quasimulti}
\frac{1}{C} \, \alpha_k(r_1,r_2) \, \alpha_k(r_2,r_3) \leq \alpha_k(r_1,r_3) \leq C \, \alpha_k(r_1,r_2) \, \alpha_k(r_2,r_3) \, .
\end{equation}

In this paper, we will only use $\alpha_4$, $\alpha_2$, $\alpha_1$, $\alpha_3^+$ and $\alpha_3^{++}$ (for this last quantity, we will actually only use that $\alpha_3^{++}(r,R) \leq \alpha_3^+(r,R)$). 
For site percolation on $\T$ it is proved in~\cite{lawler2002onearm} and~\cite{smirnov2001critical} that:
\begin{eqnarray}\label{5/48}
\alpha_1(r,R) & = & \left( r/R \right)^{5/48+\petito{1}} \, ,\\
\alpha_2(r,R) & = & \left( r/R \right)^{1/4+\petito{1}} \, ,\\
\label{5/4}
\alpha_4(r,R) & = & \left( r/R \right)^{5/4+\petito{1}} \, ,
\end{eqnarray}
where $R \geq r \geq 1$ and $\petito{1} \rightarrow 0$ as $r/R \rightarrow 0$.

Contrary to the above arm-exponents, the exponent of the $3$-arm event in the half-plane has been computed (thanks to RSW techniques) for both models: Let $R \geq r \geq 1$. It has been shown by Aizenman (see for example \cite{werner2007lectures}) that, for site percolation on $\T$ and for bond percolation on $\Z^2$:  
\begin{equation}\label{alpha3+}
\alpha_3^+(r,R) \asymp \left(r/R \right)^2\,.
\end{equation}

For bond percolation on $\Z^2$,  we have the following weaker estimates on $\alpha_4(r,R)$: there exists $C < +\infty$ and $\epsilon > 0$ such that for all $R \geq r \geq 1$:
\begin{equation}\label{alpha4}
\epsilon \; \frac{1}{\alpha_1(r,R)} \left( \frac{r}{R} \right)^{2-\epsilon} \leq \alpha_4(r,R) \leq C \frac{r}{R} \sqrt{\alpha_2(r,R)} \, .
\end{equation}
See the appendix of~\cite{GPS} for the left-hand inequality of~\eqref{alpha4} and Lemma~B$.1$ of~\cite{SS11} for the right-hand inequality (this is not exactly the content of this lemma but this is a direct consequence of its proof since the inequality~$($B$.6)$ can be replaced by $\E \left[ Y_j \right] \lesssim \alpha_2(r,R)$). See Chapter~$6$ of~\cite{book} for more references about such inequalities.



\subsection{Second moment method and exclusion sensitivity}

In what follows
, we fix a symmetric transition kernel $K$ from Definition \ref{d.kernels} on the sites of $\T$ or on the edges of $\Z^2$ and \textbf{we work at the critical point $p = p_c = 1/2$}. We shall always denote the associated exclusion dynamical percolation process by $\left( \omega_K(t) \right)_{t \geq 0}$. Inspired by the case of \textbf{i.i.d. dynamical percolation} from \cite{olle1997dynamical, SS10, GPS}, the only strategy which is known so far to identify the existence of exceptional times is to rely on the classical second moment method. In the present setting (see for example \cite{SS10,GPS,book}), it boils down to proving the following estimate: 

\begin{prop}\label{secondmoment}
Let $f_R : \Omega_R \rightarrow \lbrace 0,1 \rbrace$ be the indicator function of the radial event $\lbrace 0 \leftrightarrow R \rbrace$ ($f_R$ is well defined since this event only depends on the state of the sites/edges in $\mathcal{I}_R$, see Subsection~\ref{sectionarm} for the definitions of $\Omega_R$ and $\mathcal{I}_R$). To prove that a.s. there exist exceptional times for the $K$-exclusion dynamical percolation (of parameter $p_c=1/2$), it is sufficient to prove that there exists a constant $C=C(K)<+\infty$ such that, for all $R \geq 1$:
\begin{align}\label{e.SM1}
\int_0^1 \E \left[ f_R(\omega_K(0)) \, f_R(\omega_K(t)) \right] \, dt \leq C \alpha_1(R)^2 \, ,
\end{align}
where $f_R(\omega_K(t))$ means that we apply $f_R$ to the restriction of $\omega_K(t)$ to the sites/edges in $\mathcal{I}_R$.
\end{prop}

\begin{proof}[Proof (sketch)]
First note that it follows from Kolmogorov $0$-$1$ law that either a.s. there are exceptional times or a.s. there is no exceptional time. Hence, it is sufficient to prove that~\eqref{e.SM1} implies that there are exceptional times with positive probability. As explained in the i.i.d. case in \cite{SS10} (in the paragraph above Lemma~$5.1$, see also Proposition~$11.3$ in \cite{book}), this is a simple consequence of a second moment inequality. The only properties of the dynamical process that are used in the paragraph above Lemma~$5.1$ in~\cite{SS10} are: \textit{(a)} the fact that $\Pro_{1/2}$ is an invariant measure and \textit{(b)} a topological property about the set of exceptional times. This topological property is Lemma~$3.2$ of~\cite{olle1997dynamical}, that we state below in the case of the exclusion process and whose proof is exactly the same as for the i.i.d. process.
\end{proof}

\begin{lem}
Let $\left( \overline{\omega}_K(t) \right)_{t \geq 0}$ be obtained from $\left( \omega_K(t) \right)_{t \geq 0}$ by setting, for every $i \in \mathcal{I}$ (i.e. every edge or every site depending on the model), the set $\lbrace t \geq 0 \, : \, \overline{\omega}_K(t)_i = 1 \rbrace$ to be the closure of $\lbrace t \geq 0 \, : \, \omega_K(t)_i = 1 \rbrace$. Then a.s.:
\[
\lbrace t \, : \, 0 \overset{\overline{\omega}_K(t)}{\longleftrightarrow} \infty \rbrace = \lbrace t \, : \, 0 \overset{\omega_K(t)}{\longleftrightarrow} \infty \rbrace \, .
\]
\end{lem}

From now on, for any $R \geq 1$, $f_R$ will always be the indicator function of $\lbrace 0 \leftrightarrow R \rbrace$ (defined on $\Omega_R$).
\medskip

In Theorem~\ref{tropbien} and Proposition~\ref{tropbienlog}, we are also interested in the Hausdorff dimension of the set of exceptional times. We have the following proposition similar to Proposition~\ref{secondmoment} which provides lower-bound estimates on the Hausdorff dimension (upper-bounds are much easier to obtain, see  Proposition~\ref{upperhausdorff}). 

\begin{prop}\label{hausdorffcrit}
Let $d \in [0,1]$. To prove that the Hausdorff dimension of the set of exceptional times of a $K$-exclusion dynamical percolation (of parameter $p_c=1/2$) is an a.s. constant larger than or equal to $d$, it is sufficient to show that for any $\gamma < d$ there exists a constant $C=C(K,\gamma)$ such that, for all $R \geq 1$:
\begin{align}\label{e.SM2}
\int_0^1 \left( \frac{1}{t} \right)^{\gamma} \E \left[ f_R(\omega_K(0)) \, f_R(\omega_K(t)) \right] \, dt \leq C \alpha_1(R)^2 \, .
\end{align}
\end{prop}

\ni
{\em Proof (sketch).}
The fact that the Hausdorff dimension of the set of exceptional times is an a.s. constant follows from Kolmogorov $0$-$1$ law.
So, it is sufficient to prove that if~\eqref{e.SM2} holds for any $\gamma < d$, then the Hausdorff dimension of the set of exceptional times is at least $d$ with positive probability. The analogous result for the i.i.d. process is proved in Section~$6$ of~\cite{SS10}, where the authors use compactness arguments and the classical Frostman's criterion. Since the proof is exactly the same in our case, we refer to~\cite{SS10}. \qed
\medskip

As one can see in estimates~\eqref{e.SM1} and~\eqref{e.SM2}, proving the existence of exceptional times (and estimating their ``size'') thus requires to obtain good enough quantitative estimates on the correlations
\[
(R,t) \mapsto \E \left[ f_R(\omega_K(0)) \, f_R(\omega_K(t)) \right] \, .
\]

Usually in this situation, a legitimate intermediate problem is to analyse the \textbf{noise sensitivity} of non-degenerate percolation events (such as left-right crossing events of rectangles whose probability do not degenerate to zero as for the events $\{0 \leftrightarrow R\}$). Identifying noise sensitivity is only an intermediate step as it is far from being quantitative enough in general to imply existence of exceptional times. For example, the seminal work \cite{BKS} on noise sensitivity was not sufficient to imply the existence of exceptional times, which was achieved only later in \cite{SS10}. In \cite{BKS}, Benjamini, Kalai and Schramm consider the so-called left-right crossing events of the square $[-n,n]^2$ which are Boolean functions $g_n : \{-1 ,1 \}^{\Omega_n} \to \{0,1\}$ (see \cite{BKS, SS10, GPS}). Their main theorem is to show that for any fixed $t >0$, if one runs an \textbf{i.i.d. dynamics}, then:
\begin{align*}
\text{Cov} \left( g_n(\omega(0)), g_n(\omega(t)) \right) \underset{n\to +\infty} \longrightarrow 0 \, .
\end{align*}

In order to identify existence of exceptional times with a conservative dynamics such as $\left( \omega_K(t) \right)_{t \geq 0}$, a legitimate first step (analogous to \cite{BKS} in the i.i.d. setting) is therefore to identify noise sensitivity of Boolean functions such as $g_n$ under exclusion dynamics.   
%
This is exactly what was achieved in \cite{BGS} where Broman, the first author and Steif study the \textbf{exclusion sensitivity of Boolean functions}. Let us say a few words about it. Consider a sequence of Boolean functions $h_R : \Omega_R=\{-1,1\}^{\mathcal{I}_R} \rightarrow \lbrace 0,1 \rbrace$ (see Subsection~\ref{sectionarm} for the definition of $\mathcal{I}_R$).
The notion of exclusion sensitivity is defined as the analogue of the notion of noise sensitivity (\cite{BKS}) in the context of exclusion processes. More precisely,  the sequence $(h_R)_R$ is \textbf{$K$-exclusion sensitive} if, for all $t>0$:
\[
\text{Cov} \left( h_R(\omega_K(0)), h_R(\omega_K(t)) \right) \underset{R \rightarrow +\infty}{\longrightarrow} 0 \, ,
\]
where $h_R(\omega_K(t))$ means that we apply $h_R$ to the restriction of $\omega_K(t)$ to the sites/edges in $\mathcal{I}_R$.

In \cite{BGS}, it is proved that, if $g_n$ is the above left-right crossing event of the square $[-n,n]^2$ and if $K=K^\alpha$ is an $\alpha$-power law kernel on the sites of $\T$ with $\alpha$ sufficiently small (see Definition \ref{d.kernels}), then $(g_n)_n$ is $K^\alpha$-exclusion sensitive. (Actually, the matrices studied in \cite{BGS} are not exactly the $\alpha$-stable matrices from Definition \ref{d.kernels} but their methods apply to these last matrices at least with $\alpha$ small.)
As we shall see in the next subsections, both for i.i.d. and exclusion dynamics, the main technology behind identifying noise sensitivity and exceptional times turns out to be a careful spectral analysis of Boolean functions such as $g_n$ and $f_R$. As we will see in more details, it is useful to keep in mind the following informal distinction between the two:
\bi
\item[i)] Identifying {\bf noise sensitivity} corresponds to proving that most of the spectral mass of $g_n$ or $f_R$ is supported on large frequencies (possibly in a quantitative manner, i.e. most of the spectral mass is supported on sets of size $n^\alpha$ for example). 
\item[ii)] While identifying \textbf{exceptional times} requires a much more delicate analysis of the spectral mass: quantitative upper-bounds on the \textbf{lower tail} of $\widehat{f}_R$ are required (see \cite{SS10, GPS}).  
\ei

\subsection{The spectral sample in the i.i.d. setting}\label{ss.spectral}

Our main goal is to prove that~\eqref{e.SM1} holds (for the transition kernels of Theorem~\ref{tropbien} and Proposition~\ref{tropbienlog}). To this purpose, we analyse the quantities $\E \left[ f_R(\omega_K(0)) \, f_R(\omega_K(t)) \right]$, first by following ideas of~\cite{BGS}, and next by applying results on the \textbf{spectral sample} of $f_R$. In order to define the spectral sample and explain its links with the correlations $\E \left[ f_R(\omega_K(0)) \, f_R(\omega_K(t)) \right]$, we first need to introduce the notion of \textbf{Fourier decomposition of Boolean functions} that is used to study percolation in the seminal work \cite{BKS} (see also~\cite{book}). In this context, we see $f_R$ as an element of $L^2\left( \Omega_R,\Pro_{1/2} \right)$ which is the space of functions from $\Omega_R$ to $\R$ endowed with the scalar product $<h,h'>=\E_{1/2} \left[ h(\omega)h'(\omega) \right]$. (The probability measure $\Pro_{1/2}$ can be seen equivalently as the restriction to $\Omega_R$ of the usual $\Pro_{1/2}$ defined on $\Omega = \{ -1,1 \}^\mathcal{I}$ or as the uniform measure on $\Omega_R$.) For every $S \subseteq  \mathcal{I}_R$ and every $\omega \in \Omega_R$, let:
\[
\chi_S(\omega)=\prod_{i \in S} \omega_i \, .
\]
(In particular $\chi_\emptyset$ is the constant function $1$.) It is not difficult to check that $\left( \chi_S \right)_{S \subseteq \mathcal{I}_R}$ is an orthonormal basis of $L^2\left( \Omega_R,\Pro_{1/2} \right)$. Therefore, for any function $h = h_R \in L^2\left( \Omega_R,\Pro_{1/2} \right)$ we can define the Fourier decomposition of $h$ as the unique family of real numbers $(\widehat{h}(S))_{S \subseteq \mathcal{I}_R}$ such that:
\[
h = \sum_{S \subseteq \mathcal{I}_R} \widehat{h}(S) \, \chi_S \, .
\]
(Note that $\widehat{h}(\emptyset)=\E_{1/2} \left[ h(\omega) \right]$.) The reason to introduce this orthonormal basis is that it diagonalizes the i.i.d. dynamics $t \mapsto \omega(t)$:
\[
\E \left[ \chi_S(\omega(0)) \chi_{S'}(\omega(t)) \right] = \delta_{S,S'} \, e^{-t|S|} \, .
\]
As a result, we have:
\[
\E \left[ h(\omega(0)) h(\omega(t)) \right] = \sum_{S \subseteq \mathcal{I}_R} \widehat{h}(S)^2 e^{-t|S|} \, .
\]

To gain more geometric intuitions, it is interesting to view the coefficients $\widehat{h}(S)^2$ as weights of a probability measure on the sets $S \subseteq \mathcal{I}_R$. This is the approach followed in~\cite{GPS}:

\begin{defi}[\textup{\cite{GPS}}]\label{spectralsample}
Let $R \geq 1$ and $h \in L^2\left( \Omega_R,\Pro_{1/2} \right) \setminus \lbrace 0 \rbrace$. A \textbf{spectral sample} of $h$ is a random variable on the sets $S \subseteq \mathcal{I}_R$ whose distribution $\widehat{\Pro}_{h}$ is given by:
\begin{align}\label{e.SM}
\widehat{\Pro}_{h} \left[ \lbrace S \rbrace \right] = \frac{\widehat{h}(S)^2}{\E_{1/2} \left[ h(\omega)^2 \right]} \, .
\end{align}
We write $\widehat{\E}_h$ for the corresponding expectation. Note that if $h=f_R$ then $\E_{1/2} \left[ h(\omega)^2 \right] = \alpha_1(R)$. We will sometimes work with the unnormalized measure $\widehat{\Q}_h$ given by:
\[
\widehat{\Q}_{h} \left[ \lbrace S \rbrace \right] = \widehat{h}(S)^2 \, .
\]
\end{defi}
For some ideas behind the study of the spectral sample and its links with the pivotal set, we refer to~\cite{book} (Chapters~$9$ and~$10$). We now state one of the main theorems from~\cite{GPS} which quantifies exactly what is the \textbf{lower tail} of the spectral measure of the above radial functions $f_R : \Omega_R \rightarrow \lbrace 0,1 \rbrace$. This theorem holds for our two models: site percolation on $\T$ and bond percolation on $\Z^2$.

\begin{thm}[Theorem~$7.3$ of~\cite{GPS}, see also Theorem~$10.22$ in~\cite{book} and Exercise $10.7$ of~\cite{book} for the lower-bound part]\label{GPScool}
Let $R \geq r \geq 1$, then:
\[
\widehat{\Pro}_{f_R} \left[ |S| < r^2\alpha_4(r) \right] \asymp \frac{\alpha_1(R)}{\alpha_1(r)} \, .
\]
Let $\rho(l)=\inf \lbrace r \, : \, r^2\alpha_4(r) \geq l \rbrace$. The above estimate implies that there exists some $C < + \infty$ such that, for all $l \in \N_+$ and all $R \geq 1$:
\[
\widehat{\Pro}_{f_R} \left[ |S| < l \right]  \leq C \frac{\alpha_1(R)}{\alpha_1(\rho(l))} \, .
\]
\end{thm}
Using this spectral estimate (which is highly non-trivial), it is not very hard to deduce the existence of exceptional times for \textbf{i.i.d. dynamical percolation} on $\T$ and $\Z^2$ at $p=p_c$. Indeed writing:
\begin{align*}
\Eb{f_R(\omega(0)) f_R(\omega(t))} & = \sum_{S \subseteq \mathcal{I}_R} \widehat f_R(S)^2 e^{- t |S|} =  \widehat{\mathbb{E}}_{f_R} \left[ e^{- t |S|} \right] 
\end{align*}
and using the above theorem, it is rather straightforward to check that the hypothesis from Propositions \ref{secondmoment} and \ref{hausdorffcrit} are satisfied (see \cite{GPS,book}).

\subsection{Spectral representation of correlations in the conservative case}\label{ss.SC}

In order to apply the above strategy to a \textbf{$K$-exclusion dynamics}, the  first natural idea would be to decompose the Boolean functions $f_R$ and $g_n$ on an appropriate basis which diagonalizes the dynamics $t \mapsto \omega_K(t)$. Unfortunately, such basis are both non-local and non-explicit. Therefore, we still project the Boolean functions $f_R$ on the above basis $\{ \chi_ S\}_S$, at the cost of having additional non-diagonal terms.
As observed in \cite{BGS}, 
for all $S,S' \subseteq \mathcal{I}_R$ one has the following simple correlation structure
\begin{align}\label{e.Kt}
\E \left[ \chi_S(\omega_K(0)) \, \chi_{S'}(\omega_K(t)) \right] = K_t(S,S')\,,
\end{align}
where the transition matrix $K_t$ on sets $S \subseteq \mathcal{I}_R$ is defined as follows.
\begin{definition}
 If $S$ and $S'$ are two finite subsets of $E$ we write:
\begin{equation}\label{matrixS}
K_t(S,S') = \Pro \left[ \pi_t(S) = S' \right]\,,
\end{equation}
where $\pi_t$ is the random permutation used in Appendix \ref{a.graphical} to obtain a graphical construction of the exclusion process $t\mapsto \omega_K(t)$.
It is not difficult to see that $K_t$ is a symmetric transition matrix and that $K_t(S,S') = 0$ if $|S| \neq |S'|$. Hence, for any non-negative integers $k \leq l$, $K_t$ restricted to $\lbrace S \subseteq \mathcal{I}_R \, : \, |S| \in [k,l] \rbrace$ is still a symmetric transition matrix.
\end{definition}

\begin{remark}
The correlation formula~\eqref{e.Kt} is used throughout in \cite{BGS}. It is reminiscent of the so-called \textbf{duality formula} for exclusion processes. 
Note that our assumption that the kernels $K$ from Definition~\ref{d.kernels} are \textbf{symmetric} is crucial if one wants to rely on~\eqref{e.Kt}. 
\end{remark}

The importance of this duality formula (as used in \cite{BGS})  is due to its following consequence. One has for any Boolean function $h = h_R : \Omega_R = \{-1,1\}^{\calI_R}\to \{0,1\}$: 
\begin{equation}\label{premiersecondmoment}
\E \left[ h(\omega_K(0)) \, h(\omega_K(t)) \right] = \sum_{S,S' \subseteq \mathcal{I}_R} \widehat{h}(S) \widehat{h}(S') K_t(S,S') \, .
\end{equation}

\subsection{Outline of proof and new spectral estimates}\label{ss.outline}

Let us now give a short outline of the proof of our main result Theorem~\ref{tropbien}, and of its easier analogue Proposition~\ref{tropbienlog}. 
In order to prove that there exist exceptional times, it is sufficient to show 
(combining Proposition~\ref{secondmoment} with equation~\eqref{premiersecondmoment}) that there exists $C = C(K) < + \infty$ such that, uniformly in $R \geq 1$:
\begin{equation}\label{secondsecondmoment}
\int_0^1 \sum_{S,S' \subseteq \mathcal{I}_R} \sqrt{\widehat{\Pro}_{f_R}\left[ \lbrace S \rbrace \right]} \sqrt{\widehat{\Pro}_{f_R}\left[ \lbrace S' \rbrace \right]} \, K_t(S,S') \, dt \leq C \, \alpha_1(R) \, .
\end{equation}

(The fact that it is $\alpha_1(R)$ on the right-hand side instead of $\alpha_1(R)^2$ as in~\eqref{e.SM1} is due to the fact that the spectral measure $\widehat{\Pro}_{f_R}$ is renormalized by $\alpha_1(R)$, see equation~\eqref{e.SM}.) To obtain a lower-bound on the Hausdorff dimension of these exceptional times, one needs to show the following strengthening: 
there exists $C=C(K,\gamma) < + \infty$ such that for all $R \geq 1$:
\begin{equation}\label{secondsecondmomenthausdorff}
\int_0^1 \left( \frac{1}{t} \right)^{\gamma} \sum_{S,S' \subseteq \mathcal{I}_R} \sqrt{\widehat{\Pro}_{f_R}\left[ \lbrace S \rbrace \right]} \sqrt{\widehat{\Pro}_{f_R}\left[ \lbrace S' \rbrace \right]} \, K_t(S,S')\, dt \leq C \, \alpha_1(R) \, .
\end{equation}
As in Proposition~\ref{hausdorffcrit}, the larger the value of $\gamma$ is, the better the lower-bound on the Hausdorff dimension is. 

In order to explain the intuition which underlies our proofs, let us write informally the above sum as ``$ \, \Big< \, \sqrt{\widehat{\Pro}_{f_R}} \, , \, K_t \star \sqrt{\widehat{\Pro}_{f_R}} \, \Big> \, $". With this in mind, our purpose becomes to show \textbf{quantitatively} that ``$ \, \sqrt{\widehat{\Pro}_{f_R}}$ and $K_t \star \sqrt{\widehat{\Pro}_{f_R}} \, $ are asymptotically singular". One way to interpret this is as follows: if we let a spectral sample (of $f_R$) evolve under a $K$-exclusion process for some 
time $t>0$, then it does not look like a spectral sample any more. In other words, if we sample a spectral sample $\mathcal{S} \sim \widehat{\Pro}_{f_R}$ independently of our exclusion process and if we let $\pi_t$ be the permutations defined in~\eqref{defiofpi}, then, for any fixed $t>0$ and any $R$ sufficiently large, $\pi_t(\mathcal{S})$ does not look like a typical spectral sample any more. 
In order to prove this, one needs to identify ``almost sure'' properties of the spectral sample $\mathcal{S}\sim \widehat{\Pro}_{f_R}$ which will no longer hold (with high probability) for $\pi_t(\mathcal{S})$, namely after diffusion. The main mathematical issue we face here is that the actual purpose of the previous works about the spectral sample was to estimate its size (as one can see for example from the above Theorem~\ref{GPScool} from~\cite{GPS}). 
This is not interesting for our purpose since $\mathcal{S}$ and $\pi_t(\mathcal{S})$ have equal size. 
%
What will help us is that the strategy in~\cite{GPS} is to study closely the geometry of the spectral sample. As such, our strategy will consist in  identifying  ``almost sure'' geometric properties of $\mathcal{S}\sim \widehat{\Pro}_{f_R}$ which will no longer hold for $\pi_t(\mathcal{S})$. 

This strategy is close to the strategy of~\cite{BGS} for the proof of exclusion sensitivity of the left-right crossing events. There is a significant difference though (very similar to the difference between \cite{BKS} and \cite{SS10}) as we need to obtain quantitative bounds essentially on the \textbf{``lower tail''} (i.e. on the atypically small spectral samples $|\mathcal{S}| \ll  \widehat{\E}_{f_R} \left[ |S| \right]$). The difficulty behind this is that we will need to find singularities for all sizes of spectral samples. More precisely, let $g_n$ be the indicator function of the crossing of $[-n,n]^2$ from left to right. In order to prove that $(g_n)_n$ is $K$-exclusion sensitivity, the authors of~\cite{BGS} had to show that:
\[
\sum_{\emptyset \neq S,S' \subseteq \mathcal{I}_n} \sqrt{\widehat{\Pro}_{g_n}\left[ \lbrace S \rbrace \right]} \sqrt{\widehat{\Pro}_{g_n}\left[ \lbrace S' \rbrace \right]} \, K_t(S,S') \underset{n \rightarrow +\infty}{\longrightarrow} 0 \ .
\]
Thanks to~\cite{GPS}, Theorem~$1.1$ (and thanks to the Cauchy-Schwarz inequality and the Markov property of $K_t(S,\cdot)$), it is easy to see that:
\[
\sum_{\emptyset \neq S,S' \subseteq \mathcal{I}_n \, : \atop |S| = |S'| \ll \widehat{\E}_{g_n} \left[ |S| \right] \text{ or } |S| = |S'| \gg \widehat{\E}_{g_n} \left[ |S| \right]} \sqrt{\widehat{\Pro}_{g_n} \left[ \lbrace S \rbrace \right]} \sqrt{\widehat{\Pro}_{g_n}\left[ \lbrace S' \rbrace \right]} \, K_t(S,S') \underset{n \rightarrow +\infty}{\longrightarrow} 0 \, .
\]
This way, the authors of~\cite{BGS} only had to take into account the sets $S$ whose size is roughly $\widehat{\E}_{g_n} \left[ |S| \right]$. This made the analysis in~\cite{BGS} easier as in this regime, the spectral sample $\calS \sim \widehat{\Pro}_{g_n}$ is known to be essentially ``fractal''. In our present setting, one cannot avoid a detailed analysis of what happens in the lower tail. Indeed if one were to apply the same trick (Cauchy-Schwarz and Markov property) to small spectral sets of size $|\mathcal{S}| < r^2 \alpha_4(r)$ with $r\ll R$ and $\mathcal{S} \sim \widehat{\Pro}_{f_R}$, then one would obtain thanks to Theorem \ref{GPScool} the following bound: For all $t>0$,
\[
 \sum_{\emptyset \neq S,S' \subseteq \mathcal{I}_R \, : \atop |S| = |S'| < r^2 \alpha_4(r) } \sqrt{\widehat{\Pro}_{f_R}\left[ \lbrace S \rbrace \right]} \sqrt{\widehat{\Pro}_{f_R}\left[ \lbrace S' \rbrace \right]} \, K_t(S,S') \leq \grandO{1}  \frac {\alpha_1(R)}{\alpha_1(r)} \, .
\]
Clearly, such a bound is not quantitative enough to imply what we need, namely:
\[
\int_0^1 \sum_{S,S' \subseteq \mathcal{I}_R } \sqrt{\widehat{\Pro}_{f_R}\left[ \lbrace S \rbrace \right]} \sqrt{\widehat{\Pro}_{f_R}\left[ \lbrace S' \rbrace \right]} \, K_t(S,S')\, dt \leq \grandO{1} \alpha_1(R) \, .
\]
Because of this, we are required to identify a geometric singularity between $\calS \sim \widehat {\Pro}_{f_R}$ and $\pi_t(\calS)$ even when $\calS$ is atypically small. In other words, we need to quantify the singularity between the sub-probability measures (when $r\ll R$):
\[
\un_{|S|<r^2 \alpha_4(r)}\widehat{\Pro}_{f_R}(dS) \text{ \;  and \;   } K_t \star \left[ \un_{|S|<r^2 \alpha_4(r)} \widehat{\Pro}_{f_R}(dS) \right].
\] 
Imagine for a second that such small spectral sets typically looked (under the conditional measure $\widehat{\Pro}_{f_R}\left[ \; \cdot \; \mid \; |S|<r^2 \alpha_4(r) \right]$) like macroscopic ``Poissonnian clouds'' of points. In that case, the above sub-probability measures would even be ``absolutely continuous'' with respect to each other. To prevent this, the geometric feature of these small spectral sets which will help us detecting singularity is a certain \textbf{clustering effect} which will be proved and made quantitative in this paper (see Theorem \ref{soimportant} below). More precisely, under the conditional measure $\widehat{\Pro}_{f_R}\left[ \; \cdot \; \mid \; |S|<r^2 \alpha_4(r) \right]$, spectral sets tend to be of ``small'' diameter. Note that such a clustering effect is far from being obvious (techniques from \cite{GPS} are not well designed for such properties) and it is still an open-problem for the left-right crossing events $g_n$, see Conjecture \ref{c.LR}. Summarising the above discussion, our proof of Theorem \ref{tropbien} is divided into the following two independent steps (see also Figure \ref{f.strategy}).

\bnum
\item[A.] \textbf{Clustering property for small radial spectral sets.} This step corresponds to Theorem \ref{soimportant} below (and its Corollary \ref{soimportantcor}). Note that this step of the proof is purely {\em static} (no dynamics here). It will be the purpose of Section \ref{sectionspectralnotlocalized}. 
 
\item[B.] \textbf{From clustering to singularity to exceptional times.} The second step of the proof 
consists in implementing the above clustering property into a sufficiently quantitative singular behaviour in order to obtain existence of exceptional times (main Theorem~\ref{tropbien}). 
This second step will be the purpose of Section~\ref{sectionexcepcriticality}.
\enum

\begin{figure}[!htp]
\begin{center}
\includegraphics[width=\textwidth]{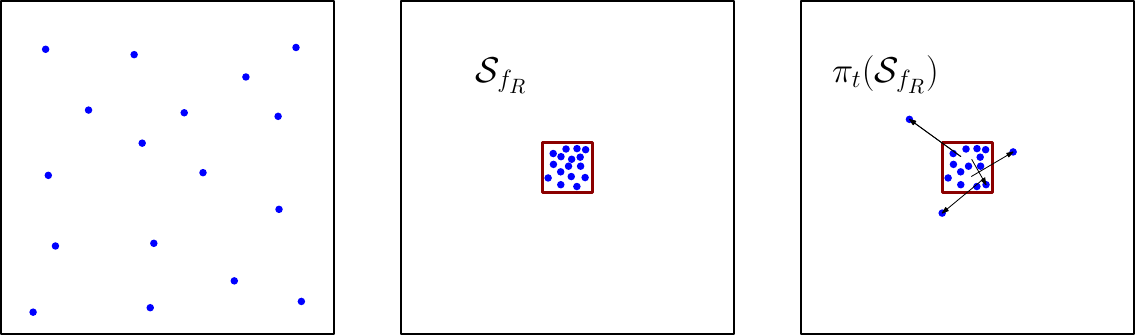}
\end{center}
\caption{This picture illustrates our strategy for our spectral analysis of the lower tail under a $K^\alpha$-exclusion dynamics. If spectral samples of radial events $f_R$ were to look like ``sparse'' random subsets of $\mathcal{I}_R$ as pictured on the left, it would be very hard to detect any singular behaviour between $\mathcal{S}$ and $\pi_t(\mathcal{S})$. This is why we provide a quantitative clustering property in Theorem~\ref{soimportant} and Corollary~\ref{soimportantcor} which shows that small spectral samples are with high probability concentrated in small balls (middle square). Once we combine all quantitative estimates, one can show that when the exponent $\alpha$ is chosen small enough, this clustering property does not hold any more after diffusion for $\pi_t(\mathcal{S})$ as pictured on the right square. This is how we manage to detect the desired (quantitative) singularity.}\label{f.strategy}
\end{figure}

We end this subsection with more details for steps A and B.

\paragraph{A. Clustering property.}
Our main result on the clustering property of the spectral sample of the \textbf{radial crossing event} $f_R$ can be stated as follows. In this result, we estimate the probability of a small residual spectral mass away from the origin.
\begin{thm}\label{soimportant}
There exist an exponent  $\epsilon > 0$ and a constant $C < +\infty$ such that, for all $1 \leq r \leq r_0 \leq R/2$:
\[
\widehat{\Pro}_{f_R} \left[ 0 < |S \setminus (-r_0,r_0)^2| < r^2 \alpha_4(r) \right] \leq C \, \frac{\alpha_1(R)}{\alpha_1(r_0)} \left( \frac{r_0}{r} \right)^{1-\epsilon} \alpha_4(r,r_0) \, .
\]
(As explained in Remark \ref{r.epscrucial}, this exponent $\eps>0$ even very small is crucial for the existence of exceptional times on $\Z^2$. It is related to the geometric event discussed in Appendix \ref{a.eps}, see also Remark \ref{aquantitaiveepsilon}.) 
\end{thm}

Theorem~\ref{soimportant} will be proved in Subsection~\ref{sectionlittlepart}. 
Its proof is mostly inspired by the global proof in~\cite{GPS}. There are three main steps in~\cite{GPS}, which correspond to Sections~4,~5 and~6. We will adapt the first step and then use the two other steps identically to Sections~5 and~6 of~\cite{GPS}. As the treatment of the first step differs in at least three key places, we provide a reasonably self-contained proof in Subsection~\ref{sss.verylittle}. To help identifying the differences with the proof in Section~$4$ of~\cite{GPS}, here are the three main ones: 
\bnum
\item First, one needs to introduce a new combinatorial \textbf{annulus structure} which is designed to analyse the spectral sample outside of some mesoscopic scale $(-r_0,r_0)^2$. 
\item In order to analyse this modified annulus structure, we need to introduce a new geometric percolation exponent (which is the exponent of the ``$4$-arm event conditioned on the percolation configuration in a half-plane'', see Lemma~\ref{lemmerigolo} and Lemma~\ref{exponenthalfplane}). This conditioned percolation event 
is at the root of the exponent $\eps>0$ in Theorem~\ref{soimportant} and will play a significant role while analysing the modified annulus structure. A key estimate on this conditioned percolation event which is also valid on $\Z^2$ is proved in Lemma~\ref{lemmerigolo}. See Remark~\ref{aquantitaiveepsilon} for the link between $\eps$ and this conditioned event and Remark~\ref{r.epscrucial} for the importance of $\eps>0$ in our proof of existence of exceptional times on $\Z^2$.
\item Finally, we need to adapt the useful spectral estimate Lemma~4.8 of \cite{GPS} to our annulus structure. That is the purpose of Lemma~\ref{generalizedannulusstructuresinequality} and Appendix \ref{a.B}.
\enum 
We shall also use extensively in Section \ref{sectionexcepcriticality} the following immediate corollary of Theorem~\ref{soimportant}, where we analyse the circumstance of a spectral sample of atypically high diameter given its size (see after the proof of Corollary~\ref{soimportantcor} for a further discussion):

\begin{cor}\label{soimportantcor} Let $\epsilon$ be the same constant as in Theorem~\ref{soimportant}. Then:
\begin{enumerate}
\item There exists a constant $C<+\infty$ such that, for all $1 \leq r \leq r_0$ and all $R \geq 1$:
\[
\widehat{\Pro}_{f_R} \left[ |S| < r^2 \alpha_4(r), \, S \nsubseteq (-r_0,r_0)^2 \right] \leq C \, \frac{\alpha_1(R)}{\alpha_1(r_0)} \left( \frac{r_0}{r} \right)^{1-\epsilon} \alpha_4(r,r_0) \, .
\]
We will use this result as follows:
\item Consider $\beta>1$ such that there exists some $\overline{k}=\overline{k}(\beta)$ such that for all $k \geq \overline{k}$ we have $\rho(2^{k+1}) \leq 2^{k\beta}$ (see Theorem~\ref{GPScool} for the definition of $\rho$). Then, there exists a constant $C=C(\beta)<+\infty$ such that, for all $k \in \N$ and all $R \geq 1$:
\[
\widehat{\Pro}_{f_R} \left[ |S| < 2^{k+1}, S \not\subseteq (-2^{k\beta},2^{k\beta})^2  \right] \leq C \, \frac{\alpha_1(R)}{\alpha_1(2^{k\beta})} \left( \frac{2^{k\beta}}{\rho(2^{k})} \right)^{1-\epsilon} \alpha_4(\rho(2^{k}),2^{k\beta}) \, .
\]
\end{enumerate}
\end{cor}
\begin{proof} We first prove item~$1$. We distinguish between three cases : \textit{(a)} If $r_0 \leq R/2$, then this is a direct consequence of Theorem~\ref{soimportant} since:
\[
\left\lbrace |S| < r^2 \alpha_4(r), \, S \not\subseteq (-r_0,r_0)^2 \right\rbrace \subseteq \left\lbrace 0 < |S \setminus (-r_0,r_0)^2| < r^2 \alpha_4(r) \right\rbrace \, .
\]
(We may have lost a lot in this inclusion, see Conjecture~\ref{conjquiseraitcool}.) \textit{(b)} If $r_0 > R+2$ then the left-hand side equals $0$ (since we have $\mathcal{I}_R \subseteq [-(R+2),R+2]^2$). \textit{(c)} If $R/2 < r_0 \leq R+2$, then this is a simple consequence of the case $r_0=R/2$ (together with the quasi-multiplicativity property and~\eqref{poly}) since we have:
\[
\left\lbrace |S| < r^2 \alpha_4(r), \, S \not\subseteq (-r_0,r_0)^2 \right\rbrace \subseteq \left\lbrace |S| < r^2 \alpha_4(r), \, S \not\subseteq (-R/2,R/2)^2 \right\rbrace \, .
\]

Let us now prove item~$2$. We distinguish between two cases: \textit{(a)} if $k < \overline{k}$, such an estimate is a direct consequence of Theorem~\ref{GPScool}. \textit{(b)} We now assume that $k \geq \overline{k}$. Then, this is a direct consequence of item~$1$ since we have $1 \leq \rho(2^{k+1}) \leq 2^{k\beta}$ and $\rho(2^{k+1})^2 \, \alpha_4(\rho(2^{k+1})) \geq 2^{k+1}$ (actually, it is not difficult to see that this last inequality is an equality). Note also that we have used the fact that $\rho(2^{k+1}) \leq \grandO{1} \rho(2^k)$, which is a simple consequence of the quasi-multiplicativity property and the left-hand inequality of~\eqref{alpha4}.
\end{proof}

At the level of this outline, let us analyse a little more the results of Theorem~\ref{soimportant} and Corollary~\ref{soimportantcor}. These results imply that if the spectral sample is small then it is ``localized in the neighbourhood of the origin":

Let us estimate $\widehat{\Pro}_{f_R} \left[ S \subseteq (-r_0,r_0)^2 \cond 0 < |S| < r^2 \alpha_4(r) \right]$. Thanks to Theorem~\ref{GPScool} and Corollary~\ref{soimportantcor} (together with the quasi-multiplicativity property), we have:
\begin{eqnarray*}
\widehat{\Pro}_{f_R} \left[ S \nsubseteq (-r_0,r_0)^2 \cond 0 < |S| < r^2 \alpha_4(r) \right] & \leq & \grandO{1} \frac{\alpha_1(r)}{\alpha_1(r_0)} \left( \frac{r_0}{r} \right)^{1-\epsilon} \alpha_4(r,r_0)\\
& \leq & \grandO{1} \frac{\alpha_4(r,r_0)}{\alpha_1(r,r_0)} \left( \frac{r_0}{r} \right)^{1-\epsilon} \, .
\end{eqnarray*}
The right-hand inequality of~\eqref{alpha4} and the FKG inequality imply that the above is at most $\grandO{1} \left( \frac{r_0}{r} \right)^{-\epsilon} \frac{\sqrt{\alpha_2(r,r_0)}}{\alpha_1(r,r_0)} \leq \grandO{1} \left( \frac{r_0}{r} \right)^{-\epsilon} \;$\footnote{Here we can see the importance of the constant $\epsilon$ in Theorem~\ref{tropbien}, even very small.}, which goes to $0$ as $r/r_0$ goes to $0$. (In the case of $\T$, thanks to the computation of the critical exponents we even know that $\frac{\alpha_4(r,r_0)}{\alpha_1(r,r_0)} \left( \frac{r_0}{r} \right)^{1-\epsilon} = \left( \frac{r_0}{r} \right)^{-7/48-\epsilon+\petito{1}}$.)
\medskip

We now state the (easier) analogue of Corollary~\ref{soimportantcor} which will be relevant for the long-range dynamics $K^a_{\log}$, namely to prove Proposition~\ref{tropbienlog}:
\begin{prop}\label{soimportantlog}
For all $\epsilon_0>0$ there exists a constant $C=C(\epsilon_0)<+\infty$ such that, for all $k \in \N_+$ and all $R \geq 1$:
\[
\widehat{\Pro}_{f_R} \left[ |S| < k, S \not\subseteq (-\exp(k^{\epsilon_0}),\exp(k^{\epsilon_0}))^2 \right] \leq C \, \alpha_1(R) \, \exp(-k^{\epsilon_0}/C) \, .
\]
\end{prop}
Proposition~\ref{soimportantlog} is a simple consequence of Corollary~\ref{soimportantcor} (and of Theorem~\ref{GPScool} for $k$ small), but is also a direct consequence of some results of Section~$4$ of~\cite{GPS}, see Subsection~\ref{sectionlog}.

\paragraph{B. From clustering to singularity.} 
Let us give a short heuristics which explains how to derive our main result Theorem~\ref{tropbien} using the above clustering property.
Let $\mathcal{S} \sim \widehat{\Pro}_{f_R}$ be a spectral sample independent of our exclusion process. Remember that we want to show that for all $t>0$ if $R$ is sufficiently large, then $\pi_t(\mathcal{S})$ does not look like a spectral sample with high probability. 
Corollary~\ref{soimportantcor} implies that, if $\beta$ is large enough, then $\mathcal{S}$ is included in the square $(-|\mathcal{S}|^\beta,|\mathcal{S}|^\beta)^2$ with high probability. Remember the definition of the transition matrices $K^\alpha$ in Definition~\ref{d.kernels}: each point of $\mathcal{S}$ has roughly probability $t/|\mathcal{S}|^{\alpha\beta}$ to ``jump" a distance greater than $3|\mathcal{S}|^\beta$. So, if $\alpha \ll 1/\beta$ then with high probability there exists a particle which has jumped a distance greater than $3|\mathcal{S}|^\beta$, and we have what we want: $\pi_t(\mathcal{S})$ is not included in $(-|\mathcal{S}|^\beta,|\mathcal{S}|^\beta)^2$, hence it is very different from a typical spectral sample  (in particular, if $|\mathcal{S}|^\beta$ is larger than $R+2$, then with high probability there even exists a particle which has jumped outside the domain of $\widehat{\Pro}_{f_R}$). We see from this heuristics why our bounds are worse and worse as the exponent $\alpha$ increases.

In order to derive our main result (Theorem \ref{tropbien}) we need a quantitative (and rigorous) version of this heuristics. This is the purpose of Section~\ref{sectionexcepcriticality}.


\section{Warm-up without spectral analysis: proofs of Propositions~\ref{outsidepc} and~\ref{dgeq11}}\label{sectionabsence}

In this section, we prove Propositions~\ref{outsidepc} and~\ref{dgeq11}. As mentioned in Subsection~\ref{sectionexclu}, the general ideas are the same as for the analogous results of \cite{olle1997dynamical}. However, there is a slightly new difficulty due to the lack of independence and that is the reason why we need the following two lemmas.

\begin{lem}\label{increas}
Let $K$ be a symmetric transition matrix on the edges of a graph $G=(V,E)$. Consider a law $\mu$ on the set of bond percolation configurations $\Omega=\lbrace -1,1 \rbrace^E$ that satisfies the following: there exists $p_0 \in [0,1]$ such that, for any $n \in \N$, any $e_1, \cdots , e_{n+1}$ distinct edges and any $i_1, \cdots , i_n \in \lbrace -1,1 \rbrace$, we have:
\[
\mu \left[ \omega(e_1)=i_1, \, \cdots , \, \omega(e_n)=i_n \right] > 0
\]
and:
\[
\mu \left[ \omega(e_{n+1})=1 \cond \omega(e_1)=i_1, \, \cdots , \, \omega(e_n)=i_n \right] \leq p_0 \, .
\]
Then, 
for any increasing event $A \in \mathcal{F}$ (i.e. an event such that, if $\omega \leq \omega'$ and $\omega \in A$, then $\omega' \in A$)
that depends on only finitely many edges, we have $\mu[A] \leq \Pro_{p_0}[A]$. 
\end{lem}
The proof of this lemma is straightforward. It is applied in our context as follows: since our graphs are locally finite, the lemma holds with:
\[
A = A_n(v) := \lbrace \exists \text{ a self avoiding path starting at } v \text{, of length } n \text{, and made of open edges} \rbrace \, .
\]
As a result, it is also true with $A = \lbrace v \leftrightarrow \infty \rbrace$ since we have $\lbrace v \leftrightarrow \infty \rbrace = \bigcap_{n \in \N} \downarrow A_n(v)$. We deduce that, if $p_0 < p_c$, then for all $v$, $\mu$-a.s. $v$ is not in an infinite cluster. Therefore, $\mu \left[ \exists \text{ an infinite cluster} \right] = 0$ (remember that the vertex set is countable).

\begin{lem}\label{increaslem}
Let $p \in (0,1)$ and let $\left( \omega_K(t) \right)_{t \geq 0}$ be a $K$-exclusion dynamical percolation of parameter $p$. Write $\omega_K^{(\epsilon)}$ for the configuration that equals $\omega_K(0)$ except that we set $\omega_K(0)_e = 1$ for every edge $e$ such that a clock associated to $e$ has rung between time $0$ and time $\epsilon$. If $e_1, \cdots , e_{n+1}$ are distinct edges and $i_1, \cdots , i_n \in \lbrace -1,1 \rbrace$, then:
\[
\Pro \left[ \omega^{(\epsilon)}_K(e_{n+1})=1 \cond \omega^{(\epsilon)}_K(e_1)=i_1, \, \cdots , \, \omega^{(\epsilon)}_K(e_n)=i_n \right] \leq p + \frac{1-p}{p} \epsilon \, .
\]
\end{lem}

\begin{proof}
Define the event $C_{e,f}^{(\epsilon)}$ as follows:
\[
C_{e,f}^{(\epsilon)} = \lbrace \text{the clock of }  \lbrace e,f \rbrace \text{ has rung between time } 0 \text{ and time } \epsilon \rbrace \, .
\]
Consider $e_1, \, \cdots , \, e_{n+1}$ and $i_1, \, \cdots , \, i_n$ as in the statement of the lemma. Note that $\omega_K(0)_{e_{n+1}}$ is independent of $\left\lbrace \omega^{(\epsilon)}_K(e_1)=i_1, \cdots, \omega^{(\epsilon)}_K(e_n)=i_n \right\rbrace$ and of this event intersected with $\lbrace \exists f \in E, \, C^{(\epsilon)}_{e_{n+1},f} \rbrace$. So, if we distinguish between the two cases $\omega(0)_{e_{n+1}}=1$ and $\omega(0)_{e_{n+1}}=-1$, we obtain:
\begin{align*}
& \Pro \left[ \omega^{(\epsilon)}_K(e_{n+1}) = 1 \cond \omega^{(\epsilon)}_K(e_1)=i_1, \, \cdots , \, \omega^{(\epsilon)}_K(e_n)=i_n \right]\\
& = p + (1-p) \, \Pro \left[ \exists f \in E, \; C^{(\epsilon)}_{e_{n+1},f} \cond \omega^{(\epsilon)}_K(e_1)=i_1, \, \cdots , \, \omega^{(\epsilon)}_K(e_n)=i_n \right]\\
& \leq p + (1-p) \sum_{f \in E} \Pro \left[ C_{e_{n+1},f}^{(\epsilon)} \cond \omega^{(\epsilon)}_K(e_1)=i_1, \, \cdots , \, \omega^{(\epsilon)}_K(e_n)=i_n \right] \, .
\end{align*}
If $f \notin \lbrace e_1, \cdots , e_n \rbrace$, then $C_{e_{n+1},f}^{(\epsilon)}$ is independent of $\left\lbrace \omega^{(\epsilon)}_K(e_1)=i_1, \, \cdots , \, \omega^{(\epsilon)}_K(e_n)=i_n \right\rbrace$. Moreover, if $f=e_j$ and $C_{e_{n+1},f}^{(\epsilon)}$ holds, then $\omega^{(\epsilon)}_K(e_j)=1$. Therefore, the above equals:

\begin{align}\label{sumovereedges}
& p + (1-p) \Bigg( \sum_{f \notin \lbrace e_1, \cdots , e_n \rbrace} \Pro \left[ C_{e_{n+1},f}^{(\epsilon)} \right] \nonumber \\
& + \sum_{j \in \lbrace 1, \cdots , n \rbrace: \atop i_j = 1} \frac{\Pro \left[ C_{e_{n+1},e_j}^{(\epsilon)}, \forall k \in \lbrace 1, \, \cdots, \, n \rbrace \setminus \lbrace j \rbrace, \, \omega^{(\epsilon)}_K(e_k)=i_k \right]}{\Pro \left[ \omega^{(\epsilon)}_K(e_1)=i_1, \cdots , \omega^{(\epsilon)}_K(e_n)=i_n \right]} \Bigg) \nonumber \\
& = p + (1-p) \Bigg( \sum_{f \notin \lbrace e_1, \cdots , e_n \rbrace} \Big( 1-\exp \big( -\epsilon K(e_{n+1},f) \big) \Big) \nonumber \\
& + \sum_{j \in \lbrace 1, \cdots, n \rbrace: \atop i_j = 1} \Big( 1-\exp \big( -\epsilon K(e_{n+1},f) \big) \Big) \frac{\Pro \left[ \forall k \in \lbrace 1, \, \cdots, \, n \rbrace \setminus \lbrace j \rbrace, \, \omega^{(\epsilon)}_K(e_k)=i_k \right]}{\Pro \left[ \omega^{(\epsilon)}_K(e_1)=i_1, \cdots , \omega^{(\epsilon)}_K(e_n)=i_n \right]} \Bigg) \, .
\end{align}
Using that $\omega_K(0)_{e_j}$ is independent of $\left\lbrace \forall k \in \lbrace 1, \cdots, n \rbrace \setminus \lbrace j \rbrace, \, \omega^{(\epsilon)}_K(e_k)=i_k \right\rbrace$, we obtain that, for all $j$ such that $i_j = 1$:
\begin{align*}
& \Pro \left[ \omega^{(\epsilon)}_K(e_1)=i_1, \, \cdots , \, \omega^{(\epsilon)}_K(e_n)=i_n \right]\\
& \geq \Pro \left[ \omega_K(0)_{e_j} = 1, \forall k \in \lbrace 1, \cdots, n \rbrace \setminus \lbrace j \rbrace, \, \omega^{(\epsilon)}_K(e_k)=i_k \right]\\
& = p \, \Pro \left[ \forall k \in \lbrace 1, \cdots, n \rbrace \setminus \lbrace j \rbrace, \, \omega^{(\epsilon)}_K(e_k)=i_k \right].
\end{align*}
Therefore, \eqref{sumovereedges} is smaller than or equal to:
\begin{align*}
& p + (1-p) \sum_{f \notin \lbrace e_1, \cdots , e_n \rbrace} \Big( 1-\exp \big( -\epsilon K(e_{n+1},f) \big) \Big) + (1-p) \sum_{j \in \lbrace 1, \cdots, n \rbrace \atop i_j = 1} \frac{1-\exp \left( -\epsilon K(e_{n+1},e_j) \right)}{p}\\
& \leq p + (1-p) \sum_{f \in E} \frac{1-\exp \left( -\epsilon K(e_{n+1},e_j) \right)}{p}  \leq p + \frac{1-p}{p} \, \sum_{f \in E} \epsilon \, K(e_{n+1},f) = p + \frac{1-p}{p}\epsilon
\end{align*}
(since $K$ is a symmetric transition matrix).
\end{proof}

\begin{proof}[Proof of Proposition~\ref{outsidepc}] We follow the ideas of~\cite{olle1997dynamical}, proof of Proposition~$1.1$. Let $p<p_c$. Note that: \textit{(a)} For any edge $e$, if there exists some time $t$ in $[0,\epsilon]$ such that $e$ is open at time $t$, then $e$ is open in $\omega_K^{(\epsilon)}$. \textit{(b)} The event $\lbrace \exists \text{ an infinite cluster} \rbrace$ is increasing. Therefore, if there exists an exceptional time between time $0$ and time $\epsilon$, then there is an infinite cluster in $\omega_K^{(\epsilon)}$. Furthermore, Lemma~\ref{increaslem} implies that if $\epsilon$ is sufficiently small, then there exists $p_0 < p_c$ such that the distribution of $\omega_K^{(\epsilon)}$ satisfies the hypotheses of Lemma~\ref{increas}. We deduce that, if $\epsilon$ is sufficiently small, then a.s. there is no exceptional time between times $0$ and $\epsilon$, which easily implies the result.

If $p > p_c$, then the same proof works with results analogous to Lemma~\ref{increas} and Lemma~\ref{increaslem} (with opposite inequalities).
\end{proof}
 
\begin{proof}[Proof of Proposition~\ref{dgeq11}] We follow the ideas of~\cite{olle1997dynamical}, proof of Theorem~$1.3$. Consider a graph $G$, a symmetric transition matrix $K$ on the edges of $G$ and a parameter $p \in (0,1)$ such that $\Pro_p \left[ \exists \text{ an infinite cluster } \right] = 0$. Let $\left( \omega_K(t) \right)_{t \geq 0}$ be a $K$-exclusion dynamical percolation of parameter~$p$. Let $v$ be a vertex of $G$ and write $N(v) \in \N \cup \lbrace + \infty \rbrace$ for the number of times $t \in [0,1]$ such that $\lbrace v \overset{\omega_K(t)}{\longleftrightarrow} \infty \rbrace$ holds. As explained in~\cite{olle1997dynamical} in the case of i.i.d. dynamical percolation, we can show that either $\E \left[ N(v) \right] = 0$ or $\E \left[ N(v) \right] = +\infty$. This is actually a consequence of general results about reversible Markov processes - see Lemma~$2.3$ of~\cite{peres1998number} - and these results are also true for our $K$-exclusion processes. (For more explanations and further references, see~\cite{peres1998number}; note that the fact that we consider symmetric matrices is important here since it implies that our exclusion processes are reversible.)

Now, take $d \geq 11$, let $K$ be a symmetric transition matrix on the edges of the Euclidean lattice $\Z^d$, and let $\left( \omega_K(t) \right)_{t \geq 0}$ be a $K$-exclusion dynamical percolation of parameter $p_c=p_c(d)$ (if $p \neq p_c$ then the result is a direct consequence of Proposition~\ref{outsidepc}). The above observation implies that, in order to prove Proposition~\ref{dgeq11}, it is sufficient to show that, for every $v \in \Z^d$, $\E \left[ N(v) \right] < +\infty$. That is the purpose of what follows.

It is known (see~\cite{hara1994mean} for $d \geq 19$ and the recent work~\cite{fitzner2015nearest} for the extension to $d \geq 11$) that there exists $C = C(d) < +\infty$ such that, for all $p \geq p_c=p_c(d)$:
\begin{equation}\label{betais1}
\Pro_p \left[ v \leftrightarrow \infty \right] \leq C \, (p-p_c) \, .
\end{equation}
For each $m \in \N_+$, write $N_m(v)$ for the number of intervals of the form $I_k^m=[k/m,(k+1)/m]$, $k \in \lbrace 0, \cdots, m-1 \rbrace$, such that there exists $t \in I_k^m$ for which $\lbrace v \overset{\omega_K(t)}{\longleftrightarrow} \infty \rbrace$ holds. Lemmas~\ref{increas} and~\ref{increaslem} (with $\epsilon = 1/m$) imply that:
\[
\Pro \left[ \exists t \in I_0^m, \,  v \overset{\omega_K(t)}{\longleftrightarrow} \infty \right] \leq \Pro_{p_c+C'/m} \left[ v \leftrightarrow \infty \right] \,  ,
\]
where $C'=(1-p_c)/p_c$. Using~\eqref{betais1} we obtain:
\[
\Pro \left[ \exists t \in I_0^m, \, v \overset{\omega_K(t)}{\longleftrightarrow} \infty \right] \leq CC'/m \, .
\]
Since our process is time-stationary, the above is also true for any $I_k^m$. Therefore, $\E \left[ N_m(v) \right] \leq m \, CC'/m = CC' < +\infty$ and by Fatou's lemma we are done since $N(v) \leq \underset{m \rightarrow +\infty}{\text{lim inf }} N_m(v)$.
\end{proof}

Kesten and Zhang~\cite{kesten1987strict} have proved that~\eqref{betais1} is not true when $d=2$. For site percolation on $\T$, it has even been proved in~\cite{smirnov2001critical} (see also~\cite{werner2007lectures}) that:
\begin{equation}\label{theta(p)and5/36}
\Pro_p \left[ v \leftrightarrow \infty \right] = (p-1/2)^{5/36+\petito{1}} \, ,
\end{equation}
where $\petito{1}$ goes to $0$ as $p \searrow 1/2$. If we follow the proof of Proposition~\ref{dgeq11} with~\eqref{theta(p)and5/36} instead of~\eqref{betais1}, we obtain the following: For any symmetric transition matrix $K$ on the sites of $\T$, any site $v \in \T$, and any $\delta >0$, there exists a constant $C=C(\delta) < + \infty$ such that:
\begin{equation}\label{hausupper}
\E \left[ N_m(v) \right] \leq C \, m.m^{\delta-5/36} = C \, m^{\delta+31/36} \, ,
\end{equation}
for the $K$-exclusion dynamical percolation of parameter $p_c=1/2$. The analogue of~\eqref{hausupper} for i.i.d. dynamical percolation is shown in~\cite{SS10} in the proof of their Theorem~$6.3$. Moreover, this is the only property that they use to prove that a.s. the Hausdorff dimension of the set of exceptional times is at most $31/36$. Therefore, we have the following result (and we refer to~\cite{SS10} for more details):

\begin{prop}\label{upperhausdorff}
Let $K$ be any symmetric transition matrix on the sites of $\T$ and consider a $K$-exclusion dynamical percolation of parameter $p=p_c=1/2$. Then, a.s. the Hausdorff dimension of the set of exceptional times is at most $31/36$.
\end{prop}

\section{Exceptional times at criticality: proofs of Theorem~\ref{tropbien} and Proposition~\ref{tropbienlog}}\label{sectionexcepcriticality}

We need the following lemma (which is a quantitative version of Lemma~$7.1$ in \cite{BGS}):

\begin{lem}\label{singular}
Let $(E,\nu)$ be a countable set endowed with a sub-probability measure $\nu \neq 0$. Furthermore, let $P$ be a symmetric sub-transition matrix on $E$. Let $F \subseteq E$ and define $\delta, \, \eta$ as follows: $\nu(F) = (1 - \delta)\nu(E)$; $\eta = \max_{x \in F} P(x,F)$. We have:
\[
\sum_{x,y \in E} \sqrt{\nu(x))}\sqrt{\nu(y)} P(x,y) \leq \nu(E)(\eta + 2\sqrt{\delta}) \, .
\]
\end{lem}
\begin{proof}
Since $P$ is symmetric, we have the following inequality:
\begin{multline*}
\sum_{x,y \in E} \sqrt{\nu(x)} \sqrt{\nu(y)} P(x,y) \leq \sum_{x \in F, y \in F} \sqrt{\nu(x)} \sqrt{\nu(y)} P(x,y)\\
+ 2 \sum_{x \in E, y \in E \setminus F} \sqrt{\nu(x)} \sqrt{\nu(y)} P(x,y) \, .
\end{multline*}
Write $A_1$ and $A_2$ for the two sums of the right-hand side of the above inequality. Let us show that $A_1 \leq \nu(E) \, \eta$ and $A_2 \leq \nu(E) \, \sqrt{\delta}$. We have (by using the Cauchy-Schwarz inequality for the measure $(x,y) \mapsto P(x,y)$):
\begin{eqnarray*}
A_1 & \leq & \sqrt{\sum_{x \in F, y \in F} \nu(x) P(x,y)} \sqrt{\sum_{x \in F, y \in F} \nu(y) P(x,y) } \\
& = & \sqrt{\sum_{x \in F} \nu(x) P(x,F)} \sqrt{\sum_{y \in F} \nu(y) P(y,F) } \text{ (since } P \text{ is symmetric)}\\
& \leq & \nu(F) \eta \leq \nu(E) \eta \, .
\end{eqnarray*}
Moreover, by using the Cauchy-Schwarz inequality twice (first for the sub-probability measures $P(.,y)$ then for the counting measure), we obtain:
\begin{eqnarray*}
A_2 & = & \sum_{y \in E \setminus F} \sqrt{\nu(y)} \; \sum_{x \in E} \sqrt{\nu(x)} P(x,y)\\
& \leq & \sum_{y \in E \setminus F} \sqrt{\nu(y)} \; \sqrt{\sum_{x \in E} \nu(x) P(x,y)} \\
& \leq & \sqrt{ \sum_{y \in E \setminus F} \nu(y) } \; \sqrt{ \sum_{y \in E \setminus F} \Big( \sum_{x \in E} \nu(x) P(x,y) \Big) }\\
& = & \sqrt{ \nu(E) \, \delta } \; \sqrt{ \sum_{x \in E} \nu(x) \Big( \sum_{y \in E \setminus F} P(x,y) \Big) } \, .
\end{eqnarray*}
We are done since $\sum_{x \in E} \nu(x) \Big( \sum_{y \in E \setminus F} P(x,y) \Big) \leq \sum_{x \in E} \nu(x) = \nu(E)$.
\end{proof}

Now, we use Lemma~\ref{singular}, Theorem~\ref{GPScool}, and Corollary~\ref{soimportantcor} in order to prove Theorem~\ref{tropbien}. Remember~\eqref{secondsecondmoment} and~\eqref{secondsecondmomenthausdorff}: we need to study the following quantities: 
\[
\sum_{S, S' \subseteq \mathcal{I}_R} \sqrt{\widehat{\Pro}_{f_R} \left[ \lbrace S \rbrace \right]}  \sqrt{\widehat{\Pro}_{f_R} \left[ \lbrace S' \rbrace \right]} \, K_t^\alpha(S,S') \, .
\]

\begin{proof}[Proof of Theorem~\ref{tropbien}]
Let $\beta > 1$ and $\alpha > 0$. For any $k \in \N$, let $E = E_k = E_k(R) = \left\lbrace S \subseteq \mathcal{I}_R \, : \, |S| \in [2^k,2^{k+1}-1] \right\rbrace$ and:
\begin{align*}
F = F_k = F_k(\beta,R) & = \left\lbrace S \in E_k \, : \, S \subseteq (-2^{k\beta},2^{k\beta})^2 \right\rbrace \\
& = \left\lbrace  S \subseteq \mathcal{I}_R \, : \, |S| \in [2^k,2^{k+1}-1]  \text{ and } S \subseteq (-2^{k\beta},2^{k\beta})^2 \right\rbrace.
\end{align*}
Let $\nu = \nu_k = \nu_k(R)$ be $\widehat{\Pro}_{f_R}$ restricted to $E_k$. Also, let $P = P_k^\alpha(t) = P_k^\alpha(t,R)$ be $K^\alpha_t$ restricted to $E_k$. Finally, let $\delta = \delta_k=\delta_k(\beta,R)$ and $\eta = \eta_k^\alpha(t)=\eta_k^\alpha(\beta,R,t)$ be as in Lemma~\ref{singular} (if $\nu_k=0$ for some $k$ then we let $\delta_k=\eta_k^\alpha(t)=0$).

Remember that, if $|S| \neq |S'|$, then $K_t^\alpha(S,S') = 0$. Thus:
\begin{multline*}
\sum_{S, S' \subseteq \mathcal{I}_R} \sqrt{\widehat{\Pro}_{f_R} \left[ \lbrace S \rbrace \right]}  \sqrt{\widehat{\Pro}_{f_R} \left[ \lbrace S' \rbrace \right]} \, K_t^\alpha(S,S')\\
= \alpha_1(R) + \sum_{k \in \N} \sum_{S, S' \in E_k} \sqrt{\nu_k(S)} \sqrt{\nu_k(S')} \, P_k^\alpha(t)(S,S') \, .
\end{multline*}
(The term $\alpha_1(R)$ is just the contribution of $S=S'=\emptyset$.) By applying Lemma~\ref{singular} for each $k$, we obtain:

\begin{equation}\label{whathelpstoprove}
\sum_{S, S' \subseteq \mathcal{I}_R} \sqrt{\widehat{\Pro}_{f_R} \left[ \lbrace S \rbrace \right]}  \sqrt{\widehat{\Pro}_{f_R} \left[ \lbrace S' \rbrace \right]} \, K_t^\alpha(S,S')
\leq  \alpha_1(R) + \sum_{k \in \N} \widehat{\Pro}_{f_R} \left[ E_k \right] \left( \eta_k^\alpha(t) + 2 \sqrt{\delta_k} \right) \, .
\end{equation}
\ni
Theorem~\ref{GPScool} gives good estimates for $\widehat{\Pro}_{f_R} \left[ E_k \right]$. It thus remains to estimate $\delta_k$ and $\eta_k^\alpha(t)$.

\paragraph{An estimate for \textbf{$\delta_k$}.} In this paragraph, we assume that $\beta$ satisfies the hypothesis of Corollary~\ref{soimportantcor}. We have the following estimate on $\delta_k$ which is a direct consequence of Corollary~\ref{soimportantcor} since $\widehat{\Pro}_{f_R} \left[ E_k \right] \delta_k = \widehat{\Pro}_{f_R} \left[ |S| \in [2^{k},2^{k+1}-1], \, S \nsubseteq (-2^{k\beta},2^{k\beta})^2 \right]$.

\begin{lem}\label{lemmadirectconsequence}
There exists a constant $C=C(\beta) < +\infty$ such that, for all $k \in \N$:
\[
\widehat{\Pro}_{f_R} [ E_k ] \, \delta_k \leq C \frac{\alpha_1(R)}{\alpha_1(2^{k\beta})} \left( \frac{2^{k\beta}}{\rho(2^k)} \right)^{1-\epsilon} \alpha_4( \rho(2^k), 2^{k\beta} ) \, ,
\]
where $\epsilon$ is the constant of Theorem~\ref{soimportant}.
\end{lem}

Thanks to Theorem~\ref{GPScool}, Lemma~\ref{lemmadirectconsequence}, and the quasi-multiplicativity property, we have:
\begin{align}\label{e.dich}
& \sum_{k \in \N} \widehat{\Pro}_{f_R} [ E_k ] \sqrt{\delta_k} \nn\\
& = \sum_{k \in \N} \sqrt{\widehat{\Pro}_{f_R} [ E_k ]} \sqrt{\widehat{\Pro}_{f_R} [ E_k ] \delta_k} \nn\\
&\leq \sum_{k \in \N} \sqrt{\widehat{\Pro}_{f_R} \left[ |S| < 2^{k+1} \right]} \sqrt{\widehat{\Pro}_{f_R} [ E_k ] \delta_k} \nn \\
& \leq \grandO{1} \sum_{k \in \N} \sqrt{\frac{\alpha_1(R)}{\alpha_1(\rho(2^k))}} \; \sqrt{\frac{\alpha_1(R)}{\alpha_1(2^{k\beta})} \left( \frac{2^{k\beta}}{\rho(2^k)} \right)^{1-\epsilon} \alpha_4( \rho(2^k), 2^{k\beta} )} \nn \\
& \leq \grandO{1} \alpha_1(R) \sum_{k \in \N} \sqrt{\frac{1}{\alpha_1(\rho(2^k)) \, \rho(2^k)^{1-\epsilon} \, \alpha_4(\rho(2^k))} \frac{2^{k\beta(1-\epsilon)} \, \alpha_4(2^{k\beta})}{\alpha_1(2^{k\beta})}} \, .
\end{align}
Thanks to~\eqref{poly} we know that (for $r \geq 1$):
\begin{equation}\label{premierestimatedelta}
\frac{1}{\alpha_1(\rho(r)) \, \rho(r)^{1-\epsilon} \, \alpha_4(\rho(r))} \leq \grandO{1} r^{\grandO{1}} \, .
\end{equation}
Thanks to the right-hand inequality of~\eqref{alpha4} (that is stated for $(r,R)$ but that we use for $(1,r)$) and thanks to the FKG-inequality (which implies that $\alpha_2(r) \leq \grandO{1} \alpha_1(r)^2$), we have:
\begin{equation}\label{deuxiemeestimatedelta}
\frac{r^{1-\epsilon}\alpha_4(r)}{\alpha_1(r)} \leq \grandO{1} \frac{r^{1-\epsilon} \alpha_4(r)}{\sqrt{\alpha_2(r)}} \leq \grandO{1} r^{-\epsilon} \, .
\end{equation}
Currently, there is no better estimate than~\eqref{deuxiemeestimatedelta} for bond percolation on $\Z^2$: for this model, it is only known that $\frac{r \, \alpha_4(r)}{\alpha_1(r)} \leq \grandO{1}$ (in particular it is not proved that $\alpha_2(r) \leq \grandO{1} r^{-h} \, \alpha_1(r)^2$ for some fixed $h>0$, see~$(9.2)$ in~\cite{SS10} for more about this inequality). However, for site percolation on $\T$, it is known that $\frac{r \, \alpha_4(r)}{\alpha_1(r)} = r^{-7/48 + \petito{1}}$.
\medskip

From~\eqref{premierestimatedelta} and~\eqref{deuxiemeestimatedelta}, we deduce that there exist some constants $\beta_0 < +\infty$ and $c_3>0$ such that, for all $\beta' \geq \beta_0$ (and for all $r \geq 1$), we have:
\begin{equation}\label{beta0}
\frac{1}{\alpha_1(\rho(r)) \, \rho(r) \, \alpha_4(\rho(r))} \frac{r^{\beta'(1-\epsilon)} \, \alpha_4(r^{\beta'})}{\alpha_1(r^{\beta'})} \leq \frac{1}{c_3} r^{-c_3} \, .
\end{equation}

Hence, we have the following: If $\beta$ satisfies the hypothesis of Corollary~\ref{soimportantcor} and is larger than or equal to $\beta_0$, then:
\begin{equation}\label{estimatedelta}
\sum_{k \in \N} \widehat{\Pro}_{f_R} [ E_k ] \sqrt{\delta_k} \leq \grandO{1} \alpha_1(R) \sum_{k \in \N} \sqrt{ \frac{1}{c_3} 2^{-kc_3} } \leq \grandO{1} \alpha_1(R) \, .
\end{equation}

\begin{remark}\label{r.epscrucial}
We see from the proof of~\eqref{estimatedelta} (see in particular~\eqref{deuxiemeestimatedelta} and the small paragraph below it) that the exponent $\eps$ in Theorem~\ref{soimportant} is crucial for our proof to work. Notice in particular the nice decoupling of scales in~\eqref{e.dich} (under the square root). Even if the exponent $\eps$ is very small, one can tune $\beta$ to be large enough so that the right-hand side in~\eqref{e.dich} wins against the left-hand side. This is the main constraint which will prevent us from obtaining exceptional times for larger values of $\alpha$ (see Figure~\ref{f.main} for the range of $\alpha$ we manage to cover with these estimates in the case of site percolation on $\T$).
\end{remark}

\paragraph{An estimate for \textbf{$\eta^\alpha_k(t)$}.} In this paragraph, we assume that $1-\alpha \, \beta > 0$. We first prove the following lemma:

\begin{lem}\label{lemmaeta}
There exists $c_1=c_1(\alpha,\beta)>0$ such that, for all $t \in [0,1]$ and all $k \in \N$:
\[
\eta_k^\alpha(t) \leq \frac{1}{c_1} \exp \left( -c_1 \, t \, 2^{k(1-\alpha\beta)} \right) \, .
\]
\end{lem}

\begin{proof}
Let $S$ be a finite subset of $\mathcal{I}$, see $S$ as a set of $|S|$ particles, and construct an interacting particles system as follows: Associate to each particle of $S$ an exponential clock of parameter $1$, independent of the other clocks. If the clock of a particle rings and if its current location is $x \in \cal I$, then the particle attempts to jump to $y \in \cal I$ with probability $K^\alpha(x,y)$. If there is another particle at $y$, then the particle stays at $x$, while if there is no other particle at $y$ then the particle jumps to $y$. This way, we obtain for each $s \geq 0$ a random set $\widetilde{\pi}_s(S) \subseteq \cal I$ of $|S|$ particles. It is not difficult to see that, while each particle do not evolve exactly like in our Definition~\ref{d.exclusion}, the \textbf{whole set} of particles does evolve like in this definition. More precisely:
\[
(\widetilde{\pi}_s(S))_{s \geq 0} \overset{(d)}{=} (\pi_s(S))_{s \geq 0} \, ,
\]
where $\pi$ is defined in Appendix~\ref{a.graphical}.

In particular, Lemma~\ref{lemmaeta} is equivalent to the following statement: There exists $c_1=c_1(\alpha,\beta)>0$ such that, for all $t \in [0,1]$ and all $k \in \N$:
\begin{equation}\label{e.reecriture_lemmaeta}
\max_{S \in F_k} \Pro \left[ \widetilde{\pi}_t(S) \in F_k \right] \leq \frac{1}{c_1} \exp \left( -c_1 \, t \, 2^{k(1-\alpha\beta)} \right) \, .
\end{equation}
Let us prove~\eqref{e.reecriture_lemmaeta}. Fix $S \in F_k$ and let $U \subseteq S$ be the random subset of all the particles whose clock has rung exactly one time between time $0$ and time $t$, which happens with probability $\asymp t$ for each particle (remember that $t \in [0,1]$). By independence and by classical estimates on Binomial distributions, we obtain that there exists $c_2 > 0$ such that, with probability al least $\frac{1}{c_2} e^{-c_2t|S|}$, the size of $U$ is at least $c_2t|S|$. Write $U = \{ u_1, \cdots, u_{|U|} \}$, where the particles $u_i$ are indexed so that the clock of $u_i$ has rung before the clock of $u_{i+1}$. Also, write $\tau_i < t$ for the first time the clock of $u_i$ has rung, and write $V_i$ for the location to which the particle initially located at $u_i$ has attempted to jump at time $\tau_i$ (remember that $V_i$ follows the probability law $K^\alpha(u_i,\cdot)$).
\medskip

Now, \textbf{condition on $U$ and on $\tau_1, \cdots, \tau_{|U|}$}, and write $\widetilde{\Pro}$ for the conditional probability measure. Let $\mathcal{F}_i$ denote the $\sigma$-algebra generated by what happens strictly before time $\tau_i$. Also, write $S_i \subseteq \left( (-2^{k\beta},2^{k\beta})^2 \right)^c$ for the set of particles that are outside of $(-2^{k\beta},2^{k\beta})^2$ at time $\tau_i^-$. Note that $S_i$ and $V_{i-1}$ are measurable with respect to $\mathcal{F}_i$ while $V_i$ is independent of $\mathcal{F}_i$. Let us estimate $\widetilde{\Pro} \left[ \widetilde{\pi}_t(S) \in F_k \right]$. To do so, note that $\{ \widetilde{\pi}_t(S) \in F_k \}$ holds if there exists $u_i \in U$ such that the particle initially located at $u_i$ has jumped outside $(-2^{k\beta},2^{k\beta})^2$ at time $\tau_i$ (indeed, since the clock of $u_i$ has rung only one time before time $t$, the particle cannot go back to $(-2^{k\beta},2^{k\beta})^2$). Moreover, the particle initially located at $u_i$ has jumped outside $(-2^{k\beta},2^{k\beta})^2$ at time $\tau_i$ if and only if $V_i \notin (-2^{k\beta},2^{k\beta})^2 \cup S_i$. Let us estimate the quantity $\widetilde{\Pro} \left[ V_i \notin (-2^{k\beta},2^{k\beta})^2 \cup S_i \cond \mathcal{F}_i \right]$. To this purpose, observe that there exists a constant $c_3=c_3(\alpha) > 0$ such that, for every $y \in A_k^\beta := [-2^{k\beta+10},2^{k\beta+10}]^2 \setminus (-2^{k\beta},2^{k\beta})^2$ and every $x \in S$, we have:
\[
K^\alpha(x,y) \geq \frac{c_3}{2^{k\beta(2+\alpha)}} \, .
\]
Observe also that $|\mathcal{I} \cap A_k^\beta| \geq 2 \times 2^{2k\beta} \geq 2^{2k\beta} + |S|$ for every $S \in F_k$. Hence:
\begin{eqnarray*}
\widetilde{\Pro} \left[ V_i \notin (-2^{k\beta},2^{k\beta})^2 \cup S_i \cond \mathcal{F}_i \right] & \geq & \left( |\mathcal{I} \cap A_k^\beta| - |S_i| \right) \times \frac{c_3}{2^{k\beta(2+\alpha)}}\\
& \geq & \left( |\mathcal{I} \cap A_k^\beta| - |S| \right) \times \frac{c_3}{2^{k\beta(2+\alpha)}}\\
& \geq & \frac{c_3}{2^{k\alpha\beta}} \, .
\end{eqnarray*}
We obtain that $\widetilde{\Pro} \left[ \widetilde{\pi}_t(S) \in F_k \right]$ is at most:
\begin{align*}
& \widetilde{\Pro} \left[ \forall i \in \{ 1, \cdots, |U| \}, \, V_i \in (-2^{k\beta},2^{k\beta})^2 \cup S_i \right]\\
& = \widetilde{\E} \left[ \widetilde{\Pro} \left[ \forall i \in \{ 1, \cdots, |U| \}, \, V_i \in (-2^{k\beta},2^{k\beta})^2 \cup S_i \cond \mathcal{F}_{|U|} \right] \right]\\
& = \widetilde{\E} \left[ \un_{ \left\lbrace  \forall i \in \{ 1, \cdots, |U|-1 \}, \, V_i \in (-2^{k\beta},2^{k\beta})^2 \cup S_i \right\rbrace } \widetilde{\Pro} \left[ V_{|U|} \in (-2^{k\beta},2^{k\beta})^2 \cup S_{|U|} \cond \mathcal{F}_{|U|} \right] \right]\\
& \leq \widetilde{\Pro} \left[ \forall i \in \{ 1, \cdots, |U|-1 \}, \, V_i \in (-2^{k\beta},2^{k\beta})^2 \cup S_i \right] \left( 1 - \frac{c_3}{2^{k\alpha\beta}} \right)\\
& \leq \cdots\\
& \leq \left( 1-\frac{c_3}{2^{k\alpha\beta}} \right)^{|U|} \, .
\end{align*}
Finally, we have:
\[
\Pro \left[ \widetilde{\pi}_t(S) \in F_k \right] \leq \Pro \left[ U < c_2t |S| \right] + (1-\frac{c_3}{2^{k\alpha\beta}})^{c_2t |S|} \, ,
\]
where $c_2$ was chosen so that $\Pro \left[ U < c_2t |S| \right] \leq \frac{1}{c_2}e^{-c_2t|S|}$. Since $2^k \leq |S|$ we have:
\[
\Pro \left[ \widetilde{\pi}_t(S) \in F_k \right] \leq \frac{1}{c_2}e^{-c_2t2^k} +  (1-\frac{c_3}{2^{k\alpha\beta}})^{c_2t 2^k} \, ,
\]
which implies~\eqref{e.reecriture_lemmaeta}.
\end{proof}

We now combine Lemma~\ref{lemmaeta} and Theorem~\ref{GPScool}. First, note that the left-hand inequality of~\eqref{alpha4} (that is stated for $(r,R)$ but that we use for $(1,\rho(l))$) implies that:
\[
\frac{1}{\alpha_1(\rho(l))} \leq \frac{1}{c_2} \, \left( \rho(l)^2 \, \alpha_4(\rho(l)) \right) \rho(l)^{-c_2} = \frac{1}{c_2} \, l \, \rho(l)^{-c_2} \, ,
\]
for some $c_2 > 0$ (we have also used that $\rho(l)^2 \alpha_4(\rho(l)) = l$, which is a simple consequence of the definition of $\rho$). Note also that $\rho(l) \geq \sqrt{l}$ (since $r^2 \alpha_4(r) \leq r^2$). Therefore, by using Theorem~\ref{GPScool} we obtain that there exists some $\epsilon_1 > 0$ such that, for all $l \in \N_+$:
\begin{equation}\label{GPSrewritten}
\widehat{\Pro}_{f_R} \left[ |S| < l \right] \leq \frac{1}{\epsilon_1} \alpha_1(R) (l/2)^{1-\epsilon_1} \, ,
\end{equation}
where the constant $2$ is included to simplify the calculations below. We fix such an $\epsilon_1$ for the rest of the proof. We can (and do) assume that $\epsilon_1 \in (0,1)$.
Remember that we have assumed that $1 - \alpha \, \beta > 0$. We have:
\begin{align*}
& \sum_{k \in \N} \widehat{\Pro}_{f_R} \left[ E_k \right] \, \eta_k^\alpha(t) \leq \, \frac{1}{c_1} \sum_{k \in \N} \widehat{\Pro}_{f_R} [ E_k ] \exp \left( - c_1 \, t \, 2^{k(1-\alpha\beta)} \right)\\
& = \frac{1}{c_1} \sum_{k' \in \N} \sum_{j \in \N \, : \, \left( \frac{2^{k'}-1}{c_1t} \right)^{\frac{1}{1-\alpha\beta}} \leq 2^j < \left( \frac{2^{k'+1}-1}{c_1t} \right)^{\frac{1}{1-\alpha\beta}}} \widehat{\Pro}_{f_R} [ E_j ] \exp \left(- c_1 \, t \, 2^{j(1-\alpha\beta)} \right)\\
& \leq \frac{1}{c_1} \sum_{k' \in \N} \widehat{\Pro}_{f_R} \left[ |S| < 2 \left( \frac{2^{k'+1}-1}{c_1 t} \right)^{1 / (1-\alpha\beta)} \right] \exp \left( -(2^{k'}-1) \right) \, .
\end{align*}
We now use~\eqref{GPSrewritten}. It implies that the above is at most:
\[
\frac{1}{\epsilon_1c_1} \alpha_1(R) \sum_{k' \in \N} \left( \frac{2^{k'+1}-1}{c_1 t} \right)^{(1-\epsilon_1) / (1-\alpha\beta)} \exp \left( -(2^{k'}-1) \right) \, .
\]
Hence, there exists a constant $C'=C'(\alpha,\beta)<+\infty$ such that, if $1-\alpha \, \beta > 0$, then:
\begin{equation}\label{estimateeta}
\sum_{k \in \N} \widehat{\Pro}_{f_R} \left[ E_k \right] \, \eta_k^\alpha(t) \leq C' \, \alpha_1(R) \left( \frac{1}{t} \right)^{(1-\epsilon_1) / (1-\alpha\beta)} \, .
\end{equation}

\begin{remark}
Notice from the above inequalities that it was important to obtain a relatively sharp upper-bound on $\eta_k^\alpha(t)$ in Lemma \ref{lemmaeta} in order to optimise the exponent  of $(1/t)$ in~\eqref{estimateeta}. 
\end{remark}

\paragraph{End of the proof of existence of exceptional times.}

We are now in shape to prove that, if $\alpha$ is sufficiently small, then a.s. there are exceptional times for the $K^\alpha$-exclusion dynamical percolation of parameter $p=1/2$. Thanks to~\eqref{secondsecondmoment}, we know that it is sufficient to prove that there exists a constant $C=C(\alpha)<+\infty$ such that, for all $R \geq 1$:
\begin{equation}\label{whatwewanttoprove}
\int_0^1 \sum_{S,S' \subseteq \mathcal{I}_R} \sqrt{\widehat{\Pro}_{f_R}\left[ \lbrace S \rbrace \right]} \sqrt{\widehat{\Pro}_{f_R}\left[ \lbrace S' \rbrace \right]} \, K_t^\alpha(S,S') \,  dt \leq C \, \alpha_1(R) \, .
\end{equation}
Fix a constant $\beta_0$ that satisfies~\eqref{beta0} and choose $\beta \geq \beta_0$ that satisfies the hypothesis of Corollary~\ref{soimportantcor}. Next, choose $\alpha > 0$ sufficiently small so that $0<(1-\epsilon_1)/(1-\alpha\beta)<1$ (so, in particular, $1-\alpha \, \beta > 0$). Now, note that~\eqref{whathelpstoprove} implies that - in order to prove~\eqref{whatwewanttoprove} - it is sufficient to show the following inequalities:
\begin{equation}\label{whatwewanttoprove2}
\int_0^1 \sum_{k \in \N} \widehat{\Pro}_{f_R} [ E_k ] \sqrt{\delta_k} \, dt \leq \grandO{1} \, \alpha_1(R) \, ,
\end{equation}
and:
\begin{equation}\label{whatwewanttoprove1}
\int_0^1 \sum_{k \in \N} \widehat{\Pro}_{f_R} \left[ E_k \right] \, \eta_k^\alpha(t) \, dt \leq \grandO{1} \, \alpha_1(R)
\end{equation}
(where the constants in the $\grandO{1}$'s may depend on $\alpha$). The inequality~\eqref{whatwewanttoprove2} is a direct consequence of~\eqref{estimatedelta} since $\beta$ satisfies the hypothesis of Corollary~\ref{soimportantcor} and $\beta \geq \beta_0$ (note that, in~\eqref{whatwewanttoprove2}, the quantities do not depend on $\alpha$). On the other hand, the inequality~\eqref{whatwewanttoprove1} is a direct consequence of~\eqref{estimateeta}. Indeed, since $(1-\epsilon_1)/(1-\alpha\beta)<1$, we have:
\[
\int_0^1 \left( \frac{1}{t} \right)^{(1-\epsilon_1)/(1-\alpha\beta)} \, dt < +\infty \, .
\]

\paragraph{The constant \textbf{$\alpha_0 = 217/816$}.}

In this paragraph, we prove the following quantitative result: For site percolation on $\T$, there exist exceptional times for any $\alpha < \alpha_0 := 217/816$. We only work with site percolation on $\T$ and we use the computations of the arm-exponents (see Subsection~\ref{sectionarm}). Thanks to these computations (which imply in particular that $\rho(l) = l^{4/3+\petito{1}}$) we can say that:

\begin{enumerate}
\item The hypothesis on $\beta$ in Corollary~\ref{soimportantcor} is satisfied for any $\beta > 4/3$.
\item Any $\beta_0 > (17/36) \cdot (48/7) = 68/21$ satisfies~\eqref{beta0} (with $c_3 = c_3(\beta_0)$). Let us detail a little this result: it comes from the two following calculations:
\begin{eqnarray*}
\frac{1}{\alpha_1(\rho(r)) \, \rho(r) \,\alpha_4(\rho(r))} & = & r^{(5/48) \cdot (4/3) - 4/3 + (5/4) \cdot (4/3) + \petito{1}}\\
& = & r^{17/36 + \petito{1}} \, ,
\end{eqnarray*}
and:
\begin{eqnarray}
\frac{r^{1-\epsilon}\alpha_4(r)}{\alpha_1(r)} & \leq & \frac{r \, \alpha_4(r)}{\alpha_1(r)} \label{wherewehavelostepsilon}\\
& = & r^{1-5/4+5/48+\petito{1}} \nonumber \\
& = & r^{-7/48+\petito{1}} \, . \nonumber
\end{eqnarray}
\item The inequality~\eqref{GPSrewritten} can be replaced by the following quantitative estimate: For all $l \in \N_+$:
\[
\widehat{\Pro}_{f_R} \left[ |S| < l \right] \leq \frac{1}{\epsilon_1} \alpha_1(R) (l/2)^{5/36+\petito{1}} \, ,
\]
where $\petito{1} \searrow 0$ as $l \rightarrow +\infty$. This implies that, in~\eqref{estimateeta}, the exponent $1-\epsilon_1$ can be replaced by any $\zeta > 5/36$: Let $\zeta > 5/36$ and assume that $1-\alpha \, \beta > 0$. Then, there exists a constant $C'=C'(\alpha,\beta,\zeta)$ such that:
\begin{equation}\label{estimateetabis}
\sum_{k \in \N} \widehat{\Pro}_{f_R} \left[ E_k \right] \eta^\alpha_k(t) \leq C' \, \alpha_1(R) \, \left( \frac{1}{t} \right)^{\frac{\zeta}{1-\alpha \beta}} \, .
\end{equation}
\end{enumerate}

If we use items~$1$ and~$2$ above, we deduce that~\eqref{estimatedelta} is true for any $\beta > 68/21 \vee 4/3 = 68/21$. Hence,~\eqref{whatwewanttoprove2} is true for any $\beta > 68/21$ (remember that the quantities in~\eqref{estimatedelta} and~\eqref{whatwewanttoprove2} do not depend on $\alpha$). Fix such a $\beta$ and let $\zeta > 5/36$. Let $\alpha > 0$ such that $0 < \zeta/(1-\alpha \beta) < 1$. By using~\eqref{estimateetabis}, we obtain that~\eqref{whatwewanttoprove1} is true with these choices of $\alpha$ and $\beta$. Finally, for any $\beta > 68/21$ and any $\zeta > 5/36$, we have the following: Let $\alpha > 0$ such that $0 < \zeta/(1-\alpha \beta) < 1$. Then, there exist exceptional times for the $K^\alpha$-exclusion process of parameter $p=1/2$. As such, we have obtained that our result of existence of exceptional times holds for any $\alpha > 0$ such that $(5/36)/(1-68\alpha/21)<1$ i.e. for any $\alpha \in (0,\alpha_0)$ with $\alpha_0 = 217/816$.

Actually, we can do a little better:\footnote{This observation can be skipped at first reading.} Note that, in~\eqref{wherewehavelostepsilon}, we have used the rough estimate $r^{-\epsilon} \leq 1$. Let us take into account the exponent $\epsilon$.
Thanks to Remark~\ref{aquantitaiveepsilon}, we know that Theorem~\ref{soimportant} holds with any $\epsilon < (-5/4+\zeta_4^{|\half}) \wedge 1/4$, where $\zeta_4^{|\half}$ is defined in Proposition~\ref{exponenthalfplane}. Therefore, any $\beta_0$ larger than
\[
\frac{17}{36} \cdot \left( \frac{7}{48} + \left( (- \frac 5 4+\zeta_4^{|\half}) \wedge \frac 1 4 \right) \right)^{-1} < \frac{68}{21}
\]
satisfies~\eqref{beta0} if $\epsilon$ is well chosen. Let:
\[
\theta := \left( \frac{17}{36} \cdot \left( \frac{7}{48} + \left( (- \frac 5 4+\zeta_4^{|\half}) \wedge \frac 1 4 \right) \right)^{-1} \right) \vee (4/3) \, .
\]
Finally, our result of existence of exceptional times for site percolation on $\T$ holds for any $\alpha$ such that:
\[
\frac{5}{36} \cdot \frac{1}{1 - \alpha \theta} < 1 \, ,
\]
i.e. for any $\alpha$ less than:
\begin{align}\label{e.alphamax}
\frac{31}{36 \theta} > \frac{217}{816} \, .
\end{align}

\paragraph{The Hausdorff dimension of the set of exceptional times.}

In this paragraph, we also only work with site percolation on $\T$ and we prove lower-bounds estimates on the Hausdorff dimension of the set of exceptional times (for upper-bounds, see Proposition~\ref{upperhausdorff}). Let us introduce the function:
\[
d(\alpha):= 1 - \frac 5 {36}  \, \frac 1 {1 - \frac{68}{21} \alpha } \, ,
\]
which was plotted in Figure \ref{f.main}. Also, let:
\[
d(\alpha,\beta,\zeta):= 1 - \frac \zeta {1 -  \alpha \, \beta } \, .
\]
Thanks to~\eqref{secondsecondmomenthausdorff}, we know that, in order to prove that the Hausdorff dimension of set of exceptional times for the $K^\alpha$-exclusion process is at least $d(\alpha)$, it is sufficient to prove the following: Let $\zeta > 5/36$ and $\beta > 68/21$. Also, let $\alpha > 0$ be such that $1-d(\alpha,\beta,\zeta) > 0$, and let $\gamma< d(\alpha,\beta,\zeta)$. Then, there exists a constant $C = C(\alpha,\beta,\zeta,\gamma) < +\infty$ such that, for all $R \geq 1$:
\[
\int_0^1 \left( \frac{1}{t} \right)^{\gamma} \sum_{S,S' \subseteq \mathcal{I}_R} \sqrt{\widehat{\Pro}_{f_R}\left[ \lbrace S \rbrace \right]} \sqrt{\widehat{\Pro}_{f_R}\left[ \lbrace S' \rbrace \right]} \, K_t^\alpha(S,S') \, dt \leq C \, \alpha_1(R) \, .
\]
The inequality~\eqref{whathelpstoprove} implies that it is actually sufficient to prove that, for any $\alpha$, $\beta$, $\zeta$ and $\gamma$ as above, we have:
\begin{equation}\label{hausdorfflast2}
\int_0^1 \left( \frac{1}{t} \right)^{\gamma} \sum_{k \in \N} \widehat{\Pro}_{f_R} \left[ E_k \right] \sqrt{\delta_k} \, dt \leq \grandO{1} \, \alpha_1(R) \, ,
\end{equation}
and:
\begin{equation}\label{hausdorfflast1}
\int_0^1 \left( \frac{1}{t} \right)^{\gamma} \sum_{k \in \N} \widehat{\Pro}_{f_R} \left[ E_k \right] \eta^\alpha_k(t)  \, dt \leq \grandO{1} \, \alpha_1(R) \, ,
\end{equation}
where the constants in the $\grandO{1}$'s may depend on $\alpha$, $\beta$, $\zeta$ and $\gamma$. Since $\widehat{\Pro}_{f_R} \left[ E_k \right] \sqrt{\delta_k}$ does not depend on $t$ and since $\gamma < 1$,~\eqref{hausdorfflast2} is actually equivalent to:
\begin{equation}\label{hausdorfflast2bis}
\sum_{k \in \N} \widehat{\Pro}_{f_R} \left[ E_k \right] \sqrt{\delta_k} \leq \grandO{1} \, \alpha_1(R) \, .
\end{equation}
The fact that~\eqref{hausdorfflast2bis} holds when $\beta > 68/21$ has been proved in the paragraph about the constant $\alpha_0 = 217/816$ (this was a consequence of items~$1$ and $2$ of this last paragraph). Now, let us concentrate on~\eqref{hausdorfflast1}. If we use~\eqref{estimateetabis}, we obtain that:
\[
\int_0^1 \left( \frac{1}{t} \right)^{\gamma} \sum_{k \in \N} \widehat{\Pro}_{f_R} \left[ E_k \right] \eta^\alpha_k(t)  \, dt \leq \grandO{1} \, \alpha_1(R) \, \int_0^1 \left( \frac{1}{t} \right)^{\gamma+\frac{\zeta}{1-\alpha \beta}} \, dt \, ,
\]
and we are done since $\gamma+\frac{\zeta}{1-\alpha \beta} = \gamma + 1-d(\alpha,\beta,\zeta) < 1$.
\end{proof}

\begin{rem}
Actually, by taking into account the $\epsilon$ of Theorem~\ref{soimportant}, we can go slightly above the quantity $d(\alpha)$: we can prove that the Hausdorff dimension of the set of exceptional times belongs to $(d(\alpha),31/36]$.
\end{rem}

\bigskip

We now use Lemma~\ref{singular}, Theorem~\ref{GPScool} and Proposition~\ref{soimportantlog} in order to prove Proposition~\ref{tropbienlog}.
\begin{proof}[Proof of Proposition~\ref{tropbienlog}] Consider $a>0$. The steps are exactly the same as in the proof of Theorem~\ref{tropbien}, with the following analogous definitions (where $\epsilon_0>0$ is such that $1-\epsilon_0(1+a)>0$): $E_k = \lbrace S \subseteq \mathcal{I}_R \, : \, |S| = k \rbrace$, $F_k = \lbrace S \in E_k \, : \, S \subseteq (-\exp((k+1)^{\epsilon_0}),\exp((k+1)^{\epsilon_0}))^2 \rbrace$, $\nu_k = \widehat{\Pro}_{f_R}$ restricted to $E_k$, $P^a_k(t) = K^a_{\log,t}$ restricted to $E_k$. We write $\delta_k$ and $\eta^a_k(t)$ as in Lemma~\ref{singular}. If we follow the proof of Lemma~\ref{lemmaeta}, we obtain that there exists $c_1 = c_1(a,\epsilon_0) > 0$ such that:
\[
\eta_k^a(t) \leq \frac{1}{c_1} \exp \left( - c_1 t k^{1-\epsilon_0(1+a)} \right).
\]
Moreover, if we follow the proof of~\eqref{estimateeta}, we obtain that there exists a constant $C'=C'(a,\epsilon_0)<+\infty$ such that:
\[
\sum_{k \in \N} \widehat{\Pro}_{f_R} \left[ E_k \right] \, \eta^a_k(t) \leq C'\alpha_1(R)\left( \frac{1}{t} \right)^{(1-\epsilon_1)/(1-\epsilon_0(1+a))} \, ,
\]
where $\epsilon_1$ is the constant of~\eqref{GPSrewritten}.

To estimate $\delta_k$, we use Proposition~\ref{soimportantlog} and, since $\exp(k^{\epsilon_0}/C)$ is super-polynomial, we easily obtain that:
\[
\sum_{k \in \N} \widehat{\Pro}_{f_R} \left[ E_k \right] \, \sqrt{\delta_k} \leq C'' \alpha_1(R) \, ,
\]
for some $C''=C''(a,\epsilon_0)<+\infty$.

Finally note that: \textit{(a)} for the two models, $\epsilon_0$ can be chosen as small as we want and: \textit{(b)} for site percolation on $\T$, we can replace the exponent $1-\epsilon_1$ by any $\zeta > 5/36$. We conclude exactly as in the proof of Theorem~\ref{tropbien}.
\end{proof}

\section{Clustering effect for the spectral sample}\label{sectionspectralnotlocalized}

In this section, we prove Theorem~\ref{soimportant} and Proposition~\ref{soimportantlog}. We are mostly inspired by the proof of the upper-bound part of Theorem~\ref{GPScool} as found in~\cite{GPS}. This proof is divided into three steps. The first step shows that there exists a constant $\theta < + \infty$ such that the following holds: For every $S \subseteq \mathcal{I}_R$, let $S_r$ be the set of the $r \times r$ squares of the grid $r\Z^2$ which intersect $S$. Then, for all $k \in \N_+$ we have:
\begin{equation}\label{verysmallspectrum}
\widehat{\Pro}_{f_R} \left[ |S_r|=k \right] \leq g(k) \frac{\alpha_1(R)}{\alpha_1(r)} \, ,
\end{equation}
where $g(k)=2^{\theta \log_2^2(k+2)}$. This is Proposition~$4.7$ in~\cite{GPS} (with $l=1$), see Subsection~\ref{sectionlittlepart} for a little discussion about the two other steps.

As mentioned in Subsection~\ref{ss.outline}, the proof of Proposition~\ref{soimportantlog} is easier than the proof of Theorem~\ref{soimportant}. More precisely, Proposition~\ref{soimportantlog} is a direct consequence of the proof of~\eqref{verysmallspectrum} as found in Section~4 of~\cite{GPS}, whereas to prove Theorem~\ref{soimportant} we will need to \textit{(a)} prove an analogue of~\eqref{verysmallspectrum} and \textit{(b)} use the two other steps of the proof of the upper-bound part of Theorem~\ref{GPScool} identically to~\cite{GPS}. See Subsection~\ref{ss.outline} for a discussion (and a list) of the differences with the proof in~\cite{GPS}.

Remember that $\widehat{\Q}_{f_R} = \alpha_1(R) \, \widehat{\Pro}_{f_R}$. For some proofs of Section~\ref{sectionspectralnotlocalized}, it will be more convenient to deal with $\widehat{\Q}_{f_R}$ than with $\widehat{\Pro}_{f_R}$.

\subsection{The proof of Proposition~\ref{soimportantlog}}\label{sectionlog}


Proposition~\ref{soimportantlog} is a simple consequence of Corollary~\ref{soimportantcor}. As explained above, this is also a direct consequence of the proofs of Section~$4$ of~\cite{GPS}. Let us say a little more about this. Consider $\epsilon_0 > 0$. What we want to prove is the existence of some $C = C(\epsilon_0) < +\infty$ such that, for all $k \in \N_+$:
\[
\widehat{\Q}_{f_R} \left[ |S| < k, S \not\subseteq (-\exp(k^{\epsilon_0}),\exp(k^{\epsilon_0}))^2 \right] \leq C \, \alpha_1(R)^2 \exp(-k^{\epsilon_0}/C) \, .
\]
In Subsection~$4.2$ of~\cite{GPS} the authors prove an analogue of~\eqref{verysmallspectrum} for the indicator function of the crossing of the square $g_n$ while \eqref{verysmallspectrum} itself is proved in Subsection~$4.4$ of~\cite{GPS}. In Remark~$4.5$ of~\cite{GPS}, it is explained that the proof written in their Subsection~$4.2$ implies a clustering effect for the spectral sample of $g_n$ conditioned to be of size less than $\log(n)$. With same ideas, a clustering effect for the spectral sample of $f_R$ conditioned to be of atypically small size can be extracted from Subsection~$4.4$ of~\cite{GPS}, and we can thus obtain Proposition~\ref{soimportantlog}.

\subsection{Small residual spectral mass away from the origin}\label{sectionlittlepart}

Let us recall what are the three main steps in~\cite{GPS} in order to prove the upper-bound part of Theorem~\ref{GPScool}:

\begin{enumerate}
\item The estimate~\eqref{verysmallspectrum} on the probability of a very small spectrum. 
\item The following result on the independence structure of the spectral sample (this is not exactly the result stated in~\cite{GPS} but its proof is exactly the same):
\begin{prop}[Proposition~$5.12$ in~\cite{GPS}]\label{weakindependence}
Let $1 \leq r \leq R$ and let $\mathcal{S}_{f_R}$ be a spectral sample of $f_R$. Also, let $W \subseteq \mathcal{I}_R$ and let $B \subseteq \R^2$ be an $l \times l'$ rectangle such that: \textit{(a)} $r \leq l,l' \leq 2r$ and \textit{(b)} $W \cap B = \emptyset$. Let $B' \subset B$ be the $r/3 \times r/3$-square which has the same center as $B$. Suppose that $B' \subset [-R,R]^2$ and $B \cap [-4r,4r]^2 = \emptyset$. We also assume that $r \geq \overline{r}$, where $\overline{r} < + \infty$ is some universal constant. Finally, let $\mathcal{Z}=\mathcal{Z}_r$ be a random subset of $\mathcal{I}_R$ that is independent of $\mathcal{S}_{f_R}$, where each element of $\mathcal{I}_R$ is in $\mathcal{Z}$ with probability $(r^2\alpha_4(r))^{-1}$ independently of the others. Then, there exists a universal constant $a>0$ such that:
\[
\Pro \left[ \mathcal{S}_{f_R} \cap B' \cap \mathcal{Z} \neq \emptyset \cond \mathcal{S}_{f_R} \cap B \neq \emptyset = \mathcal{S}_{f_R} \cap W \right] \geq a \, .
\]
\end{prop}
Let us emphasize the fact that in the conditioning we cannot have $\mathcal{S}_{f_R} \cap W' \neq \emptyset$ for some $W' \subseteq \mathcal{I}_R$. In other words, we can only deal with \textbf{negative information} about the spectral sample.
\item A large deviation result:
\begin{prop}[Proposition~$6.1$ in~\cite{GPS}]\label{largedev} Take $I \neq \emptyset$ a finite set. Let $x$ and $y$ be $\left\lbrace 0,1 \right\rbrace^I$-valued random variables such that a.s. $y_i \leq x_i$ for all $i \in I$. We write $X = \sum_{i \in I} x_i$ and $Y = \sum_{i \in I} y_i$. Suppose that there exists a constant $a \in (0,1]$ such that, for  each $i \in I$ and every $J \subseteq I \setminus \left\lbrace i \right\rbrace$:
\[
\Pro \left[ y_i = 1 \cond y_j = 0 \; \forall j \in J \right] \geq a \, \Pro \left[ x_i = 1 \cond y_j = 0 \; \forall j \in J \right] \, .
\]
Then:
\[
\Pro \left[ Y = 0 \cond X > 0 \right] \leq \frac{1}{a} \, \E \left[ \exp \left( -aX/e \right) \cond X > 0 \right] \, .
\]
\end{prop}
\end{enumerate}
These results combine well to prove the upper-bound part of Theorem~\ref{GPScool} (i.e. Theorem~$7.3$ in~\cite{GPS}). In order to prove Theorem~\ref{soimportant}, we will use Propositions~\ref{weakindependence} and~\ref{largedev} and we will need an analogue of the estimate~\eqref{verysmallspectrum} where we only look at the part of the spectral sample that is outside of the box of radius $r_0$. In the following subsection, we state and prove this analogous result.

\subsubsection{A combinatorial result in the flavour of Section~$4$ of~\cite{GPS}} \label{sss.verylittle}

The aim of this section is to prove the following result. This is the more technical part of the paper and we have chosen to: \textit{(a)} divide the proof in three paragraphs and \textit{(b)} at the beginning of each paragraph, explain what is similar to (and different from) Section~$4$ of~\cite{GPS}.

\begin{prop}\label{smallspectrumbis}
If $S \subseteq \mathcal{I}_R$, write $S^{r_0}_r$ for the set of the squares of the grid $r \Z^2$ that intersect $S \setminus (-r_0,r_0)^2$. Then, there exist constants $\theta < + \infty$ and $\epsilon > 0$ such that, for all $k \in \N_+$ and all $1 \leq r \leq r_0 \leq R/2$:
\[
\widehat{\Pro}_{f_R} \left[ |S_r^{r_0}| = k \right] \leq g(k) \, \frac{\alpha_1(R)}{\alpha_1(r_0)} \left( \frac{r_0}{r} \right)^{1-\epsilon} \alpha_4(r,r_0) \, ,
\]
where $g(k)=2^{\theta \text{log}_2^2(k+2)}$.
\end{prop}

\begin{proof} We will prove the equivalent inequality:
\begin{equation}\label{whatwewantwithQ}
\widehat{\Q}_{f_R} \left[ |S_r^{r_0}| = k \right] \leq g(k) \, \frac{\alpha_1(R)^2}{\alpha_1(r_0)} \left( \frac{r_0}{r} \right)^{1-\epsilon} \alpha_4(r,r_0) \, .
\end{equation}

\paragraph{A. The $r_0$-decorated centered annulus structures.} 
As in Section~$4$ of~\cite{GPS}, we begin with some definitions concerning annulus structures. More precisely, we first state the definition of centered annulus structures from Section~$4$ of~\cite{GPS}, and we recall the main preliminary result of~\cite{GPS} about these objects (see~\eqref{annulusstrcuturesinequality}). We then explain how to construct annulus structures more suitable for our work: the $r_0$-decorated centered annulus structures, that will be helpful when we want to take into account what happens near the boundary of the square $(-r_0,r_0)^2$. The proof of the result analogous to~\eqref{annulusstrcuturesinequality} (see Lemma~\ref{generalizedannulusstructuresinequality}) will be a little more difficult than the proof (in~\cite{GPS}) of~\eqref{annulusstrcuturesinequality}, and we will need to rely on a general property of spectral samples: Lemma~\ref{JP}. The proof of Lemma~\ref{JP}, based on ideas that come from Section~$2$ of~\cite{GPS}, is postponed to Appendix~\ref{a.B}.

\begin{defi}[Section $4$ of~\cite{GPS}]\label{defiannulussctruc}
Consider $(\mathcal{A},r_\mathcal{A})$ where $r_\mathcal{A} \in [0,R]$ and $\mathcal{A}$ is a collection of mutually percolation disjoint square annuli $A$ that satisfy:
\begin{enumerate}
\item either $A$ is included in $[-R,R]^2$ and is centered at $0$. Such $A$ are called \textbf{centered annuli},
\item or $A$  is included in $[-R,R]^2$ and the outer square of $A$ does not contain $0$. Such $A$ are called \textbf{interior annuli},
\item or $A$ is centered at a point of a side of $[-R,R]^2$ that is at distance at least the outer radius of $A$ from the other sides. Such $A$ are called \textbf{side annuli},
\item or $A$ is centered at a corner of $[-R,R]^2$ and the outer radius of $A$ is less than or equal to $R$. Such $A$ are called \textbf{corner annuli}. (Distinguishing between corner and interior annuli was interesting in~\cite{GPS} for the study of the indicator function of the crossing of the square $g_n$. In our case - where we are only interesting in $f_R$ - it will not be very useful but we have kept this distinction since: \textit{(a)} it does not add any technical difficulty and \textit{(b)} it will make it easier when refering to~\cite{GPS}.)
\end{enumerate}

Suppose also that the annuli and $[-r_\mathcal{A},r_\mathcal{A}]^2$ are percolation disjoint. Then, $(\mathcal{A},r_\mathcal{A})$ is called a \textbf{centered annulus structure}. For each $A \in \mathcal{A}$ we write $h(A)$ for the probability of having the $1$-arm event in $A$ if $A$ is a centered annulus, the $4$-arm event in $A$ if it is an interior annulus, the half-plane $3$-arm event (in $A$ intersected with $[-R,R]^2$) if it is a side annulus and the quarter-plane $3$-arm event (in $A$ intersected with $[-R,R]^2$) if it is a corner annulus. We will often write $\mathcal{A}$ instead of $(\mathcal{A},r_\mathcal{A})$. Finally, a subset $S\subseteq \mathcal{I}_R$ is called \textbf{compatible} with $\mathcal{A}$ if \textit{(a)} for each non-centered annulus $A \in \mathcal{A}$ there exists $i \in S \cap B(A)$ where $B(A)$ is the inner square of $A$ (more precisely, there exists $i \in S$ whose tile is included in $B(A)$) and \textit{(b)} there is no $i \in S$ whose tile intersects $\bigcup_{A \in \mathcal{A}} A$ (see Figure~\ref{centredAS}). (In~\cite{GPS} it is also asked that $S$ intersects $B(A)$ when $A$ is a centered annulus. That is the reason why the estimate~\eqref{annulusstrcuturesinequality} below is written with $S$ whereas the same result in~\cite{GPS} is written with $S \cup \lbrace 0 \rbrace$.)
\end{defi}

\begin{figure}[h!]
\centering
\includegraphics[scale=0.58]{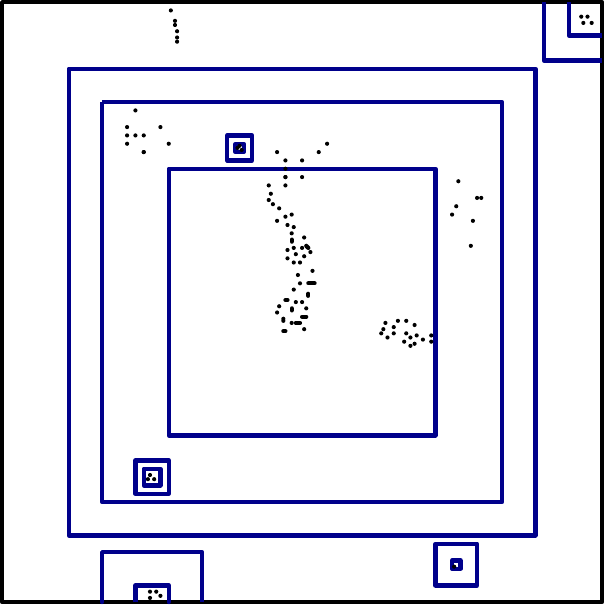}
\caption{A set $S$ compatible with a centered annulus structure.\label{centredAS}}
\end{figure}

The following is Lemma~$4.8$ in \cite{GPS}: if $\mathcal{A}$ is a centered annulus structure, then:
\begin{equation}\label{annulusstrcuturesinequality}
\widehat{\Q}_{f_R} \left[ S \text{ is compatible with } \mathcal{A} \right] \leq \alpha_1(r_\mathcal{A}) \prod_{A \in \mathcal{A}} h(A)^2 \, .
\end{equation}
We want to generalise this by adding some annuli centered on the boundary of $[-r_0,r_0]^2$. Consider $\mathcal{A}$ a centered annulus structure such that $r_\mathcal{A}=r_0$. Let $n \in \N$ and let $A_1, \cdots, A_n$ be some mutually percolation disjoint interior annuli which are percolation disjoint from all the annuli of $\mathcal{A}$. We also assume that, for all $j \in \lbrace 1, \cdots, n \rbrace$, $A_j$ is centered at a point of $\partial [-r_0,r_0]^2$ and is percolation disjoint from $[-r_0/2,r_0/2]^2$. Let $\widetilde{\mathcal{A}} = \mathcal{A} \cup \lbrace A_1, \cdots, A_n \rbrace$. We call such a set of annuli a \textbf{$r_0$-decorated centered annulus structure}.

We say that a subset $S\subseteq \mathcal{I}_R$ is compatible with $\widetilde{\mathcal{A}}$ if \textit{(a)} $S$ is compatible with $\mathcal{A}$, \textit{(b)} for all $j \in \lbrace 1, \cdots, n \rbrace$, there exists $i \in S$ such that the tile of $i$ is included in the inner square of $A_j$, and \textit{(c)} there is no $i \in S \setminus (-r_0,r_0)^2$ whose tile intersects $A_j$. Note that there may exist $i \in S \cap (-r_0,r_0)^2$ whose tile intersects $A_j$, see Figure~\ref{r0centredAS}.

In order to generalize~\eqref{annulusstrcuturesinequality} to these $r_0$-decorated centered annulus structures, we need to introduce a new notation. First, if $B \subseteq \mathcal{I}_R$, we write $\mathcal{F}_B$ for the $\sigma$-field of subsets of $\lbrace -1,1 \rbrace^{\mathcal{I}_R}$ generated by the restriction of $\omega$ to the bits in $B$. Next, for any $j \in \lbrace 1, \cdots, n \rbrace$, we write $h^{r_0}(A_j)$ for the non-negative real number such that:
\[
h^{r_0}(A_j)^2 = \E_{1/2} \left[ \Pro_{1/2} \left[ 4\text{-arm event in } A_j \cond \mathcal{F}_{\mathcal{I}_{R} \cap (-r_0,r_0)^2} \right]^2 \right] \, .
\]

\begin{figure}[h!]
\centering
\includegraphics[scale=0.58]{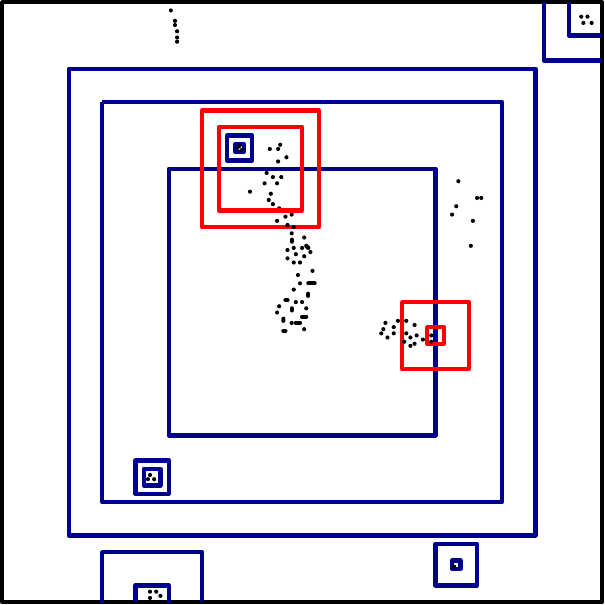}
\caption{A set $S$ compatible with a $r_0$-decorated centered annulus structure.\label{r0centredAS}}
\end{figure}

We now state the following analogue of~\eqref{annulusstrcuturesinequality}, whose proof is \textbf{one of the main steps in the present subsection that differs from \cite{GPS}} (and the main part of its proof is postponed to Appendix \ref{a.B}).
\begin{lem}\label{generalizedannulusstructuresinequality}
Let $\mathcal{A}$ and $\widetilde{\mathcal{A}}$ be as above. We have:
\[
\widehat{\Q}_{f_R} \left[ S \text{ is compatible with } \widetilde{\mathcal{A}} \right] \leq \alpha_1(r_0/2) \prod_{j = 1}^n (4 \, h^{r_0}(A_j)^2 ) \prod_{A \in \mathcal{A}} (4 \, h(A)^2 ) \, .
\]
\end{lem}
In order to prove this lemma, we first need a general property about the spectral sample. In order to state this property, we need the following definition:

\begin{defi}\label{jointlypivotal}
Let $h : \Omega_R = \lbrace -1,1 \rbrace^{\mathcal{I}_R} \rightarrow \R$. Also, let $J \subseteq \mathcal{I}_R$, and let $J_1, \cdots, J_n$ be mutually disjoint subsets of $\mathcal{I}_R$. We say that $J$ is \textbf{pivotal} for $h$ and some configuration $\omega \in \Omega_R$ if changing the values of the sites/edges in $J$ can change the value of $h$. We say that $J_1, \cdots, J_n$ are \textbf{jointly pivotal} for $h$ and some configuration $\omega$ if, for every $j_0 \in \lbrace 1, \cdots, n \rbrace$, there is a choice of configuration in $\cup_{j \neq j_0} J_j$ making $J_{j_0}$ pivotal. We will use the following notation:
\[
(\text{Jointly Pivotal})_{J_1, \cdots, J_n}(h) = JP_{J_1, \cdots, J_n}(h)=\left\lbrace J_1, \cdots, J_n \text{ are jointly pivotal for } h \right\rbrace \, .
\]
Note that $JP_{J_1, \cdots, J_n}(h)$ is an event measurable with respect to the configuration outside $\cup_j J_j$.
\end{defi}

The proof of the following lemma is postponed to Appendix~\ref{a.B}.

\begin{lem}\label{JP}
Let $h$ be as above. Let $n \in \N_+$ and let $J_1, \cdots, J_n, \, W$ be mutually disjoint subsets of $\mathcal{I}_R$. Then:
\[
\widehat{\Q}_h \left[ \forall j, S \cap J_j \neq \emptyset \, , \,S \cap W = \emptyset \right]  \leq 4^n \parallel h \parallel_\infty^2 \E_{1/2} \left[ \Pro_{1/2} \left[ JP_{J_1, \cdots, J_n}(h) \cond \mathcal{F}_{W^c} \right]^2 \right],
\]
where, for every $B \subseteq \mathcal{I}_R$, $\mathcal{F}_{B}$ is the $\sigma$-field of subsets of $\lbrace -1,1 \rbrace^{\mathcal{I}_R}$ generated by the restriction of $\omega$ to the sites/edges in $B$.
\end{lem}

\begin{proof}[Proof of Lemma~\ref{generalizedannulusstructuresinequality}]
If there are only centered annuli in $\widetilde{\mathcal{A}}$ (and more generally if $n=0$) then this is a direct consequence of~\eqref{annulusstrcuturesinequality} (and we obtain the result without the factors $4$; see also the end of the proof for another approach). Hence, we can assume that there exist non-centered annuli in $\widetilde{\mathcal{A}}$. Also, we write $W$ for the set of the $i \in \mathcal{I}_R$ whose tile intersects some $A \in \mathcal{A}$ and of the $i \in \mathcal{I}_R \setminus (-r_0,r_0)^2$ whose tile intersects some $A_j$, $j \in \lbrace 1,  \cdots, n \rbrace$. If $A$ is some annulus, let $B(A)$ be the set of the $i \in \mathcal{I}_R$ such that the tile of $i$ is included in the inner square of $A$. Also, let $\widetilde{\mathcal{A}}'$ be the subset of $\widetilde{\mathcal{A}}$ whose elements are the non-centered annuli $A \in \widetilde{\mathcal{A}}$ such that the inner square of $A$ does not contain any other annulus of $\widetilde{\mathcal{A}}$. We have:
\[
\widehat{\Q}_{f_R} \left[ S \text{ is compatible with } \widetilde{\mathcal{A}} \right] = \widehat{\Q}_{f_R} \left[ \forall A \in \widetilde{\mathcal{A}}', \, S \cap B(A) \neq \emptyset \, , \, S \cap W = \emptyset \right] \, .
\]
We now use Lemma~\ref{JP} (which is the main step in this proof). It implies that the above is at most:
\[
4^{|\widetilde{\mathcal{A}}'|} \E_{1/2} \left[ \Pro_{1/2} \left[ JP_{B(A)_{A \in \widetilde{\mathcal{A}}'}} ( f_R ) \cond \mathcal{F}_{W^c} \right]^2 \right] \leq 4^{|\widetilde{\mathcal{A}}|} \E_{1/2} \left[ \Pro_{1/2} \left[ JP_{B(A)_{A \in \widetilde{\mathcal{A}}'}} ( f_R ) \cond \mathcal{F}_{W^c} \right]^2 \right] \, .
\]
Since our annuli are mutually percolation disjoint and are percolation disjoint from the square  $[-r_0/2,r_0/2]^2$, the event $JP_{B(A)_{A \in \widetilde{\mathcal{A}}'}}(f_R)$ implies the $4$-arm event in every interior annulus $A \in \mathcal{A}$ and in $A_j$ for all $j \in \lbrace 1, \cdots, n \rbrace$. Moreover, it implies the $3$-arm event in $A$ intersected with $[-R,R]^2$ for every side or corner annulus $A \in \mathcal{A}$, the $1$-arm event in any centered annulus $A \in \mathcal{A}$ and the event $\lbrace 0 \leftrightarrow r_0/2 \rbrace$. For any interior annulus $A \in \mathcal{A}$ we have:
\[
\E_{1/2} \left[ \Pro_{1/2} \left[ 4 \text{-arm event in } A \cond \mathcal{F}_{W^c} \right]^2 \right] = h(A)^2 \, ,
\]
since the $4$-arm event in $A$ is independent of the configuration restricted to $W^c$. The analogous equalities hold for side, corner and centered annuli. Similarly, it is not difficult to see that for all $j \in \lbrace 1, \cdots, n \rbrace$ we have:
\[
\E_{1/2} \left[ \Pro_{1/2} \left[ 4 \text{-arm event in } A_j \cond \mathcal{F}_{W^c} \right]^2 \right] = h^{r_0}(A_j)^2 \, .
\]
Finally, note that:
\[
\E_{1/2} \left[ \Pro_{1/2} \left[ 0 \leftrightarrow r_0/2 \cond \mathcal{F}_{W^c} \right]^2 \right] = \Pro_{1/2} \left[ 0 \leftrightarrow r_0/2 \right] = \alpha_1(r_0/2) \, .
\]
By spatial independence, we are done. (We could have treated the case where there are only centered annuli in $\widetilde{\mathcal{A}}$ by very similar ideas but by using~$(2.9)$ from~\cite{GPS} instead of Lemma~\ref{JP}. It is actually easier than the above case.)
\end{proof}

Let $j \in \lbrace 1, \cdots, n \rbrace$ and let $\rho_j$ and $\rho_j'$ be the inner and outer radii of $A_j$. It is not difficult to see that there exists a (upper, lower, left or right, depending on $j$) half-plane $H_j$ such that the center of $A_j$ belongs to the boundary of $H_j$ and $(-r_0,r_0)^2 \subseteq H_j$. Note that, for any $B_1 \subseteq B_2 \subseteq \mathcal{I}_R$ and any function $h : \Omega_R \rightarrow \R$, we have:
\[
\E_{1/2} \left[ \E_{1/2} \left[ h \cond \mathcal{F}_{B_1} \right]^2 \right] \leq \E_{1/2} \left[ \E_{1/2} \left[ h \cond \mathcal{F}_{B_2} \right]^2 \right] \, .
\]
Let $\half_j = H_j \cap \mathcal{I_R}$. The above implies that:
\[
h^{r_0}(A_j)^2 \leq \E_{1/2} \left[ \Pro_{1/2} \left[ 4 \text{-arm event in } A_j \cond \mathcal{F}_{\half_j} \right]^2 \right] \, .
\]
Together with Lemma~\ref{lemmerigolo}, this implies that there exist $\epsilon > 0$ and $C<+\infty$ such that:
\begin{equation}\label{halfplaneinequality}
h^{r_0}(A_j) \leq C \, \alpha_4(\rho_j,\rho_j') \left( \frac{\rho_j'}{\rho_j} \right)^{-\epsilon} \, .
\end{equation}
By possibly decreasing $\epsilon$, we assume the following technical condition (for every $1 \leq \rho_1 \leq \rho_2$):
\begin{equation}\label{technicalepsilon1}
\left( \alpha_4(\rho_1,\rho_2) \, \frac{\rho_2}{\rho_1} \right)^2 \leq \grandO{1} \, \alpha_4(\rho_1,\rho_2) \, \left( \frac{\rho_2}{\rho_1} \right)^{1-\epsilon} \, .
\end{equation}
(This is possible since $\alpha_4(\rho_1,\rho_2)  \frac{\rho_2}{\rho_1}$ is polynomially small in $\frac{\rho_1}{\rho_2}$, see the right-hand inequality of~\eqref{alpha4}.) We also assume the following stronger condition: There exists $c > 0$ such that:
\begin{equation}\label{technicalepsilon2}
\left( \alpha_4(\rho_1,\rho_2) \, \frac{\rho_2}{\rho_1} \right)^2 \leq \frac{1}{c} \left( \alpha_4(\rho_1,\rho_2) \, \left( \frac{\rho_2}{\rho_1} \right)^{1-\epsilon} \right)^{1.01} \left( \frac{\rho_2}{\rho_1} \right)^{-c} \, .
\end{equation}
(This is possible since $\left( \alpha_4(\rho_1,\rho_2) \, \frac{\rho_2}{\rho_1} \right)^{0.99}$ is polynomially small in $\frac{\rho_1}{\rho_2}$. Moreover, this is a stronger condition than~\eqref{technicalepsilon1} since $\alpha_4(\rho_1,\rho_2) \, \left( \frac{\rho_2}{\rho_1} \right)^{1-\epsilon}$ is polynomially small in $\frac{\rho_1}{\rho_2}$.) See Remark~\ref{theexplanation} below where we explain the reason why we need~\eqref{technicalepsilon1} and~\eqref{technicalepsilon2}.

We now write $h(\widetilde{\mathcal{A}})=\prod_{A \in \mathcal{A}} h(A) \prod_{j \in \lbrace 1, \cdots, n \rbrace} h^{r_0}(A_j)$. More generally, for any $\widetilde{\mathcal{A}}' \subseteq \widetilde{\mathcal{A}}$, we write $h(\widetilde{\mathcal{A}}')$ for the obvious analogue where we only consider the annuli in $\widetilde{\mathcal{A}}'$.

Let us fix $1 \leq r \leq r_0 \leq R/2$ and $k \in \N_+$ for the rest of the proof. 
Recall that for any $\theta<+\infty$, we defined in the statement of Proposition \ref{smallspectrumbis} $g(k):= 2^{\theta \log_2^2(k+2)}$. Thanks to Lemma~\ref{generalizedannulusstructuresinequality}, we have the following:
\begin{lem}\label{l.criterion_annulus_struc}
To prove Proposition~\ref{smallspectrumbis}, it sufficient to show that there exists an absolute constant $\theta < +\infty$ such that, if $g(k) \frac{\alpha_1(R)}{\alpha_1(r_0)} \left( \frac{r_0}{r} \right)^{1-\epsilon} \alpha_4(r,r_0) \leq 1$ (where $\epsilon$ is as in inequalities~\eqref{halfplaneinequality},~\eqref{technicalepsilon1} and~\eqref{technicalepsilon2}), then there exists a set $\mathfrak{U}_k$ of $r_0$-decorated centered annulus structures such that: \textit{(a)} for all $S$ that satisfies $|S_r^{r_0}| = k$, there exists $\widetilde{\mathcal{A}} \in \mathfrak{U}_k$ compatible with $S$, and \textit{(b)}:
\begin{equation}\label{lastlast}
\sum_{\widetilde{\mathcal{A}} \in \mathfrak{U}_k} 4^{|\widetilde{\mathcal{A}}|} \, \alpha_1(r_0/2) \, h(\widetilde{\mathcal{A}})^2 \leq g(k) \frac{\alpha_1(R)^2}{\alpha_1(r_0)} \left( \frac{r_0}{r} \right)^{1-\epsilon} \alpha_4(r,r_0) \, .
\end{equation}
(If $g(k) \frac{\alpha_1(R)}{\alpha_1(r_0)} \left( \frac{r_0}{r} \right)^{1-\epsilon} \alpha_4(r,r_0) > 1$ then~\eqref{whatwewantwithQ} is trivial since $\alpha_1(R)$ is the total mass of $\widehat{\Q}_{f_R}$.)
\end{lem}

\begin{rem}\label{theexplanation}
The reason why we need conditions~\eqref{technicalepsilon1} and~\eqref{technicalepsilon2} on $\epsilon$ can be explained as follows: in the next paragraph, we will construct our sets of $r_0$-decorated centered annulus structures $\mathfrak{U}_k$. Then, we will estimate the quantities $\sum_{\widetilde{\mathcal{A}} \in \mathfrak{U}_k} 4^{|\widetilde{\mathcal{A}}|} \, \alpha_1(r_0/2) \, h(\widetilde{\mathcal{A}})^2$ of~\eqref{lastlast} by induction. At each step of the induction, the annuli centered on the boundary of $[-r_0,r_0]^2$ will induce estimates of the form $\left( \frac{r_0}{r} \right)^{1-\epsilon} \alpha_4(r,r_0)$ and the other interior annuli will induce estimates of the form $\left( \frac{r_0}{r} \right)^2 \alpha_4(r,r_0)^2$.
Since we will deal a lot with such terms, it will be useful to know which of them is the dominant term, and that is why we assume in~\eqref{technicalepsilon1} that $\epsilon$ is sufficiently small so that $\left( \frac{r_0}{r} \, \alpha_4(r,r_0) \right)^2 \leq \grandO{1} \alpha_4(r,r_0) \left( \frac{r_0}{r} \right)^{1-\epsilon}$.

The reason why we need the existence of the exponents $c$ (even very small) and $1.01$ in the stronger assumption~\eqref{technicalepsilon2} is that, at each step of the induction, we want to have some room to manoeuvre. Actually, we could have chosen any number $1+a \in (1,2)$ instead of $1.01$ and the proof would have been exactly the same by only replacing all the exponents $1.01$ by $1+a$ (in particular in the estimate~\eqref{recr_0} below). The reason why we have chosen $1.01$ is only that it is nice to think of $a$ as being very small so that we can have precise estimates about the exponent $\epsilon$ that we are able to consider. See Remark~\ref{aquantitaiveepsilon} for more about this.

At some point of the proof, it will probably be more natural to work with an exponent $1+a$ close to $2$ instead of $1.01$ (in the same spirit as the exponent $1.99$ that appears in Section~$4$ of~\cite{GPS}) since the exponent will be extracted from a geometric sum of the form $\sum_{d=2}^{k'} \gamma^d$ (see for instance the proof of~\eqref{recr_0} below). With the above explanations, we hope that the fact that we have chosen $1+a=1.01$ will not confuse the reader.
\end{rem}

\medskip

We now proceed to the construction of the sets $\mathfrak{U}_k$.

\paragraph{B. The construction of the $r_0$-decorated centered annulus structures.} Contrary to Paragraph~A, the novelty of this paragraph in comparison to~\cite{GPS} is only that we extend some definitions to what happens near the boundary of the box $(-r_0,r_0)^2$. Still, this paragraph is crucial to define carefully the sets $\mathfrak{U}_k$ (in particular, we will specify how we associate an annulus to the singleton $\lbrace \lbrace 0 \rbrace \rbrace$ and how we define the quantities $\gamma^{r_0}_{\rho_1}(\rho_2)$).
\medskip

In Section~$4$ of \cite{GPS}, the authors explain how we can classify the annulus structures. We will follow the same ideas to classify our $r_0$-decorated centered annulus structures. Let $S \subseteq \mathcal{I}_R$ be such that $|S_r^{r_0}| = k$. Let $j \in \N$. If $j \geq 1$, we define $G_j$ as the graph with vertices the elements of $S_r^{r_0} \cup \lbrace 0 \rbrace$ and with edges present between any two points with Euclidean distance from one to the other at most $2^jr$ (say for instance that the distance between two sets is the infimum distance between these two sets - the fact that the vertices are squares except $\lbrace 0 \rbrace$ that is a point will not be a problem). For the case $j=0$, $G_j$ is simply the graph with vertices the elements of $S_r^{r_0} \cup \lbrace 0 \rbrace$ and with no edge. The authors of \cite{GPS} explain how to construct annuli around the connected components of the $G_j$'s. Let us explain it (the difference will be that we will need to change the definition for annuli close to $\partial[-r_0,r_0]^2$):
\medskip

Let $\overline{j}=\lfloor \log_2 \left( \frac{R}{kr} \right) \rfloor-5$ and $J=\lbrace 0, \cdots, \overline{j} \rbrace$. Take $j \in \lbrace 1, \cdots, \overline{j} \rbrace$. A connected component of $G_j$ is called an \textbf{interior cluster at level $j$} if it does not contain $\lbrace 0 \rbrace$, it is not a connected component of $G_{j-1}$ and its distance to $\partial[-r_0,r_0]^2 \cup \partial [-R,R]^2$ is larger than $2^jr$. A connected component of $G_j$ is a \textbf{centered cluster at level $j$} if it contains $\lbrace 0 \rbrace$ and it is not a connected component of $G_{j-1}$. We define by induction on $j \in \lbrace 1, \cdots, \overline{j} \rbrace$ the other clusters: a connected component of $G_j$ is a \textbf{side cluster at level $j$} if it is within distance $2^jr$ of precisely one of the boundary edges of $[-R,R]^2$ and it is not a side cluster at level $j'$ for any $j' \in \lbrace 1, \cdots, j-1 \rbrace$. A connected component of $G_j$ is a \textbf{corner cluster at level $j$} if it is within distance $2^jr$ of precisely two adjacent boundary edges of $[-R,R]^2$ and it is not a corner cluster at level $j'$ for any $j' \in \lbrace 1, \cdots, j-1 \rbrace$. A connected component of $G_j$ is a \textbf{$r_0$-cluster at level $j$} if it does not contain $\lbrace 0 \rbrace$, it is within distance $2^jr$ of $\partial [-r_0,r_0]^2$ and it is not a $r_0$-cluster at level $j'$ for any $j' \in \lbrace 1, \cdots, j-1 \rbrace$. Furthermore, a connected component of $G_0$ (i.e. a singleton) that is not $\lbrace \lbrace 0 \rbrace \rbrace$ is an interior cluster at level $0$ and the singleton $\lbrace \lbrace 0 \rbrace \rbrace$ is a centered cluster at level $0$. Finally, we define a unique cluster at level $\overline{j}+1$ that is the entire set $S_r^{r_0} \cup \lbrace 0 \rbrace$ and is called the \textbf{top cluster}.

With these definitions, for any $j \in J$ and any connected component $C$ of $G_j$, there exists $j' \in J$ such that $C$ is a cluster at level $j'$ of one of the types described above. Moreover, for any type (i.e. interior, centered, side, corner, $r_0$- or top) of cluster there exists at most one level $j' \in J$ such that $C$ is a cluster at level $j'$ of this type, and for any level $j' \in J$ there exists at most one type of cluster such that $C$ is a cluster at level $j'$ of this type.

For a cluster $C$ of any type, we write $j(C)$ for the level of $C$.\\

We want to define a tree structure for our clusters. Let $C$ be a cluster of some type at some level $j \in J$. The parent of $C$ is either $C$ itself if $C$ is also a cluster of some other type at some level $j' > j$ (and we choose the smallest level $j'$ if there are more than one choice) or the smallest cluster that properly contains $C$ otherwise (and we also choose the smallest level $j' > j$). We write $C^p$ for the parent of $C$. For instance, the children of a $r_0$-cluster can only be interior and $r_0$-clusters; moreover, a $r_0$-cluster at level $j>0$ either has a single child that is an interior cluster or has at least two children.

Now, for any of the clusters $C$ described above (except for the top cluster), we define an annulus $A_C$. The inner radius of this annulus will be $2^{j(C)+4}|C|$ and the outer radius will be $2^{j(C^p)-4}$. The center will be $0$ if $C$ is a centered cluster and the corner associated to $C$ if $C$ is a corner cluster. In the other cases, we use some deterministic law to choose a vertex $v=v(C)$ of $C$ and we choose the center of the annulus as follows: If $C$ is an interior cluster, we decide that the center of $A_C$ is the (or one of the) nearest point(s) of $v$ whose coordinates are divided by $2^{j(C)}r$. If $C$ is a side cluster, the center of $A_C$ is the (or one of the) nearest point(s) of $v$ that is on $\partial[-R,R]^2$ and whose coordinate that is not $R$ is divided by $2^{j(C)}r$. We do exactly the same thing with $r_0$-clusters but now \textbf{we center the annulus on $\partial[-r_0,r_0]^2$}. (When the outer radius is larger than the inner radius, $A_C$ is the empty annulus.)

There is only \textbf{one exception}: if $C$ is the singleton $\lbrace \lbrace 0 \rbrace \rbrace$, we decide that the inner radius is $2^4r \vee (r_0+2)$ instead of $2^4r$ (and the outer radius is still $2^{j(C^p)-4}$).
\medskip

All these annuli define a $r_0$-decorated centered annulus structure $\widetilde{\mathcal{A}}_1(S)$ compatible with $S$ (to see this, write $A_1, \cdots, A_n$ for the annuli associated to the $r_0$-clusters). Let us for instance check that the annuli associated to the $r_0$-clusters are percolation disjoint from $[-r_0/2,r_0/2]^2$. Let $C$ be a $r_0$-cluster at level $j$. Some vertex of $C$ is at distance less than or equal to $2^jr$ of $\partial[-r_0,r_0]^2$ and $\lbrace 0 \rbrace \notin C$, so $2^{j(C^p)}r \leq 2(\sqrt{2}r_0 + 2^jr)$ i.e. $2^{j(C^p)-4}r \leq r_0/2^{2.5}+2^{j-3}r$. Therefore, if the annulus associated to $C$ is not empty, then $2^{j+4}r \leq 2^{j+4}r|C| \leq 2^{j(C^p)-4}r \leq r_0/2^{2.5}+2^{j-3}r$ so $2^{j-3}r \leq r_0/2^{8.5}$ and the outer radius is $2^{j(C^p)-4} < r_0/2^{2.5}+ r_0/2^{8.5} \leq r_0/4$. So, if $r_0$ is sufficiently large (for instance if $r_0 \geq 8$) then the annulus is percolation disjoint from $[-r_0/2,r_0/2]^2$. If $r_0 < 8$, it is not difficult to see (with very similar arguments) that any annulus associated to a $r_0$-cluster is empty.

For the other conditions that we have to check to prove that $\widetilde{\mathcal{A}}_1(S)$ is a $r_0$-decorated centered annulus structure compatible with $S$, we use similar arguments (see also Section~$4$ of \cite{GPS} where the authors explain some similar results).
\medskip

Actually, since we have defined these sets of annuli for every $S$ such that $|S_r^{r_0}| = k$, we have defined too many different $r_0$-decorated centered annulus structures and that would make the sum in \eqref{lastlast} much bigger than we would like (see Section~$4$ of~\cite{GPS} for more about such a problem). So, we need a few other definitions. Consider four positive real numbers $+\infty > \theta > \theta^* > \theta^{r_0} > \theta' > 1$ that we will determine later. We define $g'$, $g^{r_0}$ and $g^*$ like $g$ but with $\theta'$, $\theta^{r_0}$ and $\theta^*$ instead of $\theta$. We also define:
\[
\gamma_{\rho_1}(\rho_2) = \left( \frac{\rho_2}{\rho_1} \alpha_4(\rho_1,\rho_2) \right)^2 \, ,
\]
\[
\gamma^{r_0}_{\rho_1}(\rho_2) = \left( \frac{\rho_2}{\rho_1} \right)^{1-\epsilon} \alpha_4(\rho_1,\rho_2)
\]
(where $\epsilon$ is the constant in~\eqref{halfplaneinequality},~\eqref{technicalepsilon1}, and~\eqref{technicalepsilon2}),
\[
\gamma^*(\rho_1,\rho_2) = \alpha_1(\rho_1,\rho_2)^2 \, ,
\]
and $\overline{\gamma}_{\rho_1}(\rho_2) = \inf_{\rho' \in [1,\rho_2]} \gamma_{\rho_1}(\rho')$, and $\overline{\gamma}^{r_0}_{\rho_1}(\rho_2) = \inf_{\rho' \in [1,\rho_2]} \gamma^{r_0}_{\rho_1}(\rho')$. Note that:
\begin{equation}\label{sameorder}
\overline{\gamma}_{\rho_1}(\rho_2) \asymp \gamma_{\rho_1}(\rho_2) \, ,
\end{equation}
and similarly for $\overline{\gamma}^{r_0}$. (The quantities $\overline{\gamma}$ and $\overline{\gamma}^{r_0}$ are defined in order to work with decreasing functions in $\rho_2$, note that $\gamma^*$ is already decreasing in $\rho_2$).

Now, if $C$ is an interior, side or corner cluster then we say that $C$ is \textbf{overcrowded} if we have $g'(|C|) \overline{\gamma}_r(2^{j(C)}r) > 1$. If $C$ is a $r_0$-cluster then we say that $C$ is overcrowded if $g^{r_0}(|C|) \overline{\gamma}_r^{r_0}(2^{j(C)}r) > 1$.  Finally, if $C$ is centered then we say that $C$ is overcrowded if $g^*(|C|) \gamma^*_{r_0}(2^{j(C)}r) \, \overline{\gamma}^{r_0}_r(2^{j(C)}r) > 1$. Note that all clusters at level $0$ are overcrowded. We define a $r_0$-decorated centered annulus structure $\widetilde{\mathcal{A}}(S)$ by \textbf{removing from $\widetilde{\mathcal{A}}_1(S)$ every annulus that corresponds to a proper descendent of an overcrowded cluster}. The $r_0$-decorated centered annulus structure $\widetilde{\mathcal{A}}(S)$ is still compatible with $S$ and we can define:
\[
\mathfrak{U}_k = \left\lbrace \widetilde{\mathcal{A}}(S) : S \subseteq \mathcal{I}_R \text{ such that } |S_r^{r_0}| = k \right\rbrace \, .
\]

Note that from the definition of the $r_0$-decorated centered annulus structures and from the construction above, we have the following:
Let $S \subseteq \mathcal{I}_R$ be such that $|S_r^{r_0}|=k$. Let $A_1, \cdots, A_n$ be the annuli associated to the $r_0$-clusters of $S$ and that have not been removed from $\widetilde{\mathcal{A}}_1(S)$. Also, let $\mathcal{A}(S) = \widetilde{\mathcal{A}}(S) \setminus \lbrace A_1, \cdots, A_n \rbrace$. Then:

\[
h(\widetilde{\mathcal{A}}(S))^2 = \prod_{j=1}^n h^{r_0}(A_j)^2 \prod_{A \in \mathcal{A}(S)} h(A)^2 \, .
\]

\paragraph{C. Summations on the annulus structures.} This paragraph is analogous to the most technical parts of Section~$4$ of~\cite{GPS}. The calculations are of the same flavour as in~\cite{GPS}, but they require to deal with new quantities: those related to the $r_0$-clusters. In particular, we will have to deal with the exponent $\epsilon$ of~\eqref{halfplaneinequality} (and~\eqref{technicalepsilon1},~\eqref{technicalepsilon2}).

\paragraph{C.1. Some estimates proved inductively.} The strategy is to prove some estimates inductively and then conclude thanks to these estimates. Remember that we want to prove \eqref{lastlast}. We need a few last notations. Remember that we have fixed $r$, $R$ and $k$. Let $S \subseteq \mathcal{I}_R$ be such that $|S^{r_0}_r|=k$ and let $C$ be a cluster of $S^{r_0}_r \cup \lbrace 0 \rbrace$ (of any level and any type). We write $\widetilde{\mathcal{A}}'(C)$ for the subset of $\widetilde{\mathcal{A}}(S)$ that corresponds to the proper descendants of $C$.

Take $k' \in \N_+$ and $j \in J$. Let $B$ be a square such that there exists a set $S \subseteq \mathcal{I}_R$ with $|S^{r_0}_r|=k$ and an interior cluster $C$ of $S^{r_0}_r \cup \lbrace 0 \rbrace$ such that: \textit{(a)} the level of $C$ is $j$, \textit{(b)} $|C|=k'$ and \textit{(c)} $B$ is the inner square of the annulus associated to $C$. We also ask that the annulus $A_C$ has not been removed from $\widetilde{\mathcal{A}}_1(S)$ - i.e. we ask that $C$ is not a proper descendant of an overcrowded cluster. We define:
\[
\mathfrak{U}^{int}(B,k',j) = \left\lbrace \widetilde{\mathcal{A}}'(C) : C \text{ as above} \right\rbrace
\]
(note that $\widetilde{\mathcal{A}}'(C)$ does not depend on the choice of the set $S$ such that $C$ is a cluster of $S^{r_0}_r \cup \lbrace 0 \rbrace$) and:
\[
H^{int}(j,k') =  \sup_{B \text{ as above}} \sum_{ \widetilde{\mathcal{A}}' \in \mathfrak{U}^{int}(B,k',j)} 4^{|\widetilde{\mathcal{A}}'|} \, h(\widetilde{\mathcal{A}}')^2 \, .
\]
We do exactly the same thing for centered, side, corner and $r_0$-clusters and define respectively $H^*(j,k')$, $H^+(j,k')$, $H^{++}(j,k')$ and $H^{r_0}(j,k')$ (if there is no such $B$, then the supremum is $0$).
\medskip

We want to show by induction on $j$ that, if $\theta'$, $\theta^{r_0}/\theta'$ and $\theta^*/\theta^{r_0}$ are sufficiently large, then the following inequalities hold for any $j \in J$ and $k' \in \N_+$:
\begin{align}
& \forall \text{ symbol } \natural \in \lbrace int, +, ++ \rbrace, H^\natural(j,k') \leq g'(k') \, \overline{\gamma}_r(2^jr) \label{recintetc} \, ,\\
& H^{r_0}(j,k') \leq \left( g^{r_0}(k') \, \overline{\gamma}^{r_0}_r(2^jr) \right)^{1.01} \label{recr_0} \, ,\\
& H^*(j,k') \leq g^*(k') \, \gamma^*_{r_0}(2^jr) \, \overline{\gamma}^{r_0}_r(2^jr) \, . \label{reccent}
\end{align}

First, note that, due to the definition of overcrowded clusters, if $j \leq J'(k'):= \max \lbrace j \in \N : g'(k') \overline{\gamma}_r(2^jr) > 1 \rbrace$ then inequalities~\eqref{recintetc} are trivially true, if $j \leq J^{r_0}(k'):= \max \lbrace j \in \N : g^{r_0}(k') \, \overline{\gamma}^{r_0}_r(2^jr) > 1 \rbrace$ then inequality~\eqref{recr_0} is trivially true, and if $j \leq J^*(k'):= \max \lbrace j \in \N : g^*(k') \gamma^*_{r_0}(2^jr) \, \overline{\gamma}^{r_0}_r(2^jr) > 1 \rbrace$ then it is the case for~\eqref{reccent}.

\begin{rem}\label{power}
Assume that inequalities~\eqref{recintetc}, \eqref{recr_0} and~\eqref{reccent} hold. Then, they are still true if we raise the right-hand side to any power in $[0,1]$ (distinguish between the two cases $j \leq J^\natural(k')$ and $j>J^\natural(k')$ for any symbol $\natural \in \lbrace int, +, ++, r_0, * \rbrace$). For instance, we will also use~\eqref{recr_0} with exponent $1$ instead of $1.01$.
\end{rem}
\begin{rem}\label{calculus}
Until the end of proof, we will often use the quasi-multiplicativity property and~\eqref{poly}. We will also use that, for any $j_0$ and any $a>0$:
\[
\sum_{j' \geq j_0} \overline{\gamma}^{r_0}_r (2^{j'}r)^a \asymp \overline{\gamma}^{r_0}_r (2^{j_0}r)^a
\]
(where the constant in $\asymp$ may only depend on $a$). Of course, the analogous properties are also true for $\gamma^*$ and $\overline{\gamma}$.
\end{rem}
The estimates~\eqref{recintetc} are proved in Section~$4$ of \cite{GPS} (actually, in~\cite{GPS} there is not the $4^{|\widetilde{\mathcal{A}}'|}$ term in the definition of the $H$'s but that does not change the calculations since, at each step of the induction, the factors $4$ corresponding to the annuli we add are absorbed in the other $\grandO{1}$ terms). Note that in these estimates neither $r_0$-clusters nor centered clusters play a role since the descendents of interior, side and corner clusters cannot be neither centered nor $r_0$-clusters. The idea of the proof of~\eqref{recr_0} is very similar. Let us prove this result. We proceed by induction on $j$. If $j=0$ and $k' \in \N_+$ then we are done (and more generally if $j \leq J^{r_0}(k')$). We take some $j \in J$ and $k' \in \N_+$ such that $j>J^{r_0}(k')$, we assume that~\eqref{recr_0} is true for every $(k'',j')$ with $k'' \in \N_+$ and $j' \in \lbrace 0, \cdots, j-1 \rbrace$, and we want to prove it for $(k',j)$.

Consider some $B$ as in the definition of $H^{r_0}(j,k')$ (if there is no such $B$ then we are done since in this case $H^{r_0}(j,k')=0$). The square $B$ is the inner square of the annulus associated to some $r_0$-cluster $C$ at level $j$ such that: \textit{(a)} $|C|=k'$ and \textit{(b)} \textbf{$C$ is neither overcrowded nor the proper descendant of an overcrowded cluster}. Let $C_1, \cdots, C_d$ be the children of $C$, let $A_{C_1}, \cdots, A_{C_d}$ be the annuli associated to $C_1, \cdots, C_d$ (that have not been removed by the observation \textit{(b)} above) and let $B_1, \cdots, B_d$ be the inner squares of these annuli. Note that either $d=1$ and $C_1$ is an interior cluster or $d \geq 2$ and the $C_i$'s are either interior or $r_0$-clusters. Moreover, if we know that $C_i$ is an interior (respectively $r_0$-) cluster at level $j_i$, then there are at most $\grandO{1} (k'2^jr/(2^{j_i}r))^2 = \grandO{1} (k'2^{j-j_i})^2$ (respectively $\grandO{1} k'2^jr/(2^{j_i}r) = \grandO{1} k'2^{j-j_i}$) possible choices for $B_i$. Furthermore, if $k_i$ is the cardinal of $C_i$ then the inner radius of $A_{C_i}$ is $k_i2^{j_i+4}r$ and its outer radius is $2^{j-4}r$. Hence, if $C_i$ is an interior cluster, then $h(A_{C_i})^2 = \alpha_4(k_i 2^{j_i+4}r,2^{j-4}r)^2$. Moreover, \eqref{halfplaneinequality} implies that if $C_i$ is a $r_0$-cluster, then we have $h^{r_0}(A_{C_i})^2 \leq \grandO{1} \alpha_4(k_i 2^{j_i+4}r,2^{j-4}r) \, \left( 2^{j-4}r /(k_i2^{j_i+4}r) \right)^{-\epsilon}$.

If we distinguish between the cases $d=1$ and $d \geq 2$, we obtain that:
\begin{align}
& H^{r_0}(j,k') \leq \grandO{1} \sum_{j_1 < j} (k'2^{j-j_1})^2 \, \alpha_4(k'2^{j_1+4}r,2^{j-4}r)^2 \, H^{int}(j_1,k') \label{r_01stsum}\\ 
& + \sum_{d=2}^{k'} \sum_{j_1, \cdots, j_d < j} \sum_{ k_1, \cdots, k_d \in \N_+ : \atop k_1 + \cdots + k_d = k'} \prod_{i=1}^{d} \Bigg( \grandO{1} (k'2^{j-j_i})^2 \, \alpha_4(k_i2^{j_i+4}r,2^{j-4}r)^2 \, H^{int}(j_i,k_i) \label{r_02ndsum} \\
& + \grandO{1} k' \, 2^{j-j_i} \, \alpha_4(k_i2^{j_i+4}r,2^{j-4}r) \, \left( \frac{2^{j-4}r}{k_i2^{j_i+4}r} \right)^{-\epsilon} \, H^{r_0}(j_i,k_i) \Bigg) \, .\label{r_02ndsumbis}
\end{align}

By using~\eqref{recintetc} (with $\natural = int$),~\eqref{sameorder}, the quasi-multiplicativity property, and~\eqref{poly}, we obtain that the first sum of the above inequality (i.e. the right-hand side of~\eqref{r_01stsum}) is at most:
\begin{align}
& \grandO{1} k'^{\grandO{1}} \sum_{j_1 \leq j} \, \left( 2^{j-j_1} \, \alpha_4(2^{j_1}r,2^jr) \right)^2 \, g'(k') \, \bar{\gamma}_r(2^{j_1}r) \nonumber\\
& \leq \grandO{1} k'^{\grandO{1}} \sum_{j_1 \leq j} \frac{\bar{\gamma}_r(2^jr)}{\bar{\gamma}_r(2^{j_1}r)} \, g'(k') \, \bar{\gamma}_r(2^{j_1}r) \nonumber\\
& \leq \grandO{1} k'^{\grandO{1}} g'(k') \, j \, \bar{\gamma}_r(2^jr) \, . \label{e.usingtechnicaleps}
\end{align}
Let $c > 0$ be the constant of~\eqref{technicalepsilon2}. In terms of $\bar{\gamma}$ and ${\bar{\gamma}}^{r_0}$,~\eqref{technicalepsilon2} can be stated as follows:
\[
\bar{\gamma}_r(2^jr) \leq \grandO{1} {\bar{\gamma}}^{r_0}_r(2^jr)^{1.01} 2^{-cj} \, .
\]
Hence,~\eqref{e.usingtechnicaleps} is at most:
\[
\grandO{1} k'^{\grandO{1}} g'(k') \, {\bar{\gamma}}^{r_0}_r(2^jr)^{1.01} \, j \, 2^{-cj}  \leq \grandO{1} k'^{\grandO{1}} g'(k') \, {\bar{\gamma}}^{r_0}_r(2^jr)^{1.01} \, ,
\]
which is smaller than or equal to:
\[
1/2 \left( g^{r_0}(k') \, {\bar{\gamma}^{r_0}}_r(2^jr) \right)^{1.01} \, ,
\]
if $\theta^{r_0}/\theta'$ is sufficiently large.
\medskip

Let us now concentrate on the second sum (i.e. the quantity of lines~\eqref{r_02ndsum} and~\eqref{r_02ndsumbis}). As above, we have:
\begin{equation}\label{termint}
(k'2^{j-j_i})^2 \, \alpha_4(k_i2^{j_i+4}r,2^{j-4}r)^2 \, H^{int}(j_i,k_i) \leq \grandO{1} k'^{\grandO{1}} g'(k_i) \, \bar{\gamma}_r(2^jr)  \, .
\end{equation}
If we use our induction hypothesis on the $H^{r_0}(j_i,k_i)$'s (with an exponent $1$ instead of $1.01$, see Remark~\ref{power}), we obtain that:
\begin{equation}\label{termr_0}
k' \, 2^{j-j_i} \, \alpha_4(k_i2^{j_i+4}r,2^{j-4}r) \, \left( \frac{2^{j-4}r}{k_i2^{j_i+4}r} \right)^{-\epsilon} \, H^{r_0}(j_i,k_i) \leq \grandO{1} k'^{\grandO{1}} g^{r_0}(k_i) {\bar{\gamma}}^{r_0}_r(2^jr) \, .
\end{equation}
The inequality~\eqref{technicalepsilon1} (which implies that $\bar{\gamma}_r(2^jr) \leq \grandO{1} {\bar{\gamma}}^{r_0}_r(2^jr)$) and the fact that $\theta^{r_0} > \theta'$ imply that the right-hand side of~\eqref{termr_0} is at least $\grandO{1} k'^{\grandO{1}}$ times the right-hand side of~\eqref{termint}. In other words, our estimate on the terms that come from the $r_0$-clusters dominate our estimates on the terms that come from the interior clusters. We obtain that the second sum is at most:
\begin{align*}
& \sum_{d=2}^{k'} \sum_{j_1, \cdots, j_d \leq j} \sum_{ k_1, \cdots, k_d \in \N_+: \atop k_1 + \cdots + k_d = k'} \prod_{i=1}^{d} \left( \grandO{1} k'^{\grandO{1}} g^{r_0}(k_i) {\bar{\gamma}}^{r_0}_r(2^jr) \right)\\
& \leq \sum_{d=2}^{k'} \left( \grandO{1} k'^{\grandO{1}} \, j \, {\bar{\gamma}}^{r_0}_r(2^jr) \right)^d  \sum_{ k_1, \cdots, k_d \in \N_+: \atop k_1 + \cdots + k_d = k'} \prod_{i=1}^{d} g^{r_0}(k_i) \, .
\end{align*}
Since $\log_2^2$ is concave and increasing, the above is at most:
\[
\sum_{d=2}^{k'} \left( \grandO{1} k'^{\grandO{1}} \, j \, {\bar{\gamma}}^{r_0}_r(2^jr) \right)^d \, {k'}^d \, g^{r_0}(k'/d)^d \leq \sum_{d=2}^{k'} \left( \grandO{1} k'^{\grandO{1}} \, j \, g^{r_0}(k'/2) \, {\bar{\gamma}}^{r_0}_r(2^jr) \right)^d \, .
\]
We can show that, if $\theta^{r_0}$ is sufficiently large, then the hypothesis $j > J^{r_0}(k)$ implies that:
\begin{equation}\label{petitlemma}
\grandO{1} k'^{\grandO{1}} j \, g^{r_0}(k'/2) \, {\bar{\gamma}}^{r_0}_r(2^jr) \leq 1/2 \, \left( g^{r_0}(k') \, {\bar{\gamma}}^{r_0}_r(2^jr) \right)^{0.505} (\leq 1/2) \, .
\end{equation}
(This is the exact analogue of Lemma~4.4 of~\cite{GPS} - with $\epsilon = 0.495$ - and we refer to this paper for more details.) So, the second sum is smaller than or equal to:
\begin{eqnarray*}
\frac{\left( 1/2 \, \left( g^{r_0}(k') \, {\bar{\gamma}}^{r_0}_r(2^jr) \right)^{0.505}\right)^2}{1-1/2 \, \left( g^{r_0}(k') \, {\bar{\gamma}}^{r_0}_r(2^jr) \right)^{0.505}} & \leq & \frac{(1/2)^2}{1-1/2} \, \left( g^{r_0}(k') \, {\bar{\gamma}}^{r_0}_r(2^jr) \right)^{2 \times 0.505}\\
& = & 1/2 \, \left( g^{r_0}(k') \, {\bar{\gamma}}^{r_0}_r(2^jr) \right)^{1.01}.
\end{eqnarray*}
Finally:
\[
H^{r_0}(j,k') \leq 2 \times 1/2 \, \left( g^{r_0}(k') \, {\bar{\gamma}}^{r_0}_r(2^jr) \right)^{1.01} = \left( g^{r_0}(k') \, {\bar{\gamma}}^{r_0}_r(2^jr) \right)^{1.01}.
\]

\bigskip

Now, let us prove \eqref{reccent}. Let $j \in J$ and $k' \in \N_+$. First note that we can take $h \in (0,1/4)$ such that, for all $1 \leq \rho_1 \leq \rho_2$:
\[
{\overline{\gamma}}^{r_0}_{\rho_1}(\rho_2)^{1-2h} \leq \frac{1}{h} \gamma^*_{\rho_1}(\rho_2)^h \, .
\]

Consider some $B$ as in the definition of $H^*(j,k')$. The square $B$ is the inner square of the annulus associated to some centered cluster $C$ at level $j$ with $|C|=k'$. If $j=0$ and $k' \in \N_+$ then we are done (and more generally if $j \leq J^*(k')$). We assume that $j>J^*(k')$ and we prove the result by induction on $j$. Let $C_1, \cdots, C_d$ be the children of $C$. Note that $d \geq 2$ and that exactly one of the $C_i$'s is centered, say that it is $C_1$.

Remember Remark~\ref{power}: the induction hypothesis implies that the following is true for any $k'' \in \N_+$ and $j' \in \lbrace 0, \cdots, j-1 \rbrace$:
\begin{eqnarray}
H^*(j',k'') & \leq & \left( g^*(k'') \, \gamma^*_{r_0}(2^{j'}r) \, {\overline{\gamma}}^{r_0}_r(2^{j'}r) \right)^{1-h} \nonumber\\
& = & g^*(k'')^{1-h} \, \gamma^*_{r_0}(2^{j'}r)^{1-h} \, {\overline{\gamma}}^{r_0}_r(2^{j'}r)^{1-2h} \, {\overline{\gamma}}^{r_0}_r(2^{j'}r)^{h} \nonumber \\
& \leq & g^*(k'')^{1-h} \, \frac{1}{h} \gamma^*_{r_0}(2^{j'}r) \, {\overline{\gamma}}^{r_0}_r (2^{j'}r)^h \, . \label{reccentavech}
\end{eqnarray}
(In the last line we have used that $r \leq r_0$.) Note that, for any $i \geq 2$, $C_i$ is either an interior cluster or a $r_0$-cluster. As above, thanks to~\eqref{technicalepsilon1}, our estimates on the $r_0$-clusters dominate our estimates on the interior clusters. Let us also recall that the way to associate an annulus to the singleton $\lbrace \lbrace 0 \rbrace \rbrace$ is different from the other clusters, that is why ``$\vee (r_0+2)$" appears in the estimate below. We have:
\begin{align*}
& H^{*}(j,k') \leq \sum_{d=2}^{k'} \sum_{j_1, \cdots, j_d < j} \sum_{ k_1, \cdots, k_d \in \N_+: \atop k_1 + \cdots + k_d = k'} \grandO{1} \, \alpha_1 \left( (k_12^{j_1+4}r ) \vee (r_0+2) ,2^{j-4}r \right)^2 \, H^*(j_1,k_1)\\
& \times \prod_{i=2}^{d} \left( \grandO{1} k' \, 2^{j-j_i} \, \alpha_4(k_i2^{j_i+4}r,2^{j-4}r) \, \left( \frac{2^{j-4}r}{k_i2^{j_i+4}r} \right)^{-\epsilon} \, \left( g^{r_0}(k_i) \, {\overline{\gamma}}^{r_0}_r(2^{j_i}r) \right)^{1.01} \right)
\end{align*}
(the second line of the expression above comes from the fact that the estimates on the $r_0$-clusters dominate our estimates on the interior clusters; the term $\left( g^{r_0}(k_i) \, {\overline{\gamma}}^{r_0}_r(2^{j_i}r) \right)^{1.01}$ comes from~\eqref{recr_0}). We continue the calculation: by using~\eqref{reccentavech} to deal with $H^*(j_1,k_1)$ (and also by using that $\alpha_1 \left( (k_12^{j_1+4}r ) \vee (r_0+2) ,2^{j-4}r \right)^2 \, \gamma^*_{r_0}(2^{j_1}r) \leq \grandO{1} k_1^{\grandO{1}} \gamma^*_{r_0}(2^{j}r)$), we find that the above is at most:
\begin{align*}
& \sum_{d=2}^{k'} \sum_{j_1, \cdots, j_d < j} \sum_{ k_1, \cdots, k_d \in \N_+: \atop k_1 + \cdots + k_d = k'} \grandO{1} k^{\grandO{1}}_1 \, g^*(k_1)^{1-h} \, \frac{1}{h} \, \gamma^*_{r_0}(2^{j}r) \, {\overline{\gamma}}^{r_0}_r(2^{j_1}r)^h\\
& \times \prod_{i=2}^{d} \left( \grandO{1} k'^{\grandO{1}} \, g^{r_0}(k_i)^{1.01} \, {\overline{\gamma}}^{r_0}_r(2^{j}r) \, {\overline{\gamma}}^{r_0}_r(2^{j_i}r)^{0.01} \right)\\
& \leq \gamma^*_{r_0}(2^{j}r) \, g^*(k')^{1-h} \sum_{d=2}^{k'} \left( \sum_{j_1 \in \N} {\overline{\gamma}}^{r_0}_r(2^{j_1}r)^h \right)\\
& \times \sum_{ k_1, \cdots, k_d \in \N_+: \atop k_1 + \cdots + k_d = k'} \prod_{i=2}^{d} \left(  \grandO{1} k'^{\grandO{1}} \, g^{r_0}(k_i)^{1.01} {\overline{\gamma}}^{r_0}_r(2^{j}r) \, \sum_{j' \in \N} {\overline{\gamma}}^{r_0}_r(2^{j'}r)^{0.01} \right)\\
& \leq \gamma^*_{r_0}(2^{j}r) g^*(k')^{1-h} \sum_{d=2}^{k'} \left( \grandO{1} k'^{\grandO{1}} \, g^{r_0}(k')^{1.01} \, {\overline{\gamma}}^{r_0}_r(2^{j}r) \right)^{d-1}
\end{align*}
(since $\sum_{j_1 \in \N} {\overline{\gamma}}^{r_0}_r(2^{j_1}r)^h \leq \grandO{1}$ and $\sum_{j' \in \N} {\overline{\gamma}}^{r_0}_r(2^{j'}r)^{0.01} \leq \grandO{1}$). Next, note that, if $\theta^*/\theta^{r_0}$ is sufficiently large, then the hypothesis $j > J^*(k')$ implies that:
\[
\grandO{1} k'^{\grandO{1}} \, g^{r_0}(k')^{1.01}  \, {\overline{\gamma}}^{r_0}_r(2^{j}r) \leq 1/2
\]
(indeed, there exists $a > 0$ such that, if $j > J^*(k')$, then $\frac{1}{a} \, g^*(k') \, 2^{-aj} = \frac{1}{a} \, {g^{r_0}(k')}^{\theta^*/\theta^{r_0}} \, 2^{-aj}$ is smaller than or equal to $1$). As a result, if $\theta^*/\theta^{r_0}$ is sufficiently large, then:
\[
H^{*}(j,k') \leq \gamma^*_{r_0}(2^{j}r) g^*(k')^{1-h} \grandO{1} k'^{\grandO{1}} \, g^{r_0}(k')^{1.01}  \, {\overline{\gamma}}^{r_0}_r(2^{j}r) \, .
\]
Now, note that, again if $\theta^*/\theta^{r_0}$ is sufficiently large, we have:
\[
\grandO{1} k'^{\grandO{1}} \, g^{r_0}(k')^{1.01} \leq g^*(k')^h \, ,
\]
hence:
\[
H^*(j,k') \leq \gamma^*_{r_0}(2^{j}r) \, g^*(k') \, {\overline{\gamma}}^{r_0}_r(2^{j}r) \, ,
\]
which is what we want.

\paragraph{C.2. End of the proof.}

All that remains to prove is that~\eqref{recintetc},~\eqref{recr_0} and~\eqref{reccent} imply Proposition~\ref{smallspectrumbis}. Remember that it is sufficient to prove~\eqref{lastlast}. Remember also the definition of $\overline{j}$. By using the quasi-multiplicativity property and~\eqref{poly}, we obtain that it is sufficient to prove that there exists an absolute constant $\theta < +\infty$ such that, if: 
\begin{equation}\label{e.new_cond}
g(k) \, \sqrt{\gamma^*_{r_0}(2^{\overline{j}}r)} \, {\overline{\gamma}}^{r_0}_r(r_0) \leq 1 \, ,
\end{equation}
then:
\begin{equation}\label{e.ttffhh}
\sum_{\widetilde{\mathcal{A}} \in \mathfrak{U}_k} 4^{|\widetilde{\mathcal{A}}|} \, h(\widetilde{\mathcal{A}})^2 \leq g(k) \, \gamma^*_{r_0}(2^{\overline{j}}r) \, {\overline{\gamma}}^{r_0}_r(r_0) \, .
\end{equation}
Assume that~\eqref{e.new_cond} holds, let $S$ be some set such that $|S_r^{r_0}|=k$, and let $C$ be the top cluster of $S$. Also, let $C_1, \cdots, C_d$ be the children of $C$. Note that exactly one of the $C_i$'s is centered, say that it is $C_1$. The other $C_i$'s can be of any other type (and $d$ may equal $1$). As above (and thanks to~\eqref{alpha3+}), our estimates on the $r_0$-clusters will dominate the estimates on the interior, side and corner clusters. Note also that, if $d =1$, then $C_1$ contains $\lbrace 0 \rbrace$ and also at least one $r \times r$ square not included in $(-r_0,r_0)^2$, so $\log_2(r_0/(2r)) \leq j(C_1)$.
We have (the terms $k+1$ and $k_1+1$ come only from the fact that $|C| = |S_r^{r_0} \cup \lbrace 0 \rbrace|=k+1$):
\begin{align*}
& \sum_{\widetilde{\mathcal{A}} \in \mathfrak{U}_k} 4^{|\widetilde{\mathcal{A}}|} \, h(\widetilde{\mathcal{A}})^2 \leq \sum_{\log_2 (r_0/(2r)) \leq j_1 \leq \overline{j}} \left( \alpha_1((k +1)2^{j_1+4} r, 2^{\overline{j}+1-4})^2 \, H^*(j_1,k+1) \right)\\
& + \sum_{d=2}^{k} \sum_{k_1, \cdots, k_d  : \atop k_1 + \cdots + k_d = k} \sum_{j_1, \cdots, j_d \leq \overline{j}} \Big( \alpha_1 \left( \left( (k_1+1) 2^{j_1+4} r \right) \vee (r_0+2), 2^{\overline{j}+1-4}r \right)^2 H^*(j_1,k_1+1)\\
& \times \prod_{i=2}^{d} \Bigg( \grandO{1} k^{\grandO{1}} \, \frac{R}{2^{j_i}r} \, \alpha_4(k_i 2^{j_i+4} r, 2^{\overline{j}+1-4}r) \, \left( \frac{2^{\overline{j}+1-4}}{k_i2^{j_i+4}} \right)^{-\epsilon} \left( g^{r_0}(k_i) \, {\overline{\gamma}}^{r_0}_r(2^{j_i}r) \right)^{1.01} \Big) \Bigg) \, .
\end{align*}
We now use~\eqref{recintetc},~\eqref{recr_0} and~\eqref{reccent} to conclude. The first sum above is smaller than or equal to:
\begin{align*}
& \sum_{\log_2(r_0/(2r)) \leq j_1 \leq \bar{j}} \grandO{1} k^{\grandO{1}} \, \gamma^*_{r_0}(2^{\bar{j}}r) \, g^*(k) \, {\bar{\gamma}}^{r_0}_r(2^{j_1}r)\\
& \leq \grandO{1} k^{\grandO{1}} \, g^*(k) \, \gamma^*_{r_0}(2^{\bar{j}}r) \, \sum_{\log_2(r_0/(2r)) \leq j_1} \, {\bar{\gamma}}^{r_0}_r(2^{j_1}r) \\
& \leq \grandO{1} k^{\grandO{1}} \, g^*(k) \, \gamma^*_{r_0}(2^{\bar{j}}r) \,{\bar{\gamma}}^{r_0}_r(r_0)\\
& \leq 1/2 \, g(k) \, \gamma^*_{r_0}(2^{\bar{j}}r) \, {\bar{\gamma}}^{r_0}_r(r_0) \, ,
\end{align*}
if $\theta/\theta^*$ is sufficiently large. Let us now estimate the second sum. This sum is smaller than or equal to:
\begin{align*}
& \sum_{d=2}^{k} \sum_{k_1, \cdots, k_d  : \atop k_1 + \cdots + k_d = k} \sum_{j_1, \cdots, j_d \leq \bar{j}} \grandO{1} k^{\grandO{1}} \, \gamma^*_{r_0}(2^{\bar{j}}r) \, g^*(k) \, {\bar{\gamma}}^{r_0}_r(2^{j_1}r)\\
& \times \left( \prod_{i=2}^{d} \grandO{1} k^{\grandO{1}} \, g^{r_0}(k_i)^{1,01} \, {\bar{\gamma}}^{r_0}_r(2^{\bar{j}}r) \, {\bar{\gamma}}^{r_0}_r(2^{j_i}r)^{0.01} \right)\\
& \leq g^*(k) \, \gamma^*_{r_0}(2^{\bar{j}}r) \, \sum_{d=2}^{k} \left( \grandO{1} k^{\grandO{1}} \left( \sum_{j' \in \N} {\bar{\gamma}}^{r_0}_r(2^{j'}r)^{0.01} \right)  g^{r_0}(k)^{1.01} \, {\bar{\gamma}}^{r_0}_r(2^{\bar{j}}r) \right)^{d-1}\\
& \leq g^*(k) \, \gamma^*_{r_0}(2^{\bar{j}}r) \, \sum_{d=2}^{k} \left( \grandO{1} k^{\grandO{1}} \, g^{r_0}(k)^{1.01} \, {\bar{\gamma}}^{r_0}_r(2^{\bar{j}}r) \right)^{d-1} \, .
\end{align*}
Note that, if $\theta/\theta^{r_0}$ is sufficiently large, then~\eqref{e.new_cond} and the fact that $r_0 \leq R/2 \leq \grandO{1} k^{\grandO{1}} 2^{\overline{j}}r$ imply that:
\[
\grandO{1} k^{\grandO{1}} \, g^{r_0}(k)^{1.01} \, {\bar{\gamma}}^{r_0}_r(2^{\bar{j}}r) \leq 1/2 \, ,
\]
hence the second sum is at most:
\[
g^*(k) \, \gamma^*_{r_0}(2^{\bar{j}}r) \, \grandO{1} k^{\grandO{1}} \, g^{r_0}(k)^{1.01} \, {\bar{\gamma}}^{r_0}_r(2^{\bar{j}}r) \, .
\]
By using once again that $r_0 \leq \grandO{1} k^{\grandO{1}} 2^{\overline{j}}r$, we obtain that the above is at most:
\[
g^*(k) \, \gamma^*_{r_0}(2^{\bar{j}}r) \, \grandO{1} k^{\grandO{1}} \, g^{r_0}(k)^{1.01} \, {\bar{\gamma}}^{r_0}_r(r_0) \, .
\]
Next, note that, if $\theta/\theta^{r_0}$ and $\theta/\theta^*$ are sufficiently large, then:
\[
g^*(k) \, \grandO{1} k^{\grandO{1}} \, g^{r_0}(k)^{1.01} \leq 1/2 \, g(k) \, ,
\]
and the second sum is at most:
\[
1/2 \, g(k) \, \gamma^*_{r_0}(2^{\bar{j}}r) \, {\bar{\gamma}}^{r_0}_r(r_0) \, .
\]
Finally:
\[
\sum_{\widetilde{\mathcal{A}} \in \mathfrak{U}_k} 4^{|\widetilde{\mathcal{A}}|} \, h(\widetilde{\mathcal{A}})^2
\leq 2 \times 1/2 \, g(k) \, \gamma^*_{r_0}(2^{\bar{j}}r) \, {\bar{\gamma}}^{r_0}_r(r_0) = g(k) \, \gamma^*_{r_0}(2^{\bar{j}}r) \, {\bar{\gamma}}^{r_0}_r(r_0) \, .
\]
This ends the proof of~\eqref{e.ttffhh} and therefore of~Proposition~\ref{smallspectrumbis}.
\end{proof}

\subsubsection{The proof of Theorem~\ref{soimportant}}\label{ss.finalproof}

We now combine Propositions~\ref{weakindependence}, \ref{largedev} and~\ref{smallspectrumbis} in order to prove Theorem~\ref{soimportant}.

\begin{proof}[Proof of Theorem~\ref{soimportant}] The proof is very similar to the proof of Theorem~$7.3$ of~\cite{GPS} (which is Theorem~\ref{GPScool} in our paper). Consequently, we will omit some details. The main difference is that we have to use Proposition~\ref{smallspectrumbis} instead of the estimate~\eqref{verysmallspectrum}.

Note that we can assume that $r \geq \overline{r}$ for some absolute constant $\overline{r}$ (if $r < \overline{r}$ then Theorem~\ref{soimportant} is a direct consequence of Proposition~\ref{smallspectrumbis} applied itself with $r=1$ since $| S^{r_0}_1 | \asymp | S \setminus (-r_0,r_0)^2 |$). Let $(B_i)_i$ be a tiling of the annulus $[-(R+2),R+2]^2 \setminus (-r_0,r_0)^2$ by $l_i \times l'_i$ rectangles with $r \leq l_i, \, l'_i \leq 2r$ for every $i$ (for instance, we can tile with $r \times r$ squares expect near $\partial [-(R+2),(R+2)]^2$ where we use rectangles so that we perfectly tile the annulus). Also, let $\mathcal{S} = \mathcal{S}_{f_R}$ be a spectral sample of $f_R$ and $\mathcal{Z}=\mathcal{Z}_r$ be a random subset of $\lbrace i \in \mathcal{I}_R \, : \, i \notin (-r_0,r_0)^2 \rbrace$ that is independent of $\mathcal{S}_{f_R}$, where each element of $\lbrace i \in \mathcal{I}_R \, : \, i \notin (-r_0,r_0)^2 \rbrace$ is in $\mathcal{Z}$ with probability $\frac{1}{r^2\alpha_4(r)}$ independently of the others. 

It is sufficient to prove that (for some $\epsilon > 0$ and $C<+\infty$):
\[
\Pro \left[ \mathcal{S} \cap \mathcal{Z} = \emptyset \neq \mathcal{S} \setminus (-r_0,r_0)^2 \right] \leq C \frac{\alpha_1(R)}{\alpha_1(r_0)}  \left( \frac{r_0}{r} \right)^{1-\epsilon} \alpha_4(r,r_0) \, .
\]

We first assume that $r < r_0 / 4$. Let $x_i$ be the indicator function of $\left\lbrace \mathcal{S} \cap B_i \neq \emptyset \right\rbrace$ and $y_i$ the indicator function of $\left\lbrace \mathcal{S} \cap B_i \cap \mathcal{Z} \neq \emptyset \right\rbrace$. Let $X$ and $Y$ be as in Proposition~\ref{largedev}. The hypothesis of Proposition~\ref{largedev} is given by Proposition~\ref{weakindependence}. We will also use that $\frac{1}{C_1} |\mathcal{S}_r^{r_0}| \leq X \leq C_1 |\mathcal{S}_r^{r_0}|$ for some absolute constant $C_1 \in (0,+\infty)$. By using Proposition~\ref{largedev}, we obtain:
\begin{eqnarray*}
\Pro \left[ \mathcal{S} \cap \mathcal{Z} = \emptyset \neq \mathcal{S} \setminus (-r_0,r_0)^2 \right] & = & \Pro \left[ Y = 0 < X \right]\\
& \leq & \frac{1}{a} \E \left[ \exp(-aX/e) \un_{X > 0}  \right]\\
& \leq &  \frac{1}{a} \E \left[ \exp(-a|\mathcal{S}_r^{r_0}|/(C_1e)) \un_{\mathcal{S}_r^{r_0} \neq \emptyset}  \right] \, .
\end{eqnarray*}
Then, Proposition~\ref{smallspectrumbis} implies that:
\begin{eqnarray*}
\Pro \left[ \mathcal{S} \cap \mathcal{Z} = \emptyset \neq \mathcal{S} \setminus (-r_0,r_0)^2 \right] & \leq & \frac{1}{a} \sum_{k \in \N_+} \exp(-ak/(C_1e)) \, g(k) \, \frac{\alpha_1(R)}{\alpha_1(r_0)} \left( \frac{r_0}{r} \right)^{1-\epsilon} \alpha_4(r,r_0)\\
& = & C_2 \frac{\alpha_1(R)}{\alpha_1(r_0)} \left( \frac{r_0}{r} \right)^{1-\epsilon} \alpha_4(r,r_0) \, ,
\end{eqnarray*}
for some constant $C_2 < +\infty$ since $g$ is sub-exponential.\\

We now assume that $r \geq r_0/4$. We keep the notations $X$ and $Y$. The problem is that we cannot use Proposition~\ref{weakindependence} for the rectangles $B_i$ that intersect $[-4r,4r]^2$. However, the number of such rectangles is at most $100$. We write $J = \lbrace i \, : \, B_i \text{ does not intersect} [-4r,4r]^2 \rbrace$, $\widetilde{X} = \sum_{i \notin J} x_i$ and $\widetilde{Y} = \sum_{i \notin J} y_i$. We have:
\begin{eqnarray*}
\Pro \left[ \mathcal{S} \cap \mathcal{Z} = \emptyset \neq \mathcal{S} \setminus (-r_0,r_0)^2 \right] & = & \Pro \left[ Y = 0 < X \right]\\
& \leq & \Pro \left[ \, 0 < X \leq 100 \right] + \Pro \left[ \widetilde{Y} = 0 < \widetilde{X} \right] \, .
\end{eqnarray*}
By using Proposition~\ref{largedev} for $\widetilde{X}$ and $\widetilde{Y}$, we obtain that the above is at most:
\begin{multline*}
\Pro \left[ \, 0 < X \leq 100 \right] + \frac{1}{a} \E \left[ \exp(-a\widetilde{X}/e) \un_{\widetilde{X} > 0}  \right]\\
\leq \Pro \left[ \, 0 < X \leq 100 \right] + \frac{1}{a} \E \left[ \exp(-a(X-100)/e) \un_{X > 0}  \right] \, .
\end{multline*}
We now conclude as in the case $r<r_0/4$, i.e. by using the fact that $X \asymp | S^{r_0}_r |$ and Proposition~\ref{smallspectrumbis}.
\end{proof}

\begin{rem}\label{aquantitaiveepsilon}
In this remark, we try to be more quantitative about the $\epsilon$ of Theorem~\ref{soimportant} in the case of site percolation on $\T$, by using the computation of the arm-exponent (see Subsection~\ref{sectionarm}):

As pointed out in Remark~\ref{theexplanation}, we can replace all the exponents $1.01$ that appear in the proof of Proposition~\ref{smallspectrumbis} by $1+a$ for any $a \in (0,1)$. Now, remember that the conditions about $\epsilon$ were that~\eqref{technicalepsilon1} and~\eqref{technicalepsilon2} are satisfied and that:
\[
\E_{1/2} \left[ \Pro_{1/2} \left[ \text{\textbf{A}}_4(r_1,r_2) \cond \mathcal{F}_{\half} \right]^2 \right] \leq \grandO{1} \, \alpha_4(r_1,r_2) \, \left( \frac{r_2}{r_1} \right)^{-\epsilon} \, ,
\]
where $\text{\textbf{A}}_4(r_1,r_2)$ is the $4$-arm event in  the annulus $[-r_2,r_2] \setminus (-r_1,r_1)^2$ and $\half$ is the (lower, upper, left or right) half-plane. In the case of site percolation on $\T$, this last condition holds for any $\epsilon < -5/4 + \zeta_4^{|\half}$ (see Proposition~\ref{exponenthalfplane}). Finally, by combining this with~\eqref{technicalepsilon1} and~\eqref{technicalepsilon2}, we obtain that Proposition~\ref{smallspectrumbis} is true for any $\epsilon < (-5/4 + \zeta_4^{|\half}) \wedge 1/4$. Consequently, this is also the case for Theorem~\ref{soimportant}.
\end{rem}
\medskip

\begin{rem}\label{thepbbis}
The fact that we had to deal with $|S^{r_0}_r|$ instead of $|S_r|$ in Proposition~\ref{smallspectrumbis} \textbf{is a new difficulty} compared to the proof of \eqref{verysmallspectrum} in Section~$4$ of~\cite{GPS}, and that is the reason why we had to introduce the $r_0$-decorated centered annulus structures and deal with the ``$4$-arm event conditioned on the configuration in a half-plane". Now, imagine that we want to deal with Conjecture~\ref{conjquiseraitcool} (stated below) instead of Theorem~\ref{soimportant}. Then, we would have to consider the whole spectral sample so we would not need the notion of $r_0$-decorated centered annulus structures any more, but only the notion of centered annulus structures of~\cite{GPS}. Let us be more precise: If we follow~\cite{GPS} (Subsection~$4.4$) with another choice for the definition of overcrowded centered clusters (a centered cluster would be overcrowded if $\gamma^*_{r}(2^{j(C)}r) g^*(|C|) \overline{\gamma}_r(2^{j(C)}r) > 1$) then we would obtain the following:

For every $S \subseteq \mathcal{I}_R$, let $S_r$ be the set of the $r \times r$ squares of the grid $r\Z^2$ which intersect $S$. There exists some $\theta < +\infty$ such that, for all $k \in \N_+$ and for all $1 \leq r \leq r_0 \leq R/2$:
\[
\widehat{\Pro}_{f_R} \left[ |S_r|=k, \, S \nsubseteq (-r_0,r_0)^2 \right] \leq g(k) \frac{\alpha_1(R)}{\alpha_1(r)} \left( \frac{r_0}{r} \alpha_4(r,r_0) \right)^2 \, ,
\]
where $g(k)=2^{\theta \log_2^2(k+2)}$.

With this result, it seems that we can prove Conjecture~\ref{conjquiseraitcool} exactly as we have proved Theorem~\ref{soimportant} i.e. by using Propositions~\ref{weakindependence} and~\ref{largedev}. Here, $x_i$ would be the indicator functions of the event $\left\lbrace \mathcal{S}_{f_R} \cap B_i \neq \emptyset \, , \, S \nsubseteq (-r_0,r_0)^2 \right\rbrace$, and $y_i$ would be the indicator function of the event $\left\lbrace \mathcal{S}_{f_R} \cap B_i \cap \mathcal{Z} \neq \emptyset \, , \, S \nsubseteq (-r_0,r_0)^2 \right\rbrace$. However, this strategy does not work since the event $\left\lbrace  S \nsubseteq (-r_0,r_0)^2 \right\rbrace$ represents some \textbf{positive information} when we visit the rectangles $B_i$ that are included in $(-r_0,r_0)^2$. Therefore, Proposition~\ref{weakindependence} can no longer guarantee that the hypothesis of Proposition~\ref{largedev} is true. So, in order to prove Conjecture~\ref{conjquiseraitcool}, we would need an analogue of Proposition~\ref{weakindependence} where in the conditioning we can add the event $\lbrace S \nsubseteq (-r_0,r_0)^2 \rbrace$. Techniques from \cite{GPS} are not suitable for such a conditioning, which explains why we cannot prove at the moment the better upper-bound given by Conjecture~\ref{conjquiseraitcool}. 
\end{rem}

\section{Open questions}

Here is a list of a few open problems. 

\bnum
\item \textbf{Asymmetric exclusion dynamics.} Our hypothesis that the underlying exclusion dynamics is symmetric (i.e. $K$ is symmetric) is crucial in our proofs. Indeed the duality formula~\eqref{e.Kt} is no longer valid in the asymmetric setting. A natural question is thus to ask whether the results of the present article still hold by relaxing the symmetry condition. 

\item \textbf{Handling more local dynamics.} Our techniques brake when $\alpha$ becomes too large (the best value of $\alpha$ can be found in equation~\eqref{e.alphamax}). The most extreme (and most interesting) case would be the nearest-neighbour simple exclusion process. We are very far at this point of being able to prove the existence of exceptional times in this case.

\item \textbf{Sharp clustering effect for the radial spectral sample.} In the proof of Corollary~\ref{soimportantcor} (which is our key estimate in Theorem \ref{tropbien}), we used the crude upper-bound:
\[
\widehat{\Pro}_{f_R} \left[ |S| < r^2 \alpha_4(r), \, S \nsubseteq (-r_0,r_0)^2 \right] \leq \widehat{\Pro}_{f_R} \left[ 0 < |S \setminus (-r_0,r_0)^2| < r^2 \alpha_4(r) \right].
\]
We believe we have lost a lot in this inequality and we make the following conjecture (see Remark~\ref{thepbbis}):
\begin{conj}\label{conjquiseraitcool}
There exists a constant $C<+\infty$ such that, for all $1 \leq r \leq r_0$ and all $R \geq 1$:
\[
\widehat{\Pro}_{f_R} \left[ |S| < r^2 \alpha_4(r), \, S \nsubseteq (-r_0,r_0)^2 \right] \leq C \frac{\alpha_1(R)}{\alpha_1(r)} \left( \frac{r_0}{r} \, \alpha_4(r,r_0) \right)^2 \, .
\]
\end{conj}
Note that if we had proved this conjecture then we would have obtained a bigger $\alpha_0$ in Theorem~\ref{tropbien}.

\item \textbf{Clustering effect for left-right crossing events.} One of the main side technical contributions of this paper is our clustering result Theorem \ref{soimportant}. Even though it does not give a sharp estimate on $\widehat{\Pro}_{f_R} \left[ |S| < r^2 \alpha_4(r), S \nsubseteq (-r_0,r_0)^2 \right]$ as discussed in the item above, it provides the first  polynomial clustering estimate on the spectral sample of the one-arm event $f_R$. Indeed, such a {\em clustering effect} had already been analysed in \cite{GPS}, but it only gave rather weak (logarithmic) bounds. See Remark 4.5 in \cite{GPS}. Now, if $g_n$ is the indicator function of the left-right crossing of $[-n,n]^2$, a similar polynomial clustering effect should hold as $n \to +\infty$. More precisely, the following analogue of Conjecture~\ref{conjquiseraitcool} for the functions $g_n$ should hold:
\begin{conj}\label{c.LR}
There exists a constant $C < +\infty$ such that, for all $1 \leq r \leq r_0$ and all $n \geq 1$:
\[
\widehat{\Pro}_{g_n} \left[ |S| < r^2 \alpha_4(r), \, \text{\textup{diam}}(S) \geq r_0 \right] \leq C \left( \frac{n}{r} \, \alpha_4(r,n) \right)^2 \left( \frac{r_0}{r} \, \alpha_4(r,r_0) \right)^2 \, .
\]
\end{conj}
Somewhat surprisingly, it turns out to be easier for such clustering effects to deal with degenerate Boolean functions such as $f_R$ rather than left-right crossing events $g_n$. This is due to the fact that we know in the case of $\widehat{\Pro}_{f_R}$ that the spectral sample will most likely localize in a ball centered at the origin. The additional flexibility corresponding to where in $[-n,n]^2$ the spectral sample of $g_n$ will choose to localize adds new difficulties.

%
\enum

\appendix

\section{Appendix: graphical construction (\`a la Harris) of exclusion dynamics}\label{a.graphical}

In this appendix, we give a proper graphical construction of the exclusion dynamics we need. This is in the spirit of the graphical constructions of particle systems initiated by Harris, see for example \cite{Harris}. The content of this appendix is very basic and will probably be considered ``folklore'' by specialists. Yet, as we could not localize a reference, we include it here. 

Let us then define properly the $K$-exclusion process, for instance for dynamics on the edges of a graph $G=(V,E)$. First, we sample a percolation configuration $\omega_K(0)$ according to some initial law. Next, to each pair $\left\lbrace e,f \right\rbrace$ of edges, we associate an exponential clock of parameter $K(e,f)=K(f,e)$ (independent of the others and $\omega_K(0)$). We define the càdlàg process $(\omega_K(t))_{t \geq 0}$ on the space $\Omega:=\left\lbrace -1,1 \right\rbrace ^{E}$ (seen as the compact metric product space) as follows:
\begin{enumerate}
\item First, we want to define a (random) dynamical permutation of $E$. Take $e \in E$. Let $\tau_1(e)$ be the first time a clock associated to $e$ (i.e. the clock associated to $\lbrace e,f \rbrace$ for some edge $f$) has rung and let $e_1$ be the other edge associated to this clock. Define recursively $\tau_{n+1}(e)$ to be the first time larger than $\tau_n(e)$ such that a clock associated to $e_n$ has rung and let $e_{n+1}$ be the other edge associated to this clock. Now, for each $t \geq 0$, let $n_t(e) = \sup \left\lbrace n : \tau_n(e) \leq t \right\rbrace$ (with $\sup \emptyset = 0$) and let $\pi_t$ be the (random) permutation of $E$ defined by:
\begin{equation}\label{defiofpi}
\pi_t(e) = e_{n_t(e)}
\end{equation}
(with $e_0:=e$).

Note that a.s., for all $t$ and $e$, $\pi_t(e)$ is well defined since a.s. for all $t$ and all $e$:
\begin{enumerate}
\item there exists at most one edge $e'$ such that the clock associated to $\lbrace e,e' \rbrace$ has rung at time $t$,
\item $n_t(e)$ is finite.
\end{enumerate}
Let us prove that a.s. $\pi_t$ is indeed a permutation:
\item To this purpose, we define a function that will turn out to be the reciprocal function of $\pi_t$. To do so, we follow the same steps as for the definition of $\pi_t$ but we start from time $t$ and look back in time. More precisely, if $e$ is some edge, we denote by $\widehat{\tau}_1(e)$ the largest time less than or equal to $t$ such that a clock associated to $e$ has rung. If such a time does not exist, we write $\widehat{\tau}_1(e)= - \infty$. Otherwise, we write $\widehat{e}_1$ for the other edge associated to this clock. Then, recursively, if $\widehat{\tau}_n(e) \neq - \infty$, we write $\widehat{\tau}_{n+1}(e)$ for the largest time less than $\widehat{\tau}_n(e)$ such that a clock associated to $\widehat{e}_n$ has rung. If such a time does not exist, we write $\widehat{\tau}_{n+1}(e)=-\infty$. Otherwise, we write $\widehat{e}_{n+1}$ for the other edge associated to the clock. Let $\widehat{n}_t(e)$ be the first $k \in \N_+$ such that $\widehat{\tau}_k(e)=- \infty$. It is not difficult to show that a.s. for all $t$ the function $e \mapsto \widehat{e}_{\widehat{n}_t(e)-1}$ (with $\widehat{e}_0:=e$) is well defined and is the reciprocal function of $\pi_t$.
\item Now, we can define $\omega_K(t)$ stating that the state of the edge $\pi_t(e)$ at time $t$ is the state of the edge $e$ at time $0$. In other words, the configuration at time $t$ is:
\[
\omega_K(t) : e \longmapsto \omega(0)_{\pi_t^{-1}(e)} \, .
\]
\end{enumerate}

Using item~$2$ above (that defines explicitly $\pi_t^{-1}(e)$), it is not difficult to show that we have obtained a  c\`adl\`ag Markov process and that the probability measures $\Pro_p$ are invariant measures for this process.

\section{The proof of Lemma~\ref{JP}}\label{a.B}

In this appendix, we prove Lemma~\ref{JP}. First, we prove by induction on $n \geq 1$ that:
\begin{align}
&\widehat{\Q}_h \left[ \forall j \in \lbrace 1, \cdots, n \rbrace, \, S \cap J_j \neq \emptyset \, , \,S \cap W = \emptyset \right]\nonumber\\
& = \sum_{k=0}^{n} (-1)^k \sum_{1 \leq j_1 < \ldots < j_k \leq n} \E_{1/2} \left[ \E_{1/2} \left[ h \cond \mathcal{F}_{\left( \cup_{i=1}^k J_{j_i} \cup W \right)^c} \right]^2 \right]\label{grosseexpression} \, .
\end{align}
If $n=1$ this is equation~(2.14) in the proof of Lemma~$2.2$ of~\cite{GPS}. We also follow the proof of this lemma for $n \geq 2$. Take $n \in \N_+$, assume that the result is true for any $J'_1, \cdots, J'_n$ and $W'$ mutually disjoint subsets of $\mathcal{I}_R$, and let $J_1, \cdots, J_{n+1}$ and $W$ be mutually disjoint subsets of $\mathcal{I}_R$. We have:
\begin{align*}
& \widehat{\Q}_h \left[ \forall j \in \lbrace 1, \cdots, n+1 \rbrace, S \cap J_j \neq \emptyset \, , \,S \cap W = \emptyset \right]\\
& = \widehat{\Q}_h \left[ \forall j \in \lbrace 1, \cdots, n \rbrace, S \cap J_j \neq \emptyset \, , \,S \cap W = \emptyset \right]\\
& \hspace{1cm} - \widehat{\Q}_h \left[ \forall j \in \lbrace 1, \cdots, n \rbrace, S \cap J_j \neq \emptyset \, , \,S \cap \left( W \cup J_{n+1} \right) = \emptyset \right]\\
& = \sum_{k=0}^{n} (-1)^k \sum_{1 \leq j_1 < \ldots < j_k \leq n} \E_{1/2} \left[ \E_{1/2} \left[ h \cond \mathcal{F}_{\left( \cup_{i=1}^k J_{j_i} \cup W \right)^c} \right]^2 \right]\\
& \hspace{1cm} - \sum_{k=0}^{n} (-1)^k \sum_{1 \leq j_1 < \ldots < j_k \leq n} \E_{1/2} \left[ \E_{1/2} \left[ h \cond \mathcal{F}_{\left( \cup_{i=1}^k J_{j_i} \cup J_{n+1} \cup W \right)^c} \right]^2 \right]\\
& = \sum_{k=0}^{n+1} (-1)^k \sum_{1 \leq j_1 < \ldots < j_k \leq n+1} \E_{1/2} \left[ \E_{1/2} \left[ h \cond \mathcal{F}_{\left( \cup_{i=1}^k J_{j_i} \cup W \right)^c} \right]^2 \right],
\end{align*}
and the induction is over.
\medskip

Now, we prove Lemma~\ref{JP}, also by induction on $n$.

If $n=1$, this is Lemma~2.2 of \cite{GPS}. We assume that Lemma~\ref{JP} holds for some $n \in \N_+$ and we want to prove it for $n+1$. Thanks to \eqref{grosseexpression}, it is sufficient to study the quantity $\sum_{k=0}^{n+1} (-1)^k \sum_{1 \leq j_1 < \ldots < j_k \leq n+1} \E_{1/2} \left[ \E_{1/2} \left[ h \cond \mathcal{F}_{\left( \cup_{i=1}^{k} J_{j_i} \cup W \right)^c} \right]^2 \right]$, which equals:
\begin{multline*}
\sum_{k=0}^{n} (-1)^k \sum_{1 \leq j_1 < \ldots < j_k \leq n} \Bigg( \E_{1/2} \left[  \E_{1/2} \left[ h \cond \mathcal{F}_{\left( \cup_{i=1}^{k} J_{j_i} \cup W \right)^c} \right]^2 \right]\\
- \E_{1/2} \left[ \E_{1/2} \left[ h \cond \mathcal{F}_{\left( \cup_{i=1}^{k} J_{j_i} \cup J_{n+1} \cup W \right)^c} \right]^2 \right] \Bigg) \, .
\end{multline*}
Note that:
\begin{multline*}
\E_{1/2} \left[  \E_{1/2} \left[ h \cond \mathcal{F}_{\left( \cup_{i=1}^{k} J_{j_i} \cup W \right)^c} \right]^2 \right] - \E_{1/2} \left[ \E_{1/2} \left[ h \cond \mathcal{F}_{\left( \cup_{i=1}^{k} J_{j_i} \cup J_{n+1} \cup W \right)^c} \right]^2 \right]\\ = \E_{1/2} \left[ \left( \E_{1/2} \left[ h \cond \mathcal{F}_{\left( \cup_{i=1}^{k} J_{j_i} \cup W \right)^c} \right] - \E_{1/2} \left[ h \cond \mathcal{F}_{\left( \cup_{i=1}^{k} J_{j_i} \cup J_{n+1} \cup W \right)^c} \right] \right)^2 \right] \, .
\end{multline*}
Moreover, since $\Pro_{1/2}$ is the uniform measure, we have:
\[
\E_{1/2} \left[ h \cond \mathcal{F}_{\left( \cup_{i = 1}^k J_{j_i} \cup J_{n+1} \cup W \right)^c} \right] = \E_{1/2} \left[ \E_{1/2} \left[ h \cond \mathcal{F}_{J_{n+1}^c} \right] \cond \mathcal{F}_{\left( \cup_{i = 1}^k J_{j_i} \cup W \right)^c} \right] \, .
\]
Therefore:
\begin{multline*}
\sum_{k=0}^{n+1} (-1)^k \sum_{1 \leq j_1 < \ldots < j_k \leq n+1} \E_{1/2} \left[ \E_{1/2} \left[ h \cond \mathcal{F}_{\left( \cup_{i=1}^{k} J_{j_i} \cup W \right)^c} \right]^2 \right]\\
= \sum_{k=0}^{n} (-1)^k \sum_{1 \leq j_1 < \ldots < j_k \leq n} \E_{1/2} \left[ \E_{1/2} \left[ h - \E_{1/2} \left[ h \cond \mathcal{F}_{J_{n+1}^c} \right] \cond \mathcal{F}_{\left( \cup_{i=1}^{k} J_{j_i} \cup W \right)^c} \right]^2 \right] \, .
\end{multline*}
By using~\eqref{grosseexpression} and the induction hypothesis for $h - \E_{1/2} \left[ h \, | \, J_{n+1}^c \right]$, we obtain that the above equals:
\begin{align*}
& \widehat{\Q}_{h - \E_{1/2} [ h | \mathcal{F}_{J_{n+1}^c} ]} \left[ \forall j \in \lbrace 1, \cdots, n \rbrace, S \cap J_j \neq \emptyset \, , \, S \cap W = \emptyset \right]\\
& \leq 4^ n \left\lVert h - \E_{1/2} \left[ h \cond \mathcal{F}_{J_{n+1}^c} \right] \right\rVert_\infty^2 \E_{1/2} \left[ \Pro_{1/2} \left[ JP_{J_1, \cdots, J_n} \left( h - \E_{1/2} \left[ h \cond \mathcal{F}_{J_{n+1}^c} \right] \right) \cond \mathcal{F}_{W^c} \right]^2 \right]\\
& \leq 4^{n+1} \parallel h \parallel_\infty^2 \E_{1/2} \left[ \Pro_{1/2} \left[ JP_{J_1, \cdots, J_n} \left( h - \E_{1/2} \left[ h \cond \mathcal{F}_{J_{n+1}^c} \right] \right) \cond \mathcal{F}_{W^c} \right]^2 \right] \, .
\end{align*}
The proof is over since:
\[
JP_{J_1, \cdots, J_n} \left(h - \E_{1/2} \left[ h \cond \mathcal{F}_{J_{n+1}^c} \right] \right) \subseteq JP_{J_1, \cdots, J_{n+1}}(h) \, .
\]

\section{Appendix: the $4$-arm event conditioned on the configuration in a half-plane}\label{a.eps}

Consider bond percolation on $\Z^2$ or site percolation on $\T$ and see Lemma~\ref{JP} for the notation $\mathcal{F}_B$.

\begin{lem}\label{lemmerigolo}
Let $r_2 \geq r_1 \geq 1$, let $H$ be the lower half plane and let $\half = H \cap \mathcal{I}$. There exists an absolute constant $C < +\infty$ such that:
\[
\E_{1/2} \left[ \Pro_{1/2} \left[ \text{\textup{\bf{A}}}_4(r_1,r_2) \cond \mathcal{F}_{\half} \right]^2 \right] \leq C \, \alpha_4(r_1,r_2) \, \left( \frac{r_2}{r_1} \right)^{-1/C} \, ,
\]
where $\text{\textup{\bf{A}}}_4(r_1,r_2)$ is the $4$-arm event in  the annulus $[-r_2,r_2] \setminus (-r_1,r_1)^2$. (Such an estimate is also true with the right, left or upper half-plane and the proof is the same.)
\end{lem}
\begin{proof}
As pointed out in the beginning of Subsection~$5.3$ in~\cite{GPS} for analogous events, it is not difficult to see that, if $\omega$ and $\omega'$ are two critical percolation configurations which coincide on $\half$ but are independent on $\half^c$, then:
\begin{equation}\label{stuce}
\E_{1/2} \left[ \Pro_{1/2} \left[ \text{\textbf{A}}_4(r_1,r_2) \cond \mathcal{F}_{\half} \right]^2 \right] = \Pro \left[ \omega, \omega' \in \text{\textbf{A}}_4(r_1,r_2) \right] \, .
\end{equation}

Let $M \geq 100$ that we will choose later. First note that it is sufficient to prove the lemma for $r_1 < r_2$ of the form $\rho_l = M^l$ for some $l \in \N_+$. Let $1 \leq i < j$ be such that $r_1=\rho_i=M^i$ and $r_2=\rho_j=M^j$. Let $B(k)=B(k,M)$ be the event that there exist open paths in the annulus $[-\rho_{k+1},\rho_{k+1}]^2 \setminus (-\rho_k,\rho_k)^2$ as in Figure~\ref{pleindanneaux}. By the FKG-inequality and the RSW-estimate, there exists $c=c(M)>0$ such that for all $k$ we have $\Pro_{1/2} \left[ B(k) \right] \geq c$. Given a realization of our variables $\omega$ and $\omega'$, we write $i \leq k_1 < \cdots < k_N \leq j-1$ the $k$'s such that $B(k)$ is satisfied in $\omega$ (note that the random variables $N$ and $(k_1, \cdots, k_N)$ are measurable with respect to $\omega$ and are independent of $\omega'$). Note also that (by classical properties of the Binomial distribution and thanks to the existence of the above $c > 0$) there exists $a=a(M) \in (0,1)$ such that the probability of the event $\lbrace N \leq  a \log_M(r_2/r_1) \rbrace$ is less than or equal to $\frac{1}{a}(1-a)^{\log_M(r_2/r_1)}$.
\medskip

Next, condition on $B(k)$ and on the upper open paths that cross the rectangle $[\rho_k+2\rho_k/10,\rho_{k+1}-2\rho_k/10] \times [\rho_k/10,2\rho_k/10]$ and the rectangle $[-(\rho_k+2\rho_k/10),-(\rho_{k+1}-2\rho_k/10)] \times [\rho_k/10,2\rho_k/10]$. Write $\gamma_1$ and $\gamma_2$ for these two paths. Note that, if $\text{\textbf{A}}_4(\rho_k+10,\rho_{k+1})$ holds, then there is a $3$-arm event in the region of the annulus $[-(\rho_{k+1}-3\rho_k/10),\rho_{k+1}-3\rho_k/10]^2 \setminus (-(\rho_{k}+3\rho_k/10),\rho_k+3\rho_k/10)^2$ that is below $\gamma_1$ and $\gamma_2$, see Figure~\ref{pleindanneauxetchemins}. The percolation configuration in this region is not biased by the conditionnings.
Consequently, we can use~\eqref{alpha3+} to obtain that there exist two absolute constants $C_0,C_1<+\infty$ such that, for all $k \in \lbrace 1, \cdots, j-1 \rbrace$, we have:
\begin{equation}\label{3bras}
\Pro_{1/2} \left[ \text{\textbf{A}}_4(\rho_k + 10,\rho_{k+1}) \cond B(k) \right] \leq C_0 \left( \frac{\rho_k+3\rho_k/10}{\rho_{k+1}-3\rho_k/10} \right)^2 \leq C_1 \, M^{-2} \, .
\end{equation}

\begin{figure}[h!]
\centering
\includegraphics[scale=0.42]{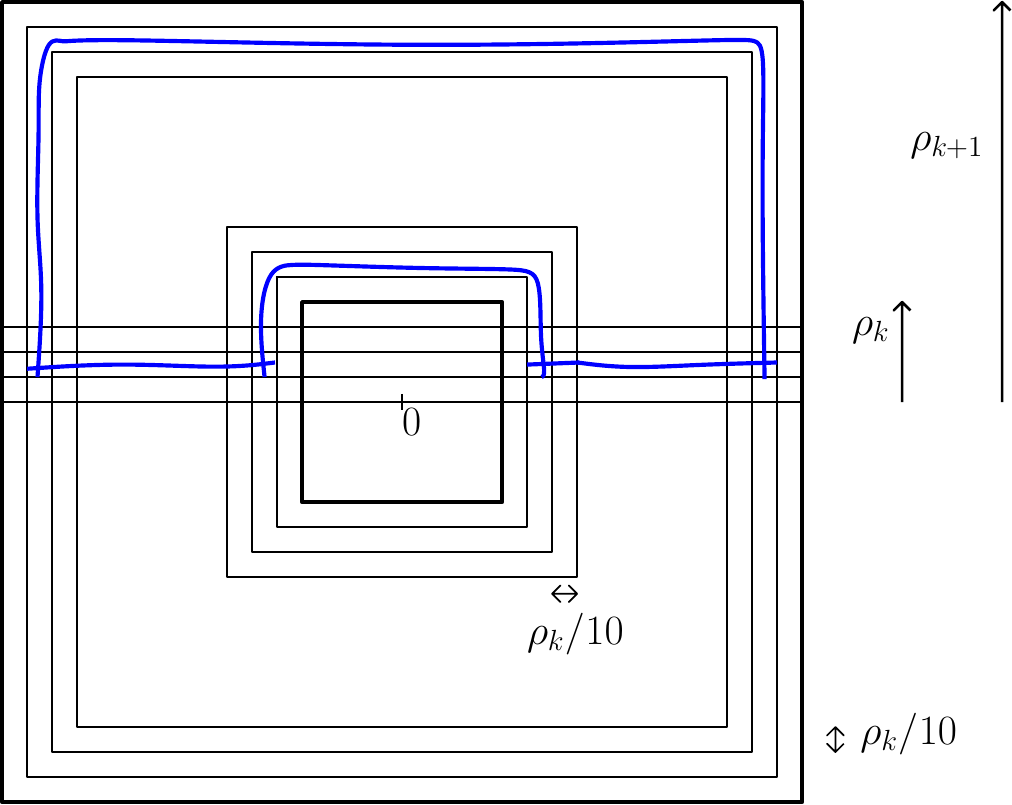}
\caption{A realization of the event $B(k)$.\label{pleindanneaux}}
\end{figure}

\begin{figure}[h!]
\centering
\includegraphics[scale=0.42]{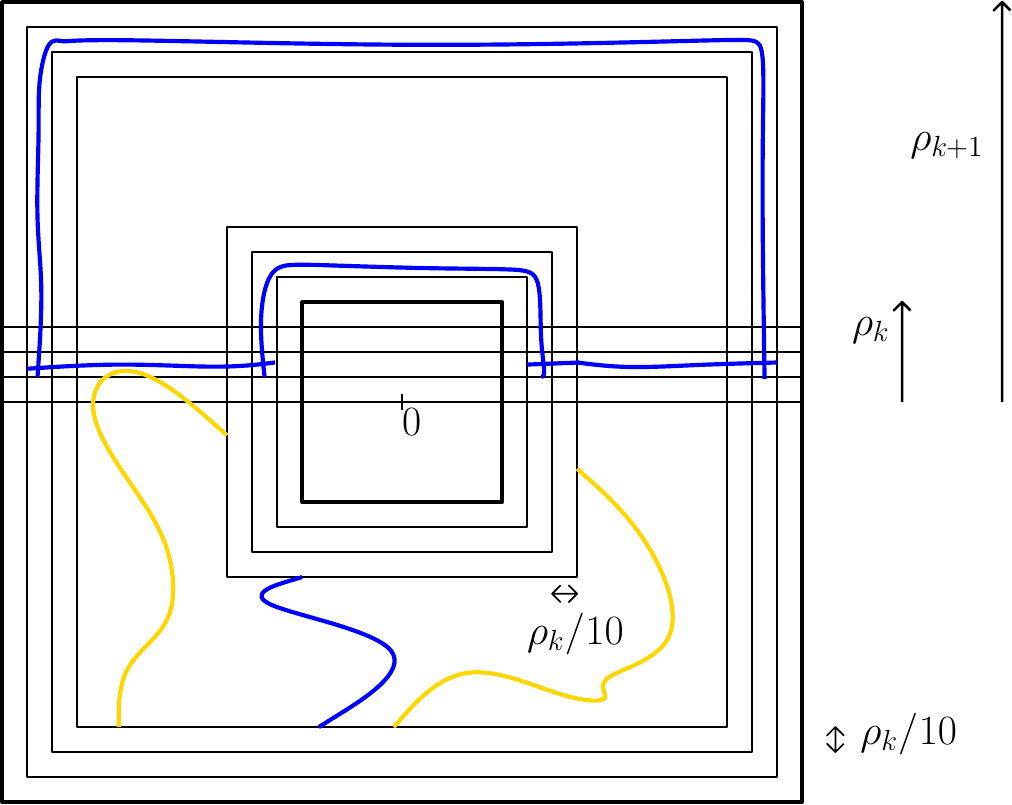}
\caption{A realization of the event $B(k)$ and of the $4$-arm event implies the realization of a $3$-arm event in a half-plane.\label{pleindanneauxetchemins}}
\end{figure}

Now, we simply use that $\lbrace \omega,\omega' \in \text{\textbf{A}}_4(r_1,r_2) \rbrace \subseteq \lbrace \omega \in \text{\textbf{A}}_4(r_1,r_2) \rbrace \cup \lbrace \omega' \in \text{\textbf{A}}_4(r_1,r_2) \rbrace$, and we choose to look only at $\omega$ in the annuli where $\omega \in B(k)$ and to look only at $\omega'$ in the other annuli. More precisely, for any $m \in \N$ and $i \leq l_1 < \cdots < l_m \leq j-1$, by spatial independence we have:
\begin{align*}
& \Pro \left[ \omega, \, \omega' \in \text{\textbf{A}}_4(r_1,r_2) \cond N=m, k_1 = l_1, \cdots, k_m = l_m \right] \leq \Pro \left[ \omega' \in \text{\textbf{A}}_4(r_1,\rho_{l_1}) \right]\\
& \times \prod_{q = 1}^{m-1} \left( \Pro \left[ \omega \in \text{\textbf{A}}_4(\rho_{l_q}+10,\rho_{l_q+1}) \cond \omega \in B(l_q) \right] \, \Pro \left[ \omega' \in \text{\textbf{A}}_4(\rho_{l_q+1}+10,\rho_{l_{q+1}}) \right] \right)\\
&\times \Pro \left[ \omega \in \text{\textbf{A}}_4(\rho_{l_m}+10,\rho_{l_m+1}) \cond \omega \in B(l_m) \right] \Pro \left[ \omega' \in \text{\textbf{A}}_4(\rho_{l_m+1}+10,r_2) \right] \, .
\end{align*}
Next,~\eqref{3bras} implies that the above is at most:
\begin{align*}
& \alpha_4(r_1,\rho_{l_1}) \, \prod_{q = 1}^{m-1} \left( \frac{C_1}{M^2} \, \alpha_4(\rho_{l_q+1}+10,\rho_{l_{q+1}}) \right) \, \frac{C_1}{M^2} \, \alpha_4(\rho_{l_m+1}+10,r_2)\\
& = \alpha_4(r_1,\rho_{l_1}) \, \prod_{q = 1}^{m-1} \left( \alpha_4(\rho_{l_q},\rho_{l_q+1}+10) \, \alpha_4(\rho_{l_q+1}+10,\rho_{l_{q+1}}) \, \frac{C_1}{M^2 \alpha_4(\rho_{l_q},\rho_{l_q+1}+10)} \right)\\
& \times \alpha_4(\rho_{l_m},\rho_{l_m+1}+10) \, \alpha_4(\rho_{l_m+1}+10,r_2)\frac{C_1}{M^2 \alpha_4(\rho_{l_m},\rho_{l_m+1}+10)} \, .
\end{align*}
The quasi-multiplicativity property implies that there exists a constant $C_2 < +\infty$ such that the above is at most:
\begin{equation}\label{lesmauvaisesechelles}
C_2 \, \alpha_4(r_1,r_2) \, \prod_{q=1}^m \frac{C_2}{M^2 \alpha_4(\rho_{l_q},\rho_{l_q+1})} \, .
\end{equation}
Thanks to the left-hand inequality of~\eqref{alpha4} we know that there exists some $\epsilon_2 > 0$ such that for all $l \in \N$:
\[
\alpha_4(\rho_{l},\rho_{l+1}) \geq \frac{1}{\epsilon_2} M^{-2+\epsilon_2} \, .
\]
We deduce that we can choose $100 \leq M < +\infty$ such that for all $l \in \N$:
\begin{equation}\label{truc}
\frac{C_2}{M^2\alpha_4(\rho_{l},\rho_{l+1})} \leq 1/2 \, .
\end{equation}
We fix such an $M$. Then,~\eqref{lesmauvaisesechelles} and~\eqref{truc} imply that:
\begin{equation}\label{tructruc}
\Pro \left[ \omega,\omega' \in \text{\textbf{A}}_4(r_1,r_2) \cond N=m, k_1 = l_1, \cdots, k_m = l_m \right] \leq C_2 \, \alpha_4(r_1,r_2) /2^m \, .
\end{equation}

\bigskip

Now, we write:
\begin{multline}
\Pro \left[ \omega, \omega' \in \text{\textbf{A}}_4(r_1,r_2) \right] \leq \Pro \left[ N \leq  a \log_M(r_2/r_1), \, \omega' \in \text{\textbf{A}}_4(r_1,r_2) \right]\\
+ \Pro \left[ N \geq a \log_M(r_2/r_1), \, \omega,\omega' \in \text{\textbf{A}}_4(r_1,r_2) \right] \, , \label{the2terms}
\end{multline}
where the constant $a=a(M)$ comes from the beginning of the proof.

By independence of $\omega$ and $\omega'$ on $\half^c$ we can say that the first term of the right-hand side of~\eqref{the2terms} equals:
\begin{eqnarray*}
\Pro \left[ N \leq  a \log_M(r_2/r_1) \right] \Pro \left[ \omega' \in \text{\textbf{A}}_4(r_1,r_2) \right] & \leq  &\frac{1}{a}(1-a)^{\log_M(r_2/r_1)} \, \alpha_4(r_1,r_2)\\
& \leq & \frac{C}{2} \, \left( \frac{r_2}{r_1} \right)^{-1/C} \, \alpha_4(r_1,r_2) \, ,
\end{eqnarray*}
for some $C < +\infty$.

Thanks to~\eqref{tructruc} we know that the second term of the right-hand side of~\eqref{the2terms} is also less than or equal to $\frac{C}{2} \, \left( \frac{r_2}{r_1} \right)^{-1/C} \alpha_4(r_1,r_2)$ if $C$ is sufficiently large. And the proof is over thanks to~\eqref{stuce}.
\end{proof}

\begin{lem}\label{exponenthalfplane}
For site percolation on $\T$, there exists $\zeta_4^{|\half} \in (5/4,5/2]$ such that:
\[
\beta_4^{|\H}(r_1,r_2):= \E_{1/2} \left[ \Pro_{1/2} \left[ \text{\textup{\bf{A}}}_4(r_1,r_2) \cond \mathcal{F}_{\half} \right]^2 \right] = \left( \frac{r_1}{r_2} \right)^{\zeta_4^{|\half}+\petito{1}} \, ,
\]
where $r_2 \geq r_1 \geq 1$ and $\petito{1} \rightarrow 0$ as $r_1/r_2 \rightarrow 0$.
\end{lem}
\begin{proof}
We shall only sketch the proof here. 
To prove that this exponent exists on $\T$ (without necessarily computing its value), one proceeds as with classical exponents which describe critical percolation in two steps (see \cite{lawler2002onearm, smirnov2001critical, werner2007lectures}):
\bnum
\item First, one needs to show that for any fixed $r<R$, the quantity $\beta_4^{|\H}(\lambda r, \lambda R)$ converges as $\lambda\to +\infty$ to a limiting real number which is expressible in terms of the continuum scaling limit of percolation. For usual {\em arm-exponents}, these limiting numbers are given by $\SLE_6$ computable quantities. In the present case, these limiting real numbers are instead described in terms of the continuum scaling limit of percolation introduced by Schramm-Smirnov \cite{SS11}. The proof follows very similar lines as the proof of Theorem~$10.1$ in \cite{GPS}. Let us be a little more precise here: In order to prove Theorem~$10.1$ in~\cite{GPS}, two results are used: \textit{(a)} the existence and uniqueness of the continuum scaling limit of percolation (see Subsection~$2.3$ of~\cite{GPS2a} for the uniqueness part) and \textit{(b)} a ``mesh independent gluing property for crossing of quads" which is Proposition~$10.3$ of~\cite{GPS} and Proposition~$4.1$ of~\cite{SS11}. The only difference in our case is that we need a gluing property for the $4$-arm event instead of the crossing of quads. Such a result follows easily from the gluing properties for crossing of quads, and from results about the scaling limit of arm events from Subsection~$2.4$ of~\cite{GPS2a} (see in particular~$(2.3)$ of this last paper). 

\item Then one needs to prove that the quantity $\beta_4^{|\H}(r_1,r_2)$ statisfies a \textbf{quasi-multiplicativity property} (see \cite{werner2007lectures}). This is Proposition~$5.1$ of~\cite{GPS} (with $W = \half^c$).
\enum
Once $\zeta_4^{|\half}$ is proved to exist, the fact that it belongs to $(5/4,5/2]$ follows directly from Lemma~\ref{lemmerigolo} and the computation of the critical exponents.
\end{proof}

%
%

\addcontentsline{toc}{section}{Bibliography}

\bibliographystyle{alpha}

\begin{thebibliography}{}

\end{thebibliography}


\begin{thebibliography}{GPS13b}

\bibitem[BGS13]{BGS}
Erik~I Broman, Christophe Garban, and Jeffrey~E. Steif.
\newblock Exclusion sensitivity of Boolean functions.
\newblock {\em Probability theory and related fields}, 155(3-4):621--663, 2013.

\bibitem[BKS99]{BKS}
Itai Benjamini, Gil Kalai, and Oded Schramm.
\newblock Noise sensitivity of {B}oolean functions and applications to
  percolation.
\newblock {\em Inst. Hautes \'Etudes Sci. Publ. Math.}, (90):5--43 (2001),
  1999.

\bibitem[BR06]{bollobas2006percolation}
B{\'e}la Bollob{\'a}s and Oliver Riordan.
\newblock {\em Percolation}.
\newblock Cambridge University Press, New York, 2006.

\bibitem[FvdH15]{fitzner2015nearest}
Robert Fitzner and Remco van~der Hofstad.
\newblock Mean-field behavior for nearest-neighbour percolation in $d> 10$.
\newblock {\em Preprint}, 2015.

\bibitem[For15a]{For15a}
Malin Pal{\"o} Forsstr{\"o}m.
\newblock A Noise Sensitivity Theorem for Schreier Graphs.
\newblock {\em arXiv preprint arXiv:1501.01828}, 2015.

\bibitem[For15b]{For15b}
Malin Pal{\"o} Forsstr{\"o}m.
\newblock Monotonicity properties of exclusion sensitivity.
\newblock {\em arXiv preprint arXiv:1503.05735}, 2015.

\bibitem[GPS10]{GPS}
Christophe Garban, G{\'a}bor Pete, and Oded Schramm.
\newblock The Fourier spectrum of critical percolation.
\newblock {\em Acta Mathematica}, 205(1):19--104, 2010.

\bibitem[GPS13a]{GPS2a}
Christophe Garban, G{\'a}bor Pete, and Oded Schramm.
\newblock Pivotal, cluster, and interface measures for critical planar
  percolation.
\newblock {\em Journal of the American Mathematical Society}, 26(4):939--1024,
  2013.

\bibitem[GPS13b]{GPS2b}
Christophe Garban, G{\'a}bor Pete, and Oded Schramm.
\newblock The scaling limits of near-critical and dynamical percolation.
\newblock To appear in {\em Journal of the European Mathematical Society.} 
\newblock {\em arXiv:1305.5526}, 2013.

\bibitem[Gri99]{grimmett1999percolation}
Geoffrey~R. Grimmett.
\newblock {\em Percolation (Grundlehren der mathematischen Wissenschaften)}.
\newblock Springer: Berlin, Germany, 1999.

\bibitem[GS14]{book}
Christophe Garban and Jeffrey Steif.
\newblock {\em Noise sensitivity of Boolean functions and percolation}.
\newblock Cambridge University Press, 2014.

\bibitem[Har60]{harris1960lower}
Theodore~E. Harris.
\newblock A lower bound for the critical probability in a certain percolation
  process.
\newblock {\em Proc. Cambridge Philos. Soc.} 56(01):13--20, 1960.

\bibitem[Har78]{Harris}
Theodore~E. Harris.
\newblock Additive set-valued Markov processes and graphical methods.
\newblock {\em Ann. Probability.} 6, no. 3, 355--378, 1978. 

\bibitem[HPS97]{olle1997dynamical}
Olle H{\"a}ggstr{\"o}m, Yuval Peres, and Jeffrey~E. Steif.
\newblock Dynamical percolation.
\newblock {\em Ann. Inst. H. Poincar\'e Probab. Statist.}, 33(4):497--528, 1997.

\bibitem[HPS15]{HPS15}
Alan Hammond, G{\'a}bor Pete, and Oded Schramm.
\newblock Local time on the exceptional set of dynamical percolation and the incipient infinite cluster.
\newblock {\em Ann. Probab.}, 43 (6):2949--3005, 2015.

\bibitem[HS94]{hara1994mean}
Takashi Hara and Gordon Slade.
\newblock Mean-field behaviour and the lace expansion.
\newblock In {\em Probability and phase transition ({C}ambridge, 1993)}, volume
  420 of {\em NATO Adv. Sci. Inst. Ser. C Math. Phys. Sci.}, 87--122.
  Kluwer Acad. Publ., Dordrecht, 1994.

\bibitem[Kes80]{kesten1980critical}
Harry Kesten.
\newblock The critical probability of bond percolation on the square lattice
  equals $1/2$.
\newblock {\em Comm. Math. Phys. 74, no. 1}, 1980.

\bibitem[Kes87]{kesten1987scaling}
Harry Kesten.
\newblock Scaling relations for 2d-percolation.
\newblock {\em Comm. Math. Phys.}, 109(1):109--156, 1987.

\bibitem[KZ87]{kesten1987strict}
Harry Kesten and Yu~Zhang.
\newblock Strict inequalities for some critical exponents in two-dimensional
  percolation.
\newblock {\em Journal of statistical physics}, 46(5-6):1031--1055, 1987.

\bibitem[Lig05]{liggett2005interacting}
Thomas~M. Liggett.
\newblock {\em Interacting particle systems}.
\newblock Classics in Mathematics. Springer-Verlag, Berlin, 2005.
\newblock Reprint of the 1985 original.

\bibitem[LSW02]{lawler2002onearm}
Gregory~F. Lawler, Oded Schramm, and Wendelin Werner.
\newblock One-arm exponent for critical 2{D} percolation.
\newblock {\em Electron. J. Probab.}, 7:no. 2, 13 pp. (electronic), 2002.

\bibitem[Nol08]{nolin2008near}
Pierre Nolin.
\newblock Near-critical percolation in two dimensions.
\newblock {\em Electron. J. Probab}, 13(55):1562--1623, 2008.

\bibitem[PS98]{peres1998number}
Yuval Peres and Jeffrey E. Steif.
\newblock The number of infinite clusters in dynamical percolation.
\newblock {\em Probab. Theory Related Fields}, 111, pp 141--165, 1998

\bibitem[Smi01]{smirnov2001criticalp}
Stanislav Smirnov.
\newblock Critical percolation in the plane: Conformal invariance, Cardy's
  formula, scaling limits.
\newblock {\em Comptes Rendus de l'Acad{\'e}mie des Sciences-Series
  I-Mathematics}, 333(3):239--244, 2001.

\bibitem[SS10]{SS10}
Oded Schramm and Jeffrey~E. Steif.
\newblock Quantitative noise sensitivity and exceptional times for percolation.
\newblock {\em Ann. of Math. (2)}, 171(2):619--672, 2010.


\bibitem[SS11]{SS11}
Oded Schramm and Stanislav Smirnov.
\newblock On the scaling limits of planar percolation.
\newblock With an appendix by Christophe Garban.
{\em Ann. Probab.} 39, no. 5, 1768--1814, 2011.

\bibitem[SW01]{smirnov2001critical}
Stanislav Smirnov and Wendelin Werner.
\newblock Critical exponents for two-dimensional percolation.
\newblock {\em Math. Res. Lett.}, 8(5-6):729--744, 2001.

\bibitem[Wer07]{werner2007lectures}
Wendelin Werner.
\newblock Lectures on two-dimensional critical percolation.
\newblock {\em IAS Park City Graduate Summer School}, 2007.

\end{thebibliography}

\ \\
{\bf Christophe Garban}\\
Universit\'e Lyon 1\\
\url{http://math.univ-lyon1.fr/~garban/}\\
garban@math.univ-lyon1.fr\\
Supported by the ANR grant Liouville 15-CE40-0013 and the ERC grant LiKo 676999 \\

\medskip
\ni
{\bf Hugo Vanneuville} \\
Universit\'e Lyon 1\\
\url{http://math.univ-lyon1.fr/~vanneuville/}\\
vanneuville@math.univ-lyon1.fr\\
Supported by the ANR grant Liouville 15-CE40-0013 and the ERC grant LiKo 676999

\end{document}